\documentclass[11pt, a4paper]{article} 

\usepackage{graphicx}%
\usepackage{multirow}%
\usepackage{amsmath,amssymb,amsfonts}%
\usepackage{amsthm}%
\usepackage{mathrsfs}%
\usepackage[title]{appendix}%
\usepackage{xcolor}%
\usepackage{textcomp}%
\usepackage{manyfoot}%
\usepackage{booktabs}
\usepackage{algorithm}%
\usepackage{algorithmicx}%
\usepackage{algpseudocode}%
\usepackage{listings}%
\usepackage{mathtools,float}
\usepackage{enumerate,enumitem}
\usepackage{hyperref}
\usepackage{pdfsync}

\newcommand{\mask}[1]{{}}

\newcommand{\setu}{{\mathrm{\mathfrak{u}}}}
\newcommand{\setv}{{\mathrm{\mathfrak{v}}}}

\newcommand{\supp}{\mathrm{supp}}
\newcommand{\diam}{\mathrm{diam}}

\newcommand{\tri}{\bigtriangleup}

\newcommand{\bigO}{\mathcal{O}}
\newcommand{\eps}{\varepsilon}

\newcommand{\satop}[2]{\stackrel{\scriptstyle{#1}}{\scriptstyle{#2}}}

\newcommand{\N}[0]{\mathbb{N}}

\newcommand{\R}[0]{\mathbb{R}}

\newcommand{\bszero}{{\boldsymbol{0}}}

\newcommand{\bsa}{{\boldsymbol{a}}}
\newcommand{\bsb}{{\boldsymbol{b}}}

\newcommand{\bse}{{\boldsymbol{e}}}

\newcommand{\bsg}{{\boldsymbol{g}}}

\newcommand{\bsk}{{\boldsymbol{k}}}
\newcommand{\bsl}{{\boldsymbol{\ell}}}
\newcommand{\bsm}{{\boldsymbol{m}}}

\newcommand{\bst}{{\boldsymbol{t}}}

\newcommand{\bsv}{{\boldsymbol{v}}}

\newcommand{\bsx}{{\boldsymbol{x}}}
\newcommand{\bsy}{{\boldsymbol{y}}}
\newcommand{\bsz}{{\boldsymbol{z}}}

\newcommand{\bsA}{{\boldsymbol{A}}}

\newcommand{\bsU}{{\boldsymbol{U}}}

\newcommand{\bsgamma}{{\boldsymbol{\gamma}}}

\newcommand{\bskappa}{{\boldsymbol{\kappa}}}

\newcommand{\bsnu}{{\boldsymbol{\nu}}}

\newcommand{\calA}{\mathcal{A}}

\newcommand{\calD}{\mathcal{D}}
\newcommand{\calF}{\mathcal{F}}
\newcommand{\calG}{\mathcal{G}}
\newcommand{\calH}{\mathcal{H}}

\newcommand{\calK}{\mathcal{K}}
\newcommand{\calL}{\mathcal{L}}

\newcommand{\calO}{\mathcal{O}}

\newcommand{\calT}{\mathcal{T}}

\newcommand{\rd}{\mathrm{d}}

\newcommand{\ML}{\mathrm{ML}}

\newcommand{\bbA}{\mathbb{A}}
\newcommand{\bbB}{\mathbb{B}}
\newcommand{\bbI}{\mathbb{I}}

\newcommand{\sfP}{\mathsf{P}}
\newcommand{\sfI}{\mathsf{I}}

\definecolor{darkred}{RGB}{139,0,0}
\definecolor{darkgreen}{RGB}{0,100,0}
\definecolor{darkmagenta}{RGB}{170,0,120}
\definecolor{darkpurple}{RGB}{110,0,180}
\definecolor{darkblue}{RGB}{40,0,200}
\definecolor{darkbrown}{rgb}{0.75,0.40,0.15}

\newcommand{\oneabstract}{{This paper introduces a multilevel kernel-based approximation method to estimate efficiently solutions to elliptic partial differential equations (PDEs) with periodic random coefficients. Building upon the work of Kaarnioja, Kazashi, Kuo, Nobile, Sloan (Numer.\ Math., 2022) on kernel interpolation with quasi-Monte Carlo (QMC) lattice point sets, we leverage multilevel techniques to enhance computational efficiency while maintaining a given level of accuracy. 
In the function space setting with product-type weight parameters, the single-level approximation can achieve an accuracy of $\varepsilon>0$ with cost $\calO(\varepsilon^{-\eta-\nu-\theta})$ for positive constants $\eta, \nu, \theta $ depending on the rates of convergence associated with dimension truncation, kernel approximation, and finite element approximation, respectively. Our multilevel approximation can achieve the same $\varepsilon$ accuracy at a reduced cost $\calO(\varepsilon^{-\eta-\max(\nu,\theta)})$.
Full regularity theory and error analysis are provided, followed by numerical experiments that validate 
the efficacy of the proposed multilevel approximation in comparison to the single-level approach.}                            }

\newtheorem{theorem}{Theorem}
\newtheorem{lemma}[theorem]{Lemma}

\newtheorem{remark}[theorem]{Remark}

\newcommand{\ourparagraph}[1]{\paragraph*{#1}}

\newenvironment{proofof}[1]{\begin{trivlist}
    \item[\hskip\labelsep{\bf Proof of {#1}.}]}{$\hfill\Box$\\\end{trivlist}}

\numberwithin{equation}{section}                           

\makeatletter
\let\@fnsymbol\@arabic
\makeatother

\title{Multilevel lattice-based kernel approximation for elliptic PDEs with random coefficients}
\date{\today}
\author{Alexander D. Gilbert\footnotemark[1] \and
	    Michael B. Giles\footnotemark[2] \and
        Frances Y. Kuo\footnotemark[1] \and
        Ian H. Sloan\footnotemark[1] \and
	    Abirami Srikumar\footnotemark[1]
	     }

\begin{document}
\maketitle

\footnotetext[1]{School of Mathematics and Statistics, UNSW Sydney, Sydney NSW 2052, Australia.\\
                           \texttt{alexander.gilbert@unsw.edu.au},\;
                           \texttt{f.kuo@unsw.edu.au},\\
                           \texttt{i.sloan@unsw.edu.au},\;
                           \texttt{a.srikumar@unswalumni.com}
                           }

\footnotetext[2]{Mathematical Institute, University of Oxford, Oxford, OX2 6GG, UK.\\
                           \texttt{mike.giles@maths.ox.ac.uk}
                           }     
\begin{abstract}
\oneabstract
\end{abstract}

\section{Introduction}\label{sec:intro}

Kernel-based interpolation using a quasi-Monte Carlo (QMC) lattice design was first 
introduced in \cite{ZLH06}, where the authors analysed splines constructed 
from reproducing kernel functions on a lattice point set for periodic function approximation. 
They found that the special structure of a lattice point set coupled with periodic kernel functions 
led to linear systems with a circulant matrix that were able to be solved efficiently via the fast Fourier transform (FFT). The recent paper \cite{KKKNS22} applies kernel interpolation for approximating the solutions to partial differential equations (PDEs) with periodic random coefficients over the parametric domain. In this paper, we seek to enhance the kernel-based approximation method for estimating PDEs over the parametric domain by leveraging multilevel methods~\cite{Giles15}.

We are interested in the following parametric elliptic PDE
\begin{align}
\label{eq:pde}
-\nabla\cdot(\Psi(\bsx, \bsy) \nabla u(\bsx, \bsy)) \,&=\, f(\bsx),\, & \bsx \in D,
\\\nonumber
u(\bsx, \bsy) \,&=\, 0, & \bsx \in \partial D,
\end{align}
where the \emph{physical variable} $\bsx$ belongs to a bounded convex domain 
$D \subset \R^d$, for $d = 1, 2, $ or~$3$, and 
\[
 \bsy \in \Omega \coloneqq [0, 1]^\N
\] %
is a countable vector of parameters. 
It is assumed that the input field $\Psi(\cdot, \bsy)$ in~\eqref{eq:pde} is represented by the 
periodic model introduced recently in~\cite{KKS20}, which is
periodic in $\bsy$ and given by the series expansion
\begin{equation}
\label{eq:coeff}
\Psi(\bsx, \bsy) \,\coloneqq\, \psi_0(\bsx) + \sum_{j = 1}^\infty \sin(2\pi y_j)\, \psi_j(\bsx).
\end{equation}
Here $\bsy \coloneqq (y_1, y_2, \ldots)$, where each $y_j$ is
independent and uniformly distributed on $[0, 1]$. The functions $\psi_j \in L^\infty(D)$ 
are known and deterministic such that $\Psi(\cdot, \bsy) \in  L^\infty(D)$ for all 
$\bsy \in \Omega$. Additional requirements on $\psi_j$ will be introduced later as necessary.

As explained in \cite{KKS22,KKKNS22}, the periodic random field model is equivalent to the affine random field model under the Chebyshev density, instead of the uniform density. The affine model with both densities are commonly considered by researchers working with sparse polynomial approximations, see e.g., \cite{ABS22}. We therefore argue that the choice between the periodic random field model and the affine uniform model is a matter of taste and convenience rather than conviction.

The goal is to approximate efficiently the solution $u(\bsx, \bsy)$ in both $\bsx \in D$ and $\bsy \in \Omega$ simultaneously. We will follow a similar method to~\cite{KKKNS22}, where the approximation is based on discretising the spatial domain $D$ using \emph{finite elements} and applying \emph{kernel approximation} over the parametric domain $\Omega$ based on a lattice point set. Building on~\cite{KKS20,KKKNS22}, this paper introduces a new \emph{multilevel} approximation that spreads the work over a hierarchy of finite element meshes and kernel interpolants so that the overall cost is reduced.

The solution to \eqref{eq:pde} with the coefficient given by \eqref{eq:coeff} lies in a periodic function space (with respect to the parametric domain) equipped with a reproducing kernel $\calK(\cdot,\cdot)$. Given evaluations $u(\cdot, \bst_k)$ on a lattice point set $\{\bst_k\}_{k=0}^{N-1} \subset \Omega$, the solution $u = u(\cdot,\bsy)$ can be approximated over~$\Omega$ using a kernel interpolant given by 
\begin{align} \label{eq:kernel_interp} 
		I_N u(\cdot,\bsy) \coloneqq \sum_{k=0}^{N-1} a_k\, \calK(\bst_k,\bsy) \qquad \mbox{for } \bsy\in\Omega,
	\end{align}
where $I_N$ is the kernel interpolation operator on $\Omega$. 
The coefficients $a_k=a_k (\bsx)$ for $k=0,\ldots,N-1$ are obtained by solving the linear system associated with interpolation at the lattice points via FFT. 
Combining \eqref{eq:kernel_interp} with a finite element discretisation
leads to an approximation of $u$ over $D \times \Omega$.
This is the method employed by \cite{KKKNS22}, which we will refer to as
the \emph{single-level kernel interpolant} (see also Section~\ref{sec:SLKI} for details).

To introduce the \emph{multilevel kernel approximation}, consider now a sequence of kernel interpolants $I_{\ell} \coloneqq I_{N_\ell}$, for $\ell = 0, 1, \ldots$, based on a sequence of embedded lattice point sets with nonincreasing size $N_0 \ge N_1 \ge \cdots$ and a sequence of nested finite element approximations $u_\ell(\cdot, \bsy)$ for $\ell = 0, 1, \ldots $, where $u_\ell$ becomes increasingly accurate (and thus has increasing cost) as $\ell$ increases.
Omitting the $\bsx$ and $\bsy$ dependence, the multilevel kernel approximation with maximum level $L \in \N$ is given by
\begin{equation}
\label{eq:ml-ker}
I^{\rm{ML}}_Lu \,\coloneqq\, I_0 u_0+ 
\sum_{\ell = 1}^L I_\ell (u_\ell - u_{\ell - 1}).
\end{equation}
The motivation of such an algorithm is to reduce the overall computational cost compared to the single-level method while achieving the same level of accuracy. The cost savings come from interpolating the difference $u_\ell-u_{\ell-1}$, which converges to 0 as $\ell$ increases, thus requiring an interpolation approximation with fewer points to achieve a comparable level of accuracy.
Indeed, as will be demonstrated later, to achieve an accuracy of $\calO(\varepsilon)$ with the single-level approximation in the function space setting with product-type weight parameters, the computational cost is $\calO(\varepsilon^{-\eta-\nu-\theta})$ for positive constants $\eta,\nu,\theta$ depending on the rates of convergence for dimension truncation, kernel approximation, and finite element approximation, respectively; whereas for the multilevel approximation, the cost is reduced to
$\calO(\varepsilon^{-\eta-\max(\nu,\theta)})$. The main contribution of this work is to introduce the multilevel kernel approximation along with a full error and cost analysis.

Parametric PDEs of the form \eqref{eq:pde} can be used to 
model steady-state flow through porous media and have been thoroughly studied in uncertainty quantification literature (see e.g., \cite{CST13,CGSS2000, CGST11,CDS10,TSGU13}). 
QMC methods have achieved much success in tackling parametric PDE problems, including evaluating expected values of quantities of interest
(see, e.g., \cite{GilbScheichl,GKS25,GrKNSSS15,KSS12,KSS15}), as well as density estimation (see \cite{GKS25}).
The two most common forms of random coefficient
are the uniform model
(see e.g., \cite{CDS10,KSS12,KSS15}) and lognormal model (see e.g., \cite{CST13,CGSS2000,TSGU13}).

Multivariate function approximation is another area where
QMC methods have recently been successful. 
One example is trigonometric approximation in periodic spaces, where lattice rules are used to evaluate the coefficients in a finite (or truncated) Fourier expansion (see e.g., \cite{BKUV17,Kam13,Kam18,KPV15, KMNN19,KSW06,KSW08}). Fast component-by-component algorithms for constructing good lattice rules for trigonometric approximation in periodic spaces were presented and analysed in \cite{CKNS20, CKNS21}. QMC methods have also been developed for approximation and reconstruction in non-periodic spaces, including reconstruction of functions in the half-period cosine space and Chebyshev space using tent-transformed lattice rules (see e.g., \cite{KMNN19,SCN16}) and approximation in certain normed Walsh spaces using digital nets (see e.g., \cite{BDK09,DKK08}). Furthermore, kernel interpolation in reproducing kernel Walsh spaces has been covered in \cite{LDH09}.

There has also been research on kernel methods for approximation such as radial basis approximation for interpolating scattered data, signal processing, and meshless methods for solving PDEs  (see e.g., \cite{Dyn89,Hardy71, NWW03,SW06,Wendland05,WS95}). 
Approximation in $L^2$ using kernel interpolation on a lattice
in a periodic setting was first analysed in \cite{ZLH06},
which also highlighted the advantage that
linear systems involved can be solved efficiently via FFT; see
also \cite{ZKH09} for the $L^\infty$ case. 
The paper \cite{KKKNS22} introduced single level kernel interpolation with lattice points for parametric PDEs and
also provided a full analysis of the $L^2$ error on $D \times \Omega$.

Multilevel methods were first introduced in~\cite{Heinrich01} for parametric integration and then developed further in~\cite{Giles07,Giles08} for pricing financial options by computing the expected value of a pay-off depending on a stochastic differential equation utilising paths simulated via Monte Carlo (MC). These papers showed that the overall computational cost of the multilevel estimator was lower than that of the direct estimate at the same level of error. Multilevel MC methods were extended to multilevel QMC in~\cite{GilesWaterhouse09}. Subsequently, multilevel methods with MC and QMC have been successfully used in several papers to compute expectations of quantities of interest from parametric PDEs (see e.g.,~\cite{CST13,CGST11,GilbScheichl,KSS15}). 

Approximating solutions to elliptic PDEs with random input has been a topic of great interest. Stochastic
collocation 
has been a popular choice (e.g.,~\cite{BNTT12,BTN10,CCS15,CDS10,TS07,NTW08a,NTW08b}). Although these methods can treat a larger class of PDEs, they are often susceptible to the curse of dimensionality. However, considering adaptive and/or anisotropic stochastic collocation can mitigate this issue (see e.g.,~\cite{CCS15,NTW08b,SF21}).  
The paper~\cite{TJWG15} employs a multilevel
strategy
to approximate the PDE solution on $D \times \Omega$, where
instead of using kernel interpolation on the parameter domain
the authors used sparse grid stochastic collocation.

To the best of our knowledge, this paper is the first to use a multilevel approach with a QMC method to approximate the solution of a PDE \emph{as a function over both the spatial and parametric domains}. 

The structure of the paper is as follows. Section~\ref{sec:background} summarises the problem setting and essential background on dimension truncation, finite element methods and kernel interpolation. Section~\ref{sec:MLKI} introduces the multilevel kernel approximation alongside a breakdown of the error and the corresponding multilevel cost analysis. Section~\ref{sec:stoch_reg} includes the regularity analysis required for the error analysis presented in Section~\ref{sec:error}. Practical details on implementing our multilevel approximation are covered in Section~\ref{sec:implementation}. Section~\ref{sec:numerical} presents results of  numerical experiments. Any technical proofs not detailed in the main text are provided in the appendix.

\section{Background}\label{sec:background}

\subsection{Notation}
Let $\bsnu=(\nu_j)_{j\geq 1}$ with $\nu_j\in\N_0$ be a multi-index, and let 
$|\bsnu| \coloneqq \sum_{j\geq 1} \nu_j$ and $\supp(\bsnu) \coloneqq \{j\geq 1:\nu_j\neq 0\}$. Define
$\calF$ to be the set of multi-indices with finite support, i.e.,
$\calF \coloneqq \{\bsnu\in\N_0^\infty\, :\, |\supp(\bsnu)| <\infty\}.$
Multi-index notation is used for the parametric derivatives, i.e.,  
the mixed partial derivative of order $\bsnu \in \calF$ with respect to $\bsy$ is denoted by
$\partial^{\bsnu} = \prod_{j\ge 1} (\partial/\partial y_j)^{\nu_j}$.

For $\bsnu, \bsm \in \calF$, we define
$\binom{\bsnu}{\bsm} \coloneqq \prod_{j\geq 1} \binom{\nu_j}{m_j}$
and interpret $\bsm\leq \bsnu$ to be $m_j \leq \nu_j$ for all $j\geq 1$ (i.e., comparison between indices is done componentwise). For a given sequence $\bsb=(b_j)_{j\geq1}$, we define
$\bsb^{\bsnu} \coloneqq \prod_{j\geq 1}b_j^{\nu_j}$.

The Stirling numbers of the second kind are given by
	\begin{align} \label{eq:stirling_def}
		S(0,0)\coloneqq 1, \quad S(n,k) \coloneqq \frac{1}{k!}\sum_{j=0}^k (-1)^{k-j}{k\choose j} j^n \quad  \text{for } 0 \leq k \leq n.
	\end{align}		
	
The notation $A\lesssim B$ means that there exists some $c>0$ such that $A\leq cB$, and $A \simeq B$ means that there exists constants $c_1,c_2>0$ such that $A \leq c_1B$ and $B \leq c_2 A$.

For $g\in  L^2(D\times\Omega)$,
we define the $L^2$-norm on $D \times \Omega$ by
	\begin{align*}
	\|g\|_{L^2(D\times\Omega)}
= \bigg(\int_\Omega\int_D |g(\bsx,\bsy)|^2\,\rd\bsx\,\rd\bsy\bigg)^{1/2}	
	= \bigg(\int_\Omega \|g(\cdot,\bsy)\|_{L^2(D)}^2\,\rd\bsy\bigg)^{1/2}.
	\end{align*}

\subsection{Parametric variational formulation}
Let $V\coloneqq H_0^1(D)$ denote the usual first order Sobolev space of functions on $\bsx\in D$  that vanish on the boundary, with associated norm $\|v\|_V \coloneqq \|\nabla v\|_{L^2(D)}.$  
Multiplying both sides of \eqref{eq:pde} by a test function $v\in V$ and then integrating with respect to $\bsx$, using the divergence theorem, yields the variational equation:
find $u(\cdot, \bsy) \in V$ such that 
	\begin{align}  \label{eq:variational_pde}
		\calA(\bsy;u(\cdot,\bsy),v) = \langle f,v\rangle \quad \text{for all } v\in V,
	\end{align}
where  $f \in V'$, with $V' \coloneqq H^{-1}(D)$ denoting the dual space of $V$, and
$\calA(\bsy;\cdot,\cdot): V \times V \to \R$ is the parametric bilinear form defined by
	\begin{align*}
	\calA(\bsy;u,v) \coloneqq \int_D \Psi(\bsx,\bsy)\nabla u(\bsx)\cdot\nabla v(\bsx) \,\rd\bsx \quad \text{for } u,v\in V.
	\end{align*}
The $L^2(D)$ inner product $\langle\,\cdot,\cdot\,\rangle$ is extended continuously to the duality pairing on $V \times V'$.

The variational problem \eqref{eq:variational_pde} is subject to the same assumptions as in \cite{KKS20,KKKNS22}:
\begin{enumerate}[label={(A\arabic*)},ref=A\theenumi]
	\item \label{asm:a_in_Linf} $\psi_0\in L^\infty(D)$ and $\sum_{j\geq 1} \|\psi_j\|_{L^\infty(D)}<\infty$,
	\item \label{asm:a_bounds} there are positive constants $\Psi_{\min}$ and $\Psi_{\max}$ such that 
	$0< \Psi_{\min}\leq \Psi(\bsx,\bsy)\leq \Psi_{\max} <\infty$ for all $\bsx\in D$ and $\bsy\in \Omega$,
	\item \label{asm:p_summ} $\sum_{j\geq 1} \|\psi_j\|^p_{L^\infty(D)}<\infty$ for some $0<p< 1$,
	\item \label{asm:a_W1inf} $\psi_0\in W^{1,\infty}(D)$ and $\sum_{j\geq 1} \|\psi_j\|_{W^{1,\infty}(D)}<\infty$,
	\item \label{asm:psi_decrease} $\|\psi_1\|_{L^\infty(D)} \geq \|\psi_2\|_{L^\infty(D)} \geq \cdots$, and
	\item \label{asm:poly_dom} the physical domain $D\subset \R^d$, where $d=1,2$ or $3$, is a convex and bounded polyhedron.
\end{enumerate}
Here $W^{1,\infty}(D)$ is the Sobolev space with essentially bounded first order weak derivatives, equipped with the norm $\|v\|_{W^{1,\infty}(D)}\coloneqq \max\{\|v\|_{L^\infty(D)},\|\nabla v\|_{L^\infty(D)}\}$. 

For the new multilevel analysis, we will also require additional assumptions:
	\begin{enumerate}[label={(A\arabic*)},ref=A\theenumi]
	\setcounter{enumi}{6}
	\item \label{asm:bbar_sum} $\sum_{j\geq 1} \|\psi_j\|_{W^{1,\infty}(D)}^q<\infty$ for some $0<p <q\leq  1$, and
	\item \label{asm:f}  $f \in L^2(D)$.
	\end{enumerate}
This paper focuses on a multilevel extension of the approximation method presented in \cite{KKKNS22}, 
and naturally inherits the assumptions from this paper. These assumptions can be loosened and similar results can be obtained, but they may come at the cost of increased computational cost or reduced order of convergence. For example, more general domains can be considered using specialised finite element methods or the forcing term can have reduced regularity, e.g., $f\in H^{-1+t}(D)$ for $t\in[0,1]$.

We additionally define
	\begin{align} \label{eq:bb_bar}
	 b_j \coloneqq \frac{\|\psi_j\|_{L^\infty(D)}}{\Psi_{\min}}
	 \quad \text{and} \quad 
	 \overline{b}_j \coloneqq \frac{\|\psi_j\|_{W^{1,\infty}(D)}}{\Psi_{\min}}.
	\end{align}

From the Lax-Milgram Lemma, \eqref{eq:variational_pde} is uniquely solvable for all $\bsy\in\Omega$ and this solution will satisfy the \emph{a priori} bound
\begin{align}\label{eq:lax-milgram}
	\|u(\cdot,\bsy)\|_V \leq \frac{\|f\|_{V'}}{\Psi_{\min}}.
	\end{align}
In addition, from \cite[Theorem~2.3]{KKS20} the parametric derivatives are bounded by 
\begin{align}\label{eq:mixed_stoch_der}
	\|\partial^{\bsnu}u(\cdot,\bsy)\|_{V} \leq \frac{\|f\|_{V'}}{\Psi_{\min}}(2\pi)^{|\bsnu|} \sum_{\bsm\leq\bsnu} |\bsm|!\,\bsb^\bsm \prod_{i\geq1}S(\nu_i,m_i)
	\quad \text{for } \bsnu \in \calF.
	\end{align}

\subsection{Dimension truncation}\label{sec:prelim_DT}
To approximate $u(\cdot, \bsy)$ in $\bsy$, the infinite-dimensional
parameter domain $\Omega$
must be first truncated to a finite number of dimensions $s$. 
This is done by setting $y_j = 0$ for $j > s$, where we define the truncated parameter $\bsy_{1:s} \coloneqq (y_1,y_2,\ldots,y_s,0,\ldots)$, or equivalently by truncating the
coefficient expansion \eqref{eq:coeff} to $s$ terms. 
The dimension-truncated solution is denoted by 
$u^s(\cdot, \bsy) \coloneqq u(\cdot, \bsy_{1:s})$ and it is obtained by 
solving the variational problem \eqref{eq:variational_pde} at $\bsy = \bsy_{1:s}$.
With a slight abuse of notation we treat $\bsy_{1:s}$ as a vector in 
the $s$-dimensional parameter domain $\Omega_s \coloneqq [0, 1]^s$.

The dimension-truncated problem is subject to all the same assumptions as the variational problem \eqref{eq:variational_pde} (i.e., Assumptions \eqref{asm:a_in_Linf}--\eqref{asm:f} hold), hence, the \emph{a priori} bound \eqref{eq:lax-milgram} and regularity bound \eqref{eq:mixed_stoch_der} also hold here. 
Additionally, 
we have from \cite[Theorem~4.1]{KKKNS22} that
		\begin{align}\label{eq:DT_KKKNS}
		\|u- u^{s}\|_{L^2(D\times\Omega)}
		\,\lesssim\, \|f\|_{V'} \,s^{-(\frac{1}{p}-\frac{1}{2})},
		\end{align}
where the implied constant is independent of $s$ and $f$.

\subsection{Finite element methods}

The solution to the PDE \eqref{eq:pde} 
will be approximated by discretising in space using the  finite element (FE) method.
We consider piecewise linear FE methods, however, the multilevel method
can be applied using more general discretisations.
Denote by $V_h\subset V$ the space of continuous piecewise linear functions on a shape-regular triangulation of $D$ with mesh width $h>0$ and $M \coloneqq \dim (V_h) = \calO(h^{-d})$.
For $\bsy \in \Omega$, the FE approximation of $u(\cdot, \bsy)$ from \eqref{eq:variational_pde} is obtained by finding $u_h(\cdot, \bsy) \in V_h$ such that
\begin{equation}
\label{eq:fe}
\calA(\bsy, u_h(\cdot, \bsy), v_h) \,=\, \langle f, v_h \rangle
\quad \text{for all } v_h \in V_h.
\end{equation}

The FE basis functions for the space $V_h$ are denoted by $\phi_{h, i}$, $i = 1, 2, \ldots, M$, and the FE approximation with coefficients given by $[u_{h, i}(\bsy)]_{i = 1}^{M}$ can be represented as 
\begin{equation}
\label{eq:fe-expansion}
u_h(\bsx, \bsy) \,=\, \sum_{i = 1}^{M} u_{h, i}(\bsy) \,\phi_{h, i}(\bsx)\,.
\end{equation}
Since $V_h\subset V$, the \emph{a priori} bound \eqref{eq:lax-milgram} and the regularity bound \eqref{eq:mixed_stoch_der} also hold for the FE approximation $u_h$, as well as for the FE approximation of the dimension-truncated solution, denoted $u_h^s$.

From \cite[Theorem~4.3]{KKKNS22}, we have that under Assumptions \eqref{asm:a_in_Linf}, \eqref{asm:a_bounds}, \eqref{asm:a_W1inf}, \eqref{asm:poly_dom}, and \eqref{asm:f}, the error of FE approximation satisfies
	\begin{align}\label{eq:FE_KKKNS}
	\|u(\cdot,\bsy)-u_h(\cdot,\bsy)\|_{L^2(D)} \,\lesssim\, h^{2}\,\|f\|_{L^2(D)}  \qquad \mbox{as}\,\,\, h\to 0,
	\end{align}
where the implied constant is independent of $h$,$f$ and $\bsy$. 
By Galerkin orthogonality, the FE error $u(\cdot, \bsy)-u_h(\cdot, \bsy)$ is orthogonal to $V_h$, i.e.,
	\begin{align}\label{eq:gal_orth}
	\calA(\bsy;\,u(\cdot,\bsy) - u_h(\cdot,\bsy),\,v_h) = 0 \quad \text{for all } v_h\in V_h. 
	\end{align}
We also define $\sfI: V \to V$ to be the identity operator and $\sfP^h_{\bsy} :V\to V_h$ to be the parametric FE projection operator onto $V_h$, which is defined for some $w\in V$ by 
	\begin{align}\label{eq:orth_proj}
	\calA(\bsy;\,(\sfI - \sfP^h_\bsy) w,\, v_h) = 0 \quad \text{for all } v_h\in V_h.
	\end{align}
The approximation $u_h(\cdot,\bsy)$ is the projection of $u$ onto $V_h$, i.e., $u_h = \sfP^h_\bsy u \in V_h$, and by the definition of a projection $(\sfP^h_\bsy)^2 = \sfP^h_\bsy$ on $V_h$. 

Under Assumptions \eqref{asm:a_bounds}, \eqref{asm:a_W1inf} and \eqref{asm:poly_dom}, we have the following result from \cite[Theorem~3.2.2]{Ciarlet} which holds for all $w~\in~H^2(D)\cap V$,
\begin{align}\label{eq:d_fe_error} 
\|(\sfI-\sfP^h_\bsy)w\|_{V} \lesssim h\,\|\Delta w\|_{L^2(D)}\quad  \text{as } h\to 0,
\end{align}
where the implied constant is independent of $h$ and $\bsy$.

\subsection{Lattice-based kernel interpolation}\label{sec:kernel_int}

The solution $u(\cdot, \bsy)$ will be approximated via kernel interpolation in the dimension-truncated parametric domain $\Omega_s$. Consider the weighted Korobov space $ \calH_{ \alpha, \bsgamma}(\Omega_s)$, which is the Hilbert space of one-periodic $L^2$ functions defined on $\Omega_s$ with absolutely convergent Fourier series and square-integrable mixed derivatives of order $\alpha$. 
We restrict $\alpha \geq 1$ to be an integer smoothness parameter\footnote{In general, $\alpha$ need not be an integer (e.g., see \cite{KKKNS22,SW2001}), however, to have a simple, closed-form representation of the reproducing kernel and norm we restrict ourselves to integer $\alpha$.} and include weight parameters $\bsgamma = \{\gamma_\setu > 0 : \setu \subseteq \{1:s\}\}$ that model the relative importance of different groups of parametric variables. The space $ \calH_{ \alpha, \bsgamma}(\Omega_s)$ is a reproducing kernel Hilbert space, equipped with the norm
	\begin{align}\label{eq:H_norm}
	\|g\|_{\calH_{\alpha, \bsgamma}(\Omega_s)}^2 \coloneqq \sum_{\setu \subseteq \{1:s\}}\frac{1}{(2\pi)^{2\alpha|\setu|}\gamma_\setu}\int_{[0,1]^{|\setu|}} \bigg|\int_{[0,1]^{s-|\setu|}} 
	\!\bigg(\prod_{j\in\setu}\frac{\partial^\alpha}{\partial y_j ^\alpha}\bigg)g(\bsy)\rd\bsy_{-\setu}\bigg|^2\,\rd\bsy_{\setu},
	\end{align}
where $\bsy_{\setu} \coloneqq (y_j)_{j\in\setu}$	and $\bsy_{-\setu} \coloneqq (y_j)_{j\in\{1:s\}\backslash\setu}$.
The reproducing kernel for this space, $\calK_{\alpha, \bsgamma} : \Omega_s \times \Omega_s \to \R$, is given by 
\begin{equation*}
\calK_{\alpha, \bsgamma}(\bsy, \bsy') \,\coloneqq\, 
\sum_{\setu\subseteq\{1:s\}}\gamma_\setu \prod_{j\in\setu}\bigg[(-1)^{\alpha+1}\frac{(2\pi)^{2\alpha}}{(2\alpha)!}B_{2\alpha}(|y_j-y'_j|)\bigg],
\end{equation*}
where $B_{2\alpha}$ is the Bernoulli polynomial of degree $2\alpha$.

Consider a set of lattice points $\{\bst_{k}\}_{k = 0}^{N - 1}$ defined by
	\begin{align*}
	\bst_k = \frac{k\bsz\, \mathrm{mod}\,N}{N} \qquad \mbox{for}\quad k = 0,\ldots,N-1,
	\end{align*}
where $\bsz \in \N^s$ is a generating vector with components in $\{1,\ldots,N-1\}$ that are coprime to~$N$. 
For $g \in \calH_{\alpha, \bsgamma}(\Omega_s)$, the lattice-based kernel interpolant
(defined using function values of $g$ evaluated at the points $\{\bst_{k}\}_{k = 0}^{N - 1}$) is given by
\begin{equation}
\label{eq:interp-g}
 I_N g(\bsy) \,\coloneqq\,  I^s_N g(\bsy)= \sum_{k = 0}^{N - 1} a_{N, k}\,\calK_{\alpha,\bsgamma}(\bst_{k}, \bsy),
\end{equation}
such that $I_N g(\bst_{k'}) = g(\bst_{k'})$ for all $k'=0,\ldots,N-1$. 
The generating vector $\bsz$ is obtained using a component-by-component (CBC) construction algorithm similar to those described in \cite{CKNS20,CKNS21,KKKNS22}, where components of the vector are selected to minimise a bound on the worst-case error of approximation using the kernel method.
To ensure that $I_N g$ interpolates $g$ at the lattice points,
the coefficients $a_{N, k}$ are obtained by solving the linear system
\begin{equation}\label{eq:KI_LinSys}
K_{N, \alpha,\bsgamma}\, \bsa_N \,=\, \bsg_N\,,
\end{equation}
with $ \bsa_N \,=\, [a_{N, k}]_{k = 0}^{N - 1}$, 
$K_{N, \alpha,\bsgamma} = [\calK_{\alpha,\bsgamma}(\bst_k, \bst_{k'})]_{k, k' = 0}^{N - 1}$ 
and $\bsg_N = [g(\bst_{k})]_{k = 0}^{N - 1}$.

Due to the periodic and symmetric nature of the kernel, along with the properties of the lattice point set, the elements of $K_{N, \alpha,\bsgamma}$ satisfy
	\begin{align*}
	[K_{N, \alpha,\bsgamma}]_{k,k'}
	= \calK_{\alpha,\bsgamma}(\bst_k, \bst_{k'})
	= \calK_{\alpha,\bsgamma}(\bst_k- \bst_{k'},\bszero) 
	=\calK_{\alpha,\bsgamma}(\bst_{(k-k')\, \mathrm{mod}\, N},\bszero) 
	\end{align*}
for $k,k' = 0,\ldots,N-1$. This implies that $K_{N, \alpha,\bsgamma}$ is a symmetric, circulant matrix uniquely determined by its first column and can be diagonalised via FFT at cost $\calO(N \log N)$. The kernel only needs to be evaluated at $\lceil N/2 \rceil$ lattice points, since the first column (excluding the first entry) is symmetric, and the linear system \eqref{eq:KI_LinSys} can be solved after diagonalising using the FFT. 

A bound on the approximation error for the kernel interpolation of $g \in \calH_{\alpha, \bsgamma}(\Omega_s)$ using a CBC generated lattice point set is given in \cite[Theorem~3.3]{KKKNS22}, which states that
	\begin{align}\label{eq:lattice_KI_error}
	\|(I-I_N)g\|_{L^2(\Omega_s)} \leq \frac{\kappa}{[\varphi(N)]^{\frac{1}{4\lambda}}}
	\bigg(
	\sum_{\setu\subseteq\{1:s\}}\max(|\setu|,1)\,\gamma_\setu^\lambda\,[2\zeta(2\alpha\lambda)]^{|\setu|}
	\bigg)^{\frac{1}{2\lambda}}
	\|g\|_{\calH_{\alpha,\bsgamma}(\Omega_s)}
	\end{align}
	for all $\lambda\in(\frac{1}{2\alpha},1]$
	with $\kappa \coloneqq \sqrt{2} \,(2^{2\alpha\lambda+1}+1)^{\frac{1}{4\lambda}}$. Here $I \coloneqq I^s:\calH_{\alpha, \bsgamma}(\Omega_s)\to \calH_{\alpha, \bsgamma}(\Omega_s)$ is the identity operator, $\zeta$ is the Riemann zeta function defined by $\zeta(x) = \sum_{j=1}^\infty j^{-x}$, and $\varphi$ is the Euler totient function. Note that the original theorem presented in \cite{KKKNS22} requires the number of points $N$ to be prime, however, using  \cite[Theorem 3.4]{KMN23} the result above has been extended to non-prime $N$. 
	
There are $2^s$ weight parameters $\{\gamma_\setu\}_{\setu\subseteq\{1:s\}}$, too many to specify individually in practice. Therefore, special forms of weights $\gamma_\setu$ have been considered, including:
\begin{itemize}
\item Product weights: $\gamma_\setu = \prod_{j\in\setu} \gamma_j $
for some positive sequence $(\gamma_j)_{j\geq 1}$;
\item POD (``product and order dependent") weights: 
$\gamma_\setu = \Gamma_{|\setu|}\prod_{j\in\setu} \gamma_j$
for positive sequences $(\gamma_j)_{j\geq 1}$ and $(\Gamma_j)_{j\geq 0}$;
\item SPOD (``smoothness-driven product and order dependent") weights:
$\gamma_\setu  = \sum_{\bsv_\setu \in \{1:\alpha\}^{|\setu|}}$ $ \Gamma_{|\bsv_\setu|}\prod_{j\in \setu}\gamma_{j,v_j}$ for positive sequences
$(\gamma_{j,v_j})_{j\geq 1}$ and $(\Gamma_j)_{j\geq 0}$.
\end{itemize}
The error bound \eqref{eq:lattice_KI_error} holds for all forms of weights, but the computational cost differs, see the next subsection.

This kernel and point set pairing resulting in a linear system that can be solved with $\calO(N\log N)$ cost is not unique, see \cite{JR19}. In particular, it is known that (higher order) digital nets paired with the Walsh kernel can achieve good convergence rates for the approximation problem, see various single-level algorithms in \cite{LDH09,DKK08,BDK09}.
Our multilevel kernel approximation method may be extended by instead using one of these alternative pairings.

\subsection{Single-level kernel interpolation for PDEs}\label{sec:SLKI} 
Lattice-based kernel interpolation was first applied to PDE problems  in \cite{KKKNS22}. 
Denoting the dimension-truncated FE solution of \eqref{eq:variational_pde} by  $u^s_h(\cdot,\bsy) \in V_h$, the estimate \eqref{eq:lattice_KI_error} can be applied to the single-level kernel interpolant (see \cite[Theorem~4.4]{KKKNS22}) to obtain 
\begin{align}\label{eq:KI_KKKNS}
\|u^s_h- I_N u^{s}_{h}\|_{L^2(D\times\Omega)}
\lesssim \varphi(N)^{-\frac{1}{4\lambda}}\|f\|_{V'}\,C_s(\lambda)
\end{align}
for all $\lambda \in (\frac{1}{2\alpha},1]$, where
\begin{align*}
[C_s(\lambda)]^{2\lambda}& \coloneqq \bigg(\sum_{\setu\subseteq\{1:s\}}\max(|\setu|,1)\,\gamma_\setu^\lambda\,[2\zeta(2\alpha\lambda)]^{|\setu|}\bigg)\\
&\qquad \times  \bigg(\sum_{\setu\subseteq\{1:s\}} \frac{1}{\gamma_\setu}
\bigg(\sum_{\bsm_\setu\in\{1:\alpha\}^{|\setu|}}|\bsm_\setu|!\,\bsb^{\bsm_\setu}
\prod_{i\in\setu}S(\alpha,m_i)\bigg)^2
 \bigg)^\lambda.
\end{align*}
In \cite{KKKNS22}, the weights $\gamma_\setu$ are chosen to ensure the constant $C_s(\lambda)$ can be bounded independently of dimension $s$. Different forms of weights (SPOD, POD, product) achieve dimension independent bounds with some concessions on the rate of convergence.

The single-level kernel interpolation methodology from \cite{KKKNS22} is summarised below: 
\begin{enumerate}

\item Compute the first column of $K_{N,\alpha,\bsgamma}$ in \eqref{eq:KI_LinSys} (i.e., compute $[\calK_{\alpha,\bsgamma}(\bst_k,\bszero)]_{k=0}^{N-1}$) with cost $\calO(s^\rho \,\alpha^{\varsigma} N)$, where $\rho=2$ and $\varsigma = 2$ for SPOD weights, $\rho=2$ and $\varsigma =0$ for POD weights, $\rho=1$ and $\varsigma = 0$ for product weights, see \cite[Table 2]{KKKNS22}.

\item Evaluate the coefficient \eqref{eq:coeff} at each lattice point and FE node to set up the stiffness matrix at the cost $\calO(s\, h^{-d}N)$.

\item Compute the FE solution $u^s_h$ (denoted as $g$ in \eqref{eq:KI_LinSys}) for the PDE associated with each lattice point to construct $\bsg_N$ on the right-hand side of \eqref{eq:KI_LinSys} for every FE node at the cost $\calO(h^{-\tau}N)$ for some $\tau>d$, with $\tau \approx d$ for a linear complexity FE solver (e.g., an algebraic multigrid solver).

\item Solve the circulant linear system \eqref{eq:KI_LinSys} for coefficients $\bsa_N$ of the interpolant for every FE node at the cost $\calO( h^{-d} N\log N)$. 
\end{enumerate}
The total cost of construction for the single-level interpolant is therefore
\begin{align}\label{eq:SL_cost}
 \mathrm{cost}(I_Nu_h^s) \simeq s^\rho\alpha^\varsigma N 
 + s\, h^{-d} N + h^{-\tau} N + h^{-d}N\log N.
\end{align}
There is a pre-computation cost (varies with the form of weights) associated with the CBC construction of the lattice generating vector $\bsz$. There is also a post-computation cost to assemble the approximation.

From \cite{KKKNS22}, the error satisfies
\begin{align*}
 	\mathrm{error}(I_Nu^s_h) \lesssim s^{-\kappa} + h^{\beta} + N^{-\mu},
\end{align*}
see \eqref{eq:DT_KKKNS}, \eqref{eq:FE_KKKNS}, \eqref{eq:KI_KKKNS},
with $\kappa= \frac{1}{p}-\frac{1}{2}$, $\beta=2$, and $\mu = \frac{1}{4\lambda}$ for $\lambda \in (\frac{1}{2\alpha},1]$.
The cost can now be expressed in terms of the error $\varepsilon>0$ as follows. We demand that each of the three components of the error is bounded above and below by multiples of $\varepsilon$, i.e., $s\simeq\varepsilon^{-\frac{1}{\kappa}}$, $h \simeq \varepsilon^\frac{1}{\beta}$ and $N\simeq \varepsilon^{-\frac{1}{\mu}}$. It follows that there exists $C>0$ such that $N^\mu \leq C s^\kappa$, which gives $\log N \leq \frac{1}{\mu}(\log C+ \kappa \log s)
  \leq \frac{1}{\mu}(\log C +\kappa)\,s$ since $s\geq 1$, and therefore $\log N \lesssim s$.
Assuming that $d\le \tau\le d + \frac{\beta}{\kappa}$ and treating $\alpha^\varsigma$ as a constant,
the cost \eqref{eq:SL_cost} can be bounded further by
  \begin{align}
  \label{eq:SLcostFinal}
  	\mathrm{cost}(I_Nu^s_h) \lesssim s^\rho N +s\, N h^{-d} 
  \simeq \varepsilon^{-\frac{\rho}{\kappa}-\frac{1 }{\mu}}+
  \varepsilon^{-\frac{1}{\kappa}-\frac{1 }{\mu}-\frac{d}{\beta}}
  \simeq \varepsilon^{-\max(\frac{\rho}{\kappa}+\frac{1}{\mu},\frac{1 }{\kappa}+\frac{1}{\mu}+\frac{d}{\beta})}.
\end{align}
In the case of product weights, we have $\rho=1$ and $\mathrm{cost}(I_Nu^s_h) \lesssim \varepsilon^{-\frac{1}{\kappa}-\frac{1}{\mu}-\frac{d}{\beta}}$.

\section{Multilevel kernel approximation}\label{sec:MLKI}

Consider a sequence of conforming FE spaces $\{V_\ell\}_{\ell = 0}^\infty$, 
where each $V_\ell \coloneqq V_{h_\ell}\subset V$ corresponds to a shape regular triangulation
$\calT_\ell$ of $D$ with mesh width 
$h_\ell \coloneqq \max \{ \diam(\tri) : \tri \in \calT_\ell\} > 0 $
and $\dim (V_\ell) = M_\ell < \infty$. Recall that $M_\ell = \calO(h_\ell^{-d})$.
Then for $\ell \in \N$, we denote the dimension-truncated FE approximation in the space $V_\ell$ by  
$u_\ell(\cdot, \bsy) \coloneqq u^s_{h_\ell}(\cdot,\bsy)=u^s_{h_\ell}(\cdot,\bsy_{1:s}) \in V_\ell$. 

For a maximum level $L \in \N$ and setting  $u_{-1} \coloneqq 0$, the multilevel  kernel approximation is given by \eqref{eq:ml-ker},
where $\{I_\ell\}_{\ell \in \N}$ is a sequence of interpolation operators such that each $I_\ell \coloneqq I_{N_\ell}$ is a real-valued kernel interpolant based on $N_\ell$ lattice points $\{\bst_{\ell, k}\}_{k = 0}^{N_\ell - 1}$ 
as defined in Section~\ref{sec:kernel_int}. 
The interpolants are ordered in terms of nonincreasing accuracy, or equivalently, nonincreasing
numbers of interpolation points, i.e., $N_0 \ge N_1 \ge N_2 \ge \cdots$. The intuition is that
 $u_\ell-u_{\ell-1}$ decreases in magnitude with increasing $\ell$, thus requiring fewer interpolation points to achieve  reasonable accuracy.

We assume that the FE solution at each level is approximated with increasingly fine meshes corresponding to mesh widths $h_0>h_1>\cdots$ (i.e., the approximation $u_\ell$ increases in accuracy as the level increases), so that approximations using large values of $N_\ell$ are compensated by coarser FE meshes, thus moderating the cost. To simplify the ML algorithm, we additionally assume that the FE spaces are nested, i.e.,
$V_{\ell} \subset V_{\ell + 1}$ for $\ell \in \N$. 
If non-nested FE spaces are used, a ``supermesh'' (the mesh corresponding to the space spanned by the basis functions of both $V_\ell$ and $V_{\ell-1}$) must be used \cite{CGRF18,FPPGW09}. This adds to the computational cost, but under standard assumptions the number of elements in the supermesh is proportional to the sum of the number of elements in the two FE meshes \cite{CF20}, and therefore the complexity order derived in this paper is not affected. The required modifications to the algorithm are described briefly at the end of Section~\ref{sec:implementation}.

Similarly, we assume the lattice points on each level are nested (i.e., the points used on each level $\ell$ form a subset of the points used on level $\ell-1$), which can be achieved using an embedded lattice rule (see \cite{KMN23}). Thus, the space spanned by the kernel basis functions on level $\ell$ is a subspace of the one spanned by the kernel basis functions on level $\ell-1$. This proves to be advantageous since the generating vector for the lattice only needs to be constructed once and FE evaluations can be reused between levels, e.g., evaluations used to compute $I_{\ell-1} u_{\ell-1}$ can be reused to compute $I_{\ell}u_{\ell-1}$.

To compute the multilevel kernel approximation, on  each level $\ell$, we compute the $N_\ell$-point interpolant $I_\ell$ of the \emph{difference of a FE approximation} on a fine mesh, $u_\ell(\cdot, \bsy) \in V_\ell$, and a coarse mesh, $u_{\ell - 1}(\cdot, \bsy) \in V_{\ell - 1}$. As a result, the final approximation is \emph{not} a direct interpolation of the solution $u$, and thus $I_L^{\rm{ML}}$ will be referred to exclusively as the multilevel kernel approximation.

\subsection{Error decomposition for multilevel methods}

The multilevel kernel approximation error can be expressed as (omitting the dependence on $\bsx$ and $\bsy$)
	\begin{align*}
	u - I^{\ML}_L u 
	&= u -\sum_{\ell = 0}^L I (u_\ell - u_{\ell - 1})+\sum_{\ell = 0}^L I(u_\ell - u_{\ell - 1})- \sum_{\ell = 0}^L I_\ell (u_\ell - u_{\ell - 1})\notag\\
	&= u - u_L + \sum_{\ell = 0}^L (I-I_\ell) (u_\ell - u_{\ell - 1}), 
	\end{align*}
where $I:\calH_{\alpha,\bsgamma}(\Omega_s) \to \calH_{\alpha,\bsgamma}(\Omega_s)$ denotes the identity operator, and we define $u_{-1} \coloneqq 0$. 
Following the methodology of \cite{KKKNS22}, we take the $L^2(D)$ norm and $L^2(\Omega)$ norm, then use the triangle inequality to obtain the error estimate, 
	\begin{align}\label{eq:total_err}
	\mathrm{error}(I^{ML}_Lu) 
	&= \|u - I^{\ML}_L u \|_{L^2(D\times\Omega)} \notag\\
	&\leq 
	\|u - u_L\|_{L^2(D\times\Omega)}
	+\sum_{\ell=0}^L\|(I-I_\ell) (u_\ell - u_{\ell - 1})\|_{L^2(D\times\Omega)}.
\end{align}

The first term in \eqref{eq:total_err} is often referred to as the \emph{bias} in multilevel literature, which can be further separated into a dimension truncation error and a FE error
\begin{align}\label{eq:split}
 \|u - u_L\|_{L^2(D\times\Omega)}
 \leq \|u - u^{s}\|_{L^2(D\times\Omega)}+\|u^{s} - u^{s}_{h_L}\|_{L^2(D\times\Omega)},
\end{align}
where the two components can be bounded using \eqref{eq:DT_KKKNS} and \eqref{eq:FE_KKKNS}.
The bias is controlled by choosing $s$ and $h_L$ as necessary to obtain a prescribed error. The second term in \eqref{eq:total_err} is the error associated with the multilevel scheme, which is controlled by the choice of interpolation points $N_\ell$ and the fineness of the FE mesh $h_\ell$ at each level.
 
\subsection{Cost analysis of multilevel methods}
The cost of constructing the multilevel kernel approximation is similar to the single-level interpolant, except now the cost is ``spread" over multiple levels.  Recall from Subsection~\ref{sec:SLKI} that the total cost of construction for the single-level kernel interpolant is given by \eqref{eq:SL_cost}.
For the multilevel algorithm, the total cost is from
	\begin{enumerate}
	\item evaluating the kernel functions for the full lattice point set, $\calO(s^\rho \,\alpha^{\varsigma} N_0)$,
	\item then, for each level $\ell$ with $N_\ell\le N_0$,
		\begin{enumerate}
		\item evaluating the coefficient \eqref{eq:coeff} at each lattice point and FE node to set up the stiffness matrix, $\calO(s\, h^{-d}_{\ell} N_{\ell})$,
		\item solving for the FE solution at each lattice point, $\calO(h_\ell^{-\tau} N_\ell)$, and
		\item solving the linear system for coefficients of the interpolant at every FE node, $\calO(h_\ell^{-d} N_\ell \log N_\ell)$.
		\end{enumerate} 
	\end{enumerate}
Since the lattice points are nested, 
the total cost of computation is
	\begin{align}\label{eq:ML_cost}
	\mathrm{cost}(I^\mathrm{ML}_Lu) \simeq s^\rho \alpha^\varsigma N_0 + \sum_{\ell=0}^L  (s\,h_\ell^{-d} N_\ell + h_\ell^{-\tau} N_\ell + h_\ell^{-d}N_\ell\log N_\ell).
	\end{align}
Since kernel functions only need to be evaluated for each lattice point once and can be reused at each level as necessary, this cost, given by $s^\rho \alpha^\varsigma N_0$, is independent of $\ell$ and is outside the summation. 	For details on practical implementation, see Section~\ref{sec:implementation}.

\subsection{Abstract complexity analysis}

Theorem~\ref{thm:ml-complexity} below is an abstract complexity theorem for the error and cost of our multilevel approximation. It specifies a choice of $s$, $L$, $N_\ell$ and $h_\ell$ for $\ell =0,\ldots,L$ such that the total error of approximation can be bounded by some given $\varepsilon>0$. 
Assumption~\eqref{asm:trun_error} is motivated by the bias split \eqref{eq:split} together with \eqref{eq:DT_KKKNS} and \eqref{eq:FE_KKKNS}, and $s$ is chosen to balance the two terms.
Assumption~\eqref{asm:cost} is motivated by the cost estimate \eqref{eq:ML_cost}, where $\alpha^\varsigma$ is treated as constant, and a bound on $\tau$ is assumed to simplify the cost.
The error analysis in Section~\ref{sec:error} will justify Assumption \eqref{asm:ml_error} with precise values of the relevant constants.

\begin{theorem} \label{thm:ml-complexity}
Given $h_0\in (0,1)$ and $d\ge1$, define $h_{\ell} \coloneqq h_0\,2^{-\ell}$ for $\ell\geq 0$, and suppose there are positive constants 
$\beta,\kappa,\mu,\rho$, and $\tau$ such that
\begin{enumerate}[label={\textnormal{(M\arabic*)}},ref=M\theenumi]
\item $\|u-u_L\|_{L^2(D\times\Omega)}
\lesssim  s^{-\kappa}+ h_L^{\beta} $,
\label{asm:trun_error}
\item $\|(I-I_0) u_0\|_{L^2(D\times\Omega)}
\lesssim N_0^{-\mu}$ and $\|(I-I_\ell) (u_\ell - u_{\ell-1})\|_{L^2(D\times\Omega)}
\lesssim N_\ell^{-\mu}h_{\ell-1}^{\beta}$\,\, for $\ell= 1,\ldots,L$, and 

\label{asm:ml_error}
\item $\mathrm{cost}(I^{\ML}_L u) \,\lesssim \,
	 \displaystyle 
	 s^\rho N_0 + \sum_{\ell=0}^L N_\ell\,( s\, h_\ell^{-d} + h_\ell^{-\tau}+h_\ell^{-d} \log N_\ell)$.
\label{asm:cost}
\end{enumerate}
Given $0<\varepsilon<\min(1,2h_0^\beta)$,
and assuming
$d\leq\tau \leq d + \frac{\beta}{\kappa}$, we may choose integers $L$  given by  \eqref{eq:cond_L}, $s\simeq h_L^{-\frac{\beta}{\kappa}}$, and $N_0,\ldots,N_L$ given by \eqref{eq:Nvals}, such that $\mathrm{error}(I^{\ML}_Lu) \lesssim \varepsilon$ and
	\begin{align}\label{eq:complexity}
	\mathrm{cost}(I^{\ML}_Lu)\lesssim
	\begin{cases}
	    \eps^{-\frac{\rho}{\kappa}-\frac{1}{\mu}}
	    	&\quad \mbox{when}\,\, \frac{d}{\beta}<\frac{1}{\mu}, \\
 	    \eps^{-\frac{\rho}{\kappa}-\frac{1}{\mu}}\, (\log \eps^{-1})^{1+\frac{1}{\mu}}
 	    	&\quad \mbox{when}\,\, \frac{d}{\beta}=\frac{1}{\mu}, \\
  		\eps^{-\frac{\rho}{\kappa}-\frac{1}{1+\mu}(\frac{d}{\beta}+1)} 
  			&\quad \mbox{when}\,\, \frac{1}{\mu}<\frac{d}{\beta}
  			\leq \frac{1}{\mu}+(\frac{1}{\mu}+1)\frac{\rho-1}{\kappa},\\
  		\eps^{-\frac{1}{\kappa}-\frac{d}{\beta}} 
  			&\quad \mbox{when}\,\, \frac{d}{\beta}>\frac{1}{\mu}+(\frac{1}{\mu}+1)\frac{\rho-1}{\kappa},
	\end{cases}
	\end{align}
where the implied constants depend on $h_0, d, \beta,\kappa,\mu,\rho$, and $\tau$.
\end{theorem}

\begin{proof}
Substituting Assumptions \eqref{asm:trun_error} and \eqref{asm:ml_error} into \eqref{eq:total_err} gives the error bound
	\begin{align*} 
	\mathrm{error}(I^{\ML}_Lu)\, \lesssim\, 
	s^{-\kappa} + h_L^{\beta} + N_0^{-{\mu}}+\sum_{\ell=1}^L  N_\ell^{-{\mu}}h_{\ell-1}^{\beta}.
	\end{align*}
Choosing $s$ to balance the two components of error in Assumption \eqref{asm:trun_error}, i.e., setting  $s^{-\kappa} \simeq h_L^\beta$, the bound further simplifies to
	\begin{align}\label{eq:simp_error}
	\mathrm{error}(I^{\ML}_L u)\, 
	&\lesssim\,
	{h_L^{\beta} 
	+ \sum_{\ell=0}^L  N_\ell^{-{\mu}}h_{\ell}^{\beta},}
	\end{align}
where the implied constant depends on a factor $\max(h_0^{-\beta},2^\beta)$.

We require that $\mathrm{error}(I^{\ML}_Lu)\, \lesssim\,\varepsilon$, which holds if each of the two terms in \eqref{eq:simp_error} is bounded by $\varepsilon/2$.  
So from the first term we choose $L$ such that $h_L^\beta = h_0^\beta\,2^{-L\beta} \leq \varepsilon/2$, yielding the conditions $L \geq \log_2 (2h_0^\beta\,\varepsilon^{-1})/\beta$.
Taking the smallest allowable value of $L$ with the ceiling function, we obtain
		\begin{align}\label{eq:cond_L}
		L \coloneqq \bigg\lceil \frac{\log_2(2h_0^\beta\, \varepsilon^{-1})}{\beta}\bigg\rceil, 
		\quad \mbox{which implies}\quad
		2^L \simeq \varepsilon^{-\frac{1}{\beta}}.
		\end{align}
To ensure $L\geq1$ so that we are not in the trivial single-level case, we require that the value inside the ceiling function is positive, giving the condition $\varepsilon < 2h_0^\beta$.  
For the second term in \eqref{eq:simp_error}, we demand that
	\begin{align}\label{eq:new_simp_error}
	\sum_{\ell=0}^L  N_\ell^{-{\mu}}h_{\ell}^{\beta} \leq  \frac{\varepsilon}{2}.
	\end{align}

Since we have assumed that $\tau \leq d + \frac{\beta}{\kappa}$, which follows from desiring $s\, h_\ell^{-d} \simeq h_L^{-\frac{\beta}{\kappa}}h_\ell^{-d} \geq h_\ell^{-\tau}$, Assumption \eqref{asm:cost} simplifies to
	\begin{align}\label{eq:simp_cost1}
	\mathrm{cost}(I^{\ML}_L u) 
	\lesssim  s^\rho N_0
	+\sum_{\ell=0}^L  N_\ell \,(s\, h_\ell^{-d} + h_\ell^{-d}\log N_\ell).
	\end{align}
If $\log N_\ell\lesssim s$ for all $\ell=0,\ldots,L$, then, using $s^{-\kappa} \simeq h_L^\beta$, the cost \eqref{eq:simp_cost1} can be bounded by 
\begin{align} \label{eq:simp_cost}
	\mathrm{cost}(I^{\ML}_L u) 
	\lesssim h_L^{-\frac{\rho\,\beta}{\kappa}} N_0 + h_L^{-\frac{\beta}{\kappa}}  \sum_{\ell=0}^L  N_\ell\,  h_\ell^{-d},
\end{align}	
where the first term represents a setup cost and the second term is the multilevel cost. 

We now proceed to choose $N_0,\ldots,N_L$ by minimising the multilevel cost term in  \eqref{eq:simp_cost}
subject to the constraint \eqref{eq:new_simp_error}, with equality instead of $\leq$. 
We will later verify that the condition $\log N_\ell \lesssim s$ is indeed true.
The Lagrangian for this optimisation is
	\begin{align*}
	\calL(\widehat N_0,\ldots, \widehat N_L,\chi) \coloneqq
	  h_L^{-\frac{\beta}{\kappa}} \sum_{\ell=0}^L  \widehat N_\ell \,h_\ell^{-d} 
	 + \chi\bigg(\sum_{\ell=0}^L \widehat N_\ell^{-{\mu}}h^{\beta}_\ell -  \frac{\varepsilon}{2}\bigg),
	\end{align*}
	where $\chi$ is the Lagrange multiplier and $\widehat N_\ell$ for $\ell=0,\ldots,L$ are continuous variables. This gives us the following first-order optimality conditions
		\begin{align}
		\frac{\partial \calL}{\partial \widehat N_\ell} 
		&=  h_L^{-\frac{\beta}{\kappa}} h_\ell^{-d} 
		- {\chi}\,{\mu}\,\widehat N_\ell^{-{\mu}-1}h_{\ell}^\beta = 0 \quad \mbox{for } \ell=0,\ldots,L,
		\label{eq:FOC1}\\
		\frac{\partial \calL}{\partial \chi} 
		&= \sum_{\ell=0}^L \widehat N_\ell^{-{\mu}}h_{\ell}^{\beta} -  \frac{\varepsilon}{2} =0. \label{eq:FOC2}
		\end{align}
Rearranging \eqref{eq:FOC1} gives
$\widehat N_\ell^{1 +{\mu}}h_\ell^{-(d+\beta)} = \chi\,\mu\, h_L^\frac{\beta}{\kappa}$ for $\ell=0,\ldots,L$, noting that the right-hand side is independent of~$\ell$. Thus, we have
$\widehat N_\ell^{1 + {\mu}}h_\ell^{-(d+\beta)} = \widehat N_0^{1 + {\mu}}h_0^{-(d+\beta)}$ for $\ell=1,\ldots,L$, so 
\begin{align}\label{eq:N_ell_choice}
	\widehat N_\ell = \widehat N_0\, \bigg(\frac{h_\ell}{h_0}\bigg)^{\frac{d+\beta}{1+\mu}}
	=\widehat N_0\, (2^{-\ell})^{\frac{d+\beta}{1+\mu}}
	\quad \mbox{for } \ell=1,\ldots,L.
	\end{align}
Substituting \eqref{eq:N_ell_choice} into \eqref{eq:FOC2}  gives
	\begin{align} \label{eq:N0}
	\widehat N_0^{-{\mu}}h_0^\beta\sum_{\ell=0}^L (2^\ell)^{\frac{\mu(d+\beta)}{1+\mu}-\beta}
	= \frac{\varepsilon}{2},
	\quad\mbox{which yields}\quad
	\widehat N_0
	= \bigg(2\,\varepsilon^{-1} h_0^\beta\sum_{\ell=0}^L 2^{\frac{(d\mu-\beta)\ell}{1+\mu}}\bigg)^\frac{1}{\mu}.
	\end{align}
To obtain integer values for $N_\ell$, we define 
		\begin{align}\label{eq:Nvals}
		N_\ell \coloneqq \big\lceil \widehat{N}_\ell \big\rceil=
		\Big\lceil\widehat N_0\, {2}^{\frac{-(d+\beta)\ell}{1+\mu}}\Big\rceil
		\qquad\mbox{for } \ell=0,\ldots,L.
	\end{align}
Since $N_\ell \geq \widehat N_\ell$, the bound \eqref{eq:new_simp_error} continues to hold for this choice of $N_\ell$. 
	
Clearly, $N_0 \ge N_1 \ge \cdots \ge N_L$ as required.
We now verify that $\log N_0 \lesssim s$. Since $\varepsilon <2h_0^\beta$, we have $\widehat{N}_0>1$ from \eqref{eq:N0}, and therefore
\begin{align*}
 N_0 <\widehat{N}_0+1 < 2\widehat{N}_0 
 &\le  2\bigg(2\,\varepsilon^{-1} h_0^\beta\, 
 (L+1)\,2^{\frac{|d\mu-\beta|}{1+\mu}L}
 \bigg)^\frac{1}{\mu}
 \lesssim \varepsilon^{- (1 + \frac{1}{\beta} + \frac{|d\mu-\beta|}{(1+\mu)\beta})\frac{1}{\mu}},
\end{align*}		
where we loosely overestimated the geometric series in \eqref{eq:N0} by taking $L+1$ times the largest possible term with absolute value in the exponent, and then used $L+1\le 2^L$ and $2^L \simeq \varepsilon^{-\frac{1}{\beta}}$ from \eqref{eq:cond_L}. Thus  
$\log N_0 \lesssim \log \varepsilon^{-1}$.
On the other hand, we have $s^{-\kappa} \simeq h_L^\beta \leq \frac{\varepsilon}{2}$ and so $s \gtrsim \varepsilon^{-\frac{1}{\kappa}}$. Hence we have $\log N_0 \lesssim s$ as required.
We conclude that the results from the optimisation with respect to the simplified cost function \eqref{eq:simp_cost} can be applied to the multilevel problem with cost given by Assumption \eqref{asm:cost}.
	
We now verify that the cost satisfies \eqref{eq:complexity} by substituting $N_0 \leq 2\widehat N_0$, $N_\ell \leq \widehat N_\ell+1 = \widehat N_0\,2^{-\frac{(d+\beta)\ell}{1+\mu}} + 1$, \eqref{eq:N0}, $h_\ell =h_0\,2^{-\ell}$, and $2^L \simeq \varepsilon^\frac{1}{\beta}$ into \eqref{eq:simp_cost}, resulting in
		\begin{align*}
		&\mathrm{cost}(I^{\ML}_L u) 
		\lesssim h_0^{-\frac{\rho\beta}{\kappa}}\varepsilon^{-\frac{\rho}{\kappa}}\widehat N_0+
		h_0^{-\frac{\beta}{\kappa}-d} \varepsilon^{-\frac{1}{\kappa}}
		\sum_{\ell=0}^L \Big(\widehat N_0\,2^{-\frac{(d+\beta)\ell}{1+\mu}} + 1\Big)\, 2^{d\ell}\\
		&= h_0^{-\frac{\rho\beta}{\kappa}}\varepsilon^{-\frac{\rho}{\kappa}}
		\Big(2\,\varepsilon^{-1} h_0^\beta\,E_L
		\Big)^\frac{1}{\mu}
		+h_0^{-\frac{\beta}{\kappa}-d}\varepsilon^{-\frac{1}{\kappa}}
		\Big(2\,\varepsilon^{-1} h_0^\beta\, E_L
		\Big)^\frac{1}{\mu}
		E_L
		+ h_0^{-\frac{\beta}{\kappa}-d}\varepsilon^{-\frac{1}{\kappa}}\sum_{\ell=0}^L 2^{d\ell}\\
		&\lesssim \varepsilon^{-\frac{\rho}{\kappa}-\frac{1}{\mu}}\,
		E_L^{\frac{1}{\mu}}
		+\varepsilon^{-\frac{1}{\kappa}-\frac{1}{\mu}}\,
		E_L^{\frac{1}{\mu}+1}
		+\varepsilon^{-\frac{1}{\kappa}-\frac{d}{\beta}}\\
		&\simeq
		\begin{cases}
	 \varepsilon^{-\frac{\rho}{\kappa}-\frac{1}{\mu}} 
	 &\mbox{when } \frac{d}{\beta}<\frac{1}{\mu},\\
	 \varepsilon^{-\frac{\rho}{\kappa}-\frac{1}{\mu}}(\log \varepsilon^{-1})^{\frac{1}{\mu}+1}
	 &\mbox{when } \frac{d}{\beta}=\frac{1}{\mu},\\
	 \varepsilon^{-\frac{\rho}{\kappa}-\frac{1}{1+\mu}(\frac{d}{\beta}+1)} +\varepsilon^{-\frac{1}{\kappa}-\frac{d}{\beta}}
	  &\mbox{when } \frac{d}{\beta}>\frac{1}{\mu},
	 \end{cases}
		\end{align*}
where the implied constant includes a factor $h_0^{-\max(\frac{\rho\beta}{\kappa},\frac{\beta}{\kappa}+d)}$, and we used
	\begin{align*}
	E_L :=
	\sum_{\ell=0}^L 2^{\frac{(d\mu-\beta)\ell}{1+\mu}}
	\simeq 
	 \begin{cases}
	 1 &\mbox{when } d\mu<\beta \; (\mbox{i.e., } \frac{d}{\beta}<\frac{1}{\mu}) ,\\
	 L
	 \simeq \log \varepsilon^{-1} &\mbox{when } d\mu=\beta\; 
	 (\mbox{i.e., }  \frac{d}{\beta}=\frac{1}{\mu} ),\\
	 2^{\frac{d\mu-\beta}{1+\mu}L}
	 \simeq \varepsilon^{-\frac{d\mu-\beta}{\beta(1+\mu)}}
	  &\mbox{when } d\mu>\beta \; (\mbox{i.e., } \frac{d}{\beta}>\frac{1}{\mu} ).
	 \end{cases}
	\end{align*}%
The final case in the cost (i.e., when $\frac{d}{\beta}>\frac{1}{\mu}$) is split into two cases based on which of the two terms dominates, resulting in the four cases in \eqref{eq:complexity}.
\end{proof}

The third case in \eqref{eq:complexity} becomes obsolete when $\rho=1$, i.e., when we have product weights. In this scenario, to achieve an accuracy $\bigO(\varepsilon)$, the cost for the multilevel approximation is $\calO\big(\varepsilon^{-\frac{1}{\kappa}-\max(\frac{1}{\mu},\frac{d}{\beta})}\big)$, compared to  $\calO\big(\varepsilon^{-\frac{1}{\kappa}-\frac{1}{\mu}-\frac{d}{\beta}}\big)$ for the single-level approximation.
This multilevel cost is near-optimal, since the cost of a single FE evaluation at the finest level is $\bigO(\varepsilon^{-\frac{{1}}{\kappa}-\frac{d}{\beta}})$ and the cost of interpolation at level $0$ at a single node is $\bigO(\varepsilon^{-\frac{1}{\kappa}-{\frac{1}{\mu}}})$.

We note that when $\rho>1$, we may encounter the scenario where the single-level cost and multilevel cost have the same order. This occurs when the setup cost $s^\rho N_0$ in Assumption \eqref{asm:cost} dominates, either due to a large dimension $s$ or because the contribution to the cost from the FE solves is relatively small due to fast convergence in FE error (i.e., $\beta$ is large). Such cases are exceptional, and it would be unnecessary to consider a multilevel approach in these situations.

\section{Parametric regularity analysis} \label{sec:stoch_reg}

Analysis of the multilevel kernel approximation error for parametric PDEs requires bounds on the mixed derivatives with respect to both the physical variable $\bsx$ and the parametric variable $\bsy$ simultaneously. The proofs of all parametric regularity lemmas in this section are given in Appendix~\ref{app:stoch_reg_pf}.

The following lemma provides a bound on the Laplacian of the derivatives (with respect to the parametric variables) of the solution to \eqref{eq:variational_pde}. It shows that $u$ and its derivatives with respect to $\bsy$ possess sufficient spatial regularity to establish estimates for the FE error.
Recall that Stirling numbers are defined by \eqref{eq:stirling_def}.

\begin{lemma}\label{lem:laplace_d_u}
Under Assumptions \eqref{asm:a_in_Linf}, \eqref{asm:a_bounds}, \eqref{asm:a_W1inf} and \eqref{asm:f}, for every $\bsy\in \Omega$, let $u(\cdot,\bsy)\in V$ be the solution to the problem \eqref{eq:variational_pde}.  Then for every $\bsnu\in \calF$ and all $\bsy\in\Omega$, we have that $\partial^\bsnu u(\cdot,\bsy)\in H^2(D)\cap V$ and 
	\begin{align}\label{eq:laplace_derivative}
	\|\Delta(\partial^{\bsnu}u(\cdot,\bsy))\|_{L^2(D)} \lesssim \,\|f\|_{L^2(D)}\, (2\pi)^{|\bsnu|}\sum_{\bsm\leq \bsnu} (|\bsm|+1)!\,\overline{\bsb}^{\bsm} \prod_{i\geq 1} S(\nu_i,m_i),
	\end{align}
where $\overline{\bsb}=(\overline{b}_j)_{j\geq 1}$ is defined in \eqref{eq:bb_bar} and the implied constant is independent of~$\bsy$.
\end{lemma}

The following lemma gives a bound on the derivatives with respect to the parametric variables of the FE error in $V$.

\begin{lemma}\label{lem:d_u-uh} Under Assumptions \eqref{asm:a_in_Linf}, \eqref{asm:a_bounds}, \eqref{asm:a_W1inf} and \eqref{asm:f}, for every $\bsy\in \Omega$, let $u(\cdot,\bsy)\in V$ be the solution to \eqref{eq:variational_pde} and $u_h(\cdot,\bsy)\in V_h$ be its piecewise linear FE approximation \eqref{eq:fe}. Then for every $\bsnu\in \calF$, and sufficiently small $h>0$, we have 
	\begin{align}\label{eq:FE_error_d_bnd}
	\|\partial^{\bsnu}(u-u_h)(\cdot,\bsy)\|_{V} \lesssim \,h\,\|f\|_{L^2(D)}\, (2\pi)^{|\bsnu|}\sum_{\bsm \leq \bsnu}(|\bsm|+2)!\,\overline{\bsb}^\bsm\prod_{i\geq 1}S(\nu_i,m_i),
	\end{align}
where $\overline{\bsb}=(\overline{b}_j)_{j\geq 1}$ is defined in \eqref{eq:bb_bar} and the implied constant is independent of $h$ and $\bsy$.
\end{lemma}

We are interested in regularity estimates with respect to the $L^2$ norm. Thus, Lemma \ref{lem:L2_FE_Err_der} presents the bound on the $L^2$ norm of the derivatives of FE error obtained using a duality argument. 

\begin{lemma}\label{lem:L2_FE_Err_der}
Under Assumptions \eqref{asm:a_in_Linf}, \eqref{asm:a_bounds}, \eqref{asm:a_W1inf} and \eqref{asm:f}, for every $\bsy\in \Omega$, let $u(\cdot,\bsy)\in V$ be the solution to \eqref{eq:variational_pde} and $u_h(\cdot,\bsy)\in V_h$ be its piecewise linear FE approximation \eqref{eq:fe}. Then for every $\bsnu\in \calF$, and  sufficiently small $h>0$, we have 
\begin{align*}
\|\partial^\bsnu (u-u_h)(\cdot,\bsy)\|_{L^2(D)} 
\lesssim h^2\,\|f\|_{L^2(D)} \,(2\pi)^{|\bsnu|} \sum_{\bsm\leq\bsnu} (|\bsm|+5)!\, \overline{\bsb}^{\bsm} \prod_{i\geq 1} S(\nu_i,m_i),
\end{align*}
where $\overline{\bsb}=(\overline{b}_j)_{j\geq 1}$ is defined in \eqref{eq:bb_bar} and the implied constant is independent of $h$ and $\bsy$.
\end{lemma}

An alternative method for deriving similar regularity bounds is to exploit the $(b,\epsilon)$-holomorphy of the PDE solution map, as discussed in \cite{CCS15,DGLS18}. However, this would require some modification to the theoretical framework of this problem. 

\section{Multilevel error analysis}\label{sec:error}

We are now ready to derive an estimate for the final component of the error in \eqref{eq:total_err}.

\subsection{Estimating the multilevel FE error}
Clearly, $\|g\|_{L^2(\Omega)} \equiv \|g\|_{L^2(\Omega_s)}$ for any $g$ that is a function solely of $\bsy_{1:s}$. Then, taking the $L^2(\Omega)$-norm and using the triangle inequality, we have from \eqref{eq:lattice_KI_error} the following pointwise bound for some $\bsx\in D$,
	\begin{align}\label{eq:ML_KI}
	&\|(I-I_\ell) (u_\ell - u_{\ell - 1})(\bsx,\cdot)\|_{L^2(\Omega)}
	=\|(I-I_\ell) (u_\ell - u_{\ell - 1})(\bsx,\cdot)\|_{L^2(\Omega_s)} \notag\\
	 &\leq \frac{\kappa}{[\varphi(N_\ell)]^\frac{1}{4\lambda}}
	\bigg(
	\sum_{\setu\subseteq\{1:s\}}\max(|\setu|,1)\,\gamma_\setu^\lambda\,[2\zeta(2\alpha\lambda)]^{|\setu|}
	\bigg)^\frac{1}{2\lambda} 
	\|(u_\ell - u_{\ell - 1})(\bsx,\cdot)\|_{\calH_{\alpha,\bsgamma}(\Omega_s)}.
	\end{align}
where
\begin{align} \label{eq:ML_DT_FE}
\| (u_\ell - u_{\ell - 1})(\bsx,\cdot)\|_{\calH_{\alpha,\bsgamma}(\Omega_s)}\leq 
\| (u^{s} - u^{s}_{h_{\ell}})(\bsx,\cdot)\|_{\calH_{\alpha,\bsgamma}(\Omega_s)} +
 \|( u^{s} - u^{s}_{h_{\ell-1}})(\bsx,\cdot)\|_{\calH_{\alpha,\bsgamma}(\Omega_s)}.
\end{align}
It then follows that the sum over $\ell$ in \eqref{eq:total_err} can be bounded by
\begin{align}\label{eq:ML_error_decomp}
&\|(I-I_0)\,u_0\|_{L^2(\Omega_s \times D)}
+ \sum_{\ell=1}^L \|(I-I_\ell) (u_\ell - u_{\ell - 1})\|_{L^2(\Omega_s\times D)}\notag\\
&\leq 
\frac{\|f\|_{V'}C_s(\lambda)}{[\varphi(N_0)]^\frac{1}{4\lambda}}
+ \sum_{\ell=1}^L  \frac{\kappa}{[\varphi(N_\ell)]^\frac{1}{4\lambda}}
	\bigg(
	\sum_{\setu\subseteq\{1:s\}}\max(|\setu|,1)\,\gamma_\setu^\lambda\,[2\zeta(2\alpha\lambda)]^{|\setu|}
	\bigg)^\frac{1}{2\lambda} \notag\\
	&\quad\times \bigg(\sqrt{\int_D \|(u^{s} -u^{s}_{h_\ell})(\bsx,\cdot)\|_{\calH_{\alpha,\bsgamma}(\Omega_s)}^2\,\rd\bsx} +\sqrt{\int_D \|( u^{s}- u^{s}_{h_{\ell-1}})(\bsx,\cdot)\|_{\calH_{\alpha,\bsgamma}(\Omega_s)}^2\,\rd\bsx} \,\bigg),
\end{align}
where we separate out the $\ell=0$ term as it is simply the standard single-level kernel interpolation error for which the bound \eqref{eq:KI_KKKNS} applies. Combining \eqref{eq:ML_KI} and \eqref{eq:ML_DT_FE}  with the triangle inequality gives the last line.

Therefore, we seek to estimate the individual finite element error terms in the above expression to estimate the full multilevel kernel approximation error.

\begin{theorem}
\label{thm:ML_FE_error} Under Assumptions \eqref{asm:a_in_Linf}, \eqref{asm:a_bounds}, \eqref{asm:a_W1inf} and \eqref{asm:f}, for $\alpha\geq 1$ and weight parameters $(\gamma_{\setu})_{\setu\subset \N}$, let $u^s \in V$ be the solution to \eqref{eq:variational_pde} and $u^s_h\in V_h$ be the FE approximation \eqref{eq:fe} to $u^s$. Then the  following estimate holds
\begin{align*}
& \sqrt{\int_D \|(u^{s} -u^{s}_{h})(\bsx,\cdot) \|_{\calH_{\alpha,\bsgamma}(\Omega_s)}^2\,\rd\bsx} 
\\&
\lesssim
h^2\, \|f\|_{L^2(D)}\sqrt{\sum_{\setu\subseteq \{1:s\}}\frac{1}{\gamma_\setu}\bigg(\sum_{\bsm_\setu \in \{1:\alpha\}^{|\setu|}}(|\bsm_\setu|+5)!\,\overline \bsb^{\bsm_\setu}\prod_{i\in\setu} S(\alpha,m_i)\bigg)^2},
\end{align*}
where the implied constant is independent of $h$.
\end{theorem}

\begin{proof}
From the definition of the norm of $\calH_{\alpha,\bsgamma}(\Omega_s)$ given in \eqref{eq:H_norm} and the Cauchy-Schwarz inequality, it follows that for any $\bsx\in D$,
	\begin{align*}
	&\|( u^{s}- u^{s}_{h})(\bsx,\cdot)\|_{\calH_{\alpha,\bsgamma}(\Omega_s)}^2 \\
	&= \sum_{\setu\subseteq \{1:s\}} \frac{1}{(2\pi)^{2\alpha|\setu|}\gamma_\setu}\int_{[0,1]^{|\setu|}}
	\bigg|\int_{[0,1]^{s-|\setu|}}\bigg(\prod_{j\in\setu}\frac{\partial^\alpha}{\partial y_j^\alpha} \bigg) ( u^{s}- u^{s}_{h})(\bsx,\bsy)\, \rd \bsy_{-\setu} \bigg|^2\rd\bsy_{\setu}\\
	&\leq \sum_{\setu\subseteq \{1:s\}} \frac{1}{(2\pi)^{2\alpha|\setu|}\gamma_\setu}\int_{[0,1]^{|\setu|}}
	\int_{[0,1]^{s-|\setu|}}\bigg|\bigg(\prod_{j\in\setu} \frac{\partial^\alpha}{\partial y_j^\alpha} \bigg) ( u^{s}- u^{s}_{h})(\bsx,\bsy)\bigg|^2\, \rd \bsy_{-\setu}\, \rd\bsy_{\setu}.
	\end{align*}
Now, taking the $L^2$-norm with respect to $D$ and applying the Fubini-Tonelli theorem gives
	\begin{align*}
 	&	\int_D \| ( u^{s}- u^{s}_{h})(\bsx,\cdot)\|_{\calH_{\alpha,\bsgamma}(\Omega_s)}^2\,\rd\bsx\\
 		&\leq \sum_{\setu\subseteq \{1:s\}} \frac{1}{(2\pi)^{2\alpha|\setu|}\gamma_\setu} 
	\int_{[0,1]^{s}}\int_D\bigg(\bigg(\prod_{j\in\setu} \frac{\partial^\alpha}{\partial y_j^\alpha} \bigg) ( u^{s}- u^{s}_{h})(\bsx,\bsy)\bigg)^2\,\rd\bsx\, \rd \bsy\\	 
	&=\sum_{\setu\subseteq \{1:s\}} \frac{1}{(2\pi)^{2\alpha|\setu|}\gamma_\setu} 
	\int_{[0,1]^{s}}\bigg\|\bigg(\prod_{j\in\setu} \frac{\partial^\alpha}{\partial y_j^\alpha} \bigg) ( u^{s}- u^{s}_{h})(\cdot,\bsy)\bigg\|^2_{L^2(D)}\, \rd \bsy\\
		&\lesssim  \sum_{\setu\subseteq \{1:s\}} \frac{h^4\,\|f\|^2_{L^2(D)}}{\gamma_\setu}
	\int_{[0,1]^{s}}
	 \bigg(\sum_{\bsm_\setu\in \{1:\alpha\}^{|\setu|}} (|\bsm_\setu|+5)! \, \overline{\bsb}^{\bsm_\setu}
	\prod_{i\in\setu}  S(\alpha,m_i)\bigg)^2
	\, \rd \bsy\\
	&=\, \,h^4\,\|f\|^2_{L^2(D)}  \sum_{\setu\subseteq \{1:s\}} \frac{1}{\gamma_\setu} \,
	 \bigg(\sum_{\bsm_\setu\in \{1:\alpha\}^{|\setu|}} (|\bsm_\setu|+5)! \, \overline{\bsb}^{\bsm_\setu}
	\prod_{i\in\setu}S(\alpha,m_i)\bigg)^2.
	\end{align*}
In the step with $\lesssim$, we applied Lemma~\ref{lem:L2_FE_Err_der} for each $\setu$ with $\bsnu$ given by $\nu_i = \alpha$ for $i\in\setu$ and $\nu_i = 0$ for $i\notin\setu$,
and $\bsm_\setu$ denotes the multi-index with $m_i=0$ for all $i\not\in\setu$, with
$\overline{\bsb}^{\bsm_\setu}\coloneqq \prod_{i\in\setu} \overline{b}_i^{m_i}$. The implied constant is independent of $h$. Taking the square-root completes the proof.
\end{proof}

Applying Theorem \ref{thm:ML_FE_error} to \eqref{eq:ML_error_decomp}, we derive the bound on the error of the multilevel kernel approximation of $u_L = u^s_{h_L}$ which is presented below. The constant $C_s(\lambda)$ in~\eqref{eq:KI_KKKNS} is bounded from above by $\calD_s(\lambda)$ defined in \eqref{eq:implied_cons} since $|\bsm_\setu|!<  (|\bsm_\setu|+5)!$ and $b_j \leq \overline{b}_j$. We also use the fact that $\|f\|_{V'} \lesssim \|f\|_{L^2(D)}$ to include the first term in \eqref{eq:ML_error_decomp} into the summation as the $\ell=0$ term.

\begin{theorem}\label{thm:ML_full_error} Under Assumptions \eqref{asm:a_in_Linf}--\eqref{asm:poly_dom} and \eqref{asm:f}, for $\alpha \in \N$ and $\lambda\in (\frac{1}{2\alpha},1]$, the error of the multilevel kernel approximation of the dimension-truncated FE solution $u_L$ satisfies
	\begin{align*}
	\sum_{\ell=0}^L\|(I-I_\ell)(u_\ell-u_{\ell-1})\|_{L^2(\Omega \times D)}
	\lesssim \|f\|_{L^2(D)}\, \calD_{s}(\lambda)
	\,\sum_{\ell=0}^L [\varphi(N_\ell)]^{-\frac{1}{4\lambda}}h_{\ell-1}^2,
	\end{align*}
with
	\begin{align}\label{eq:implied_cons}
	 [\calD_{s}(\lambda)]^{2\lambda} &\coloneqq
	  \bigg(\sum_{\setu\subseteq\{1:s\}}\max(|\setu|,1)\,\gamma_\setu^\lambda\,[2\zeta(2\alpha\lambda)]^{|\setu|}\bigg)\notag\\
&\qquad \times  \bigg(\sum_{\setu\subseteq\{1:s\}}
 \frac{1}{\gamma_\setu}
\bigg(\sum_{\bsm_\setu\in\{1:\alpha\}^{|\setu|}}(|\bsm_\setu|+5)!\,
\overline\bsb^{\bsm_\setu}\prod_{i\in\setu}S(\alpha,m_i)\bigg)^2\bigg)^\lambda,
	\end{align}
where $\overline{\bsb}=(\overline{b}_j)_{j\geq 1}$ is defined in \eqref{eq:bb_bar} 
and $h_{-1} \coloneqq 1$.
\end{theorem}

\subsection{Choosing the weight parameters $\gamma_\setu$}
We now choose the weights $\gamma_\setu$ to minimise $\calD_{s}(\lambda)$ and select $\lambda$ to obtain the best possible convergence rate, while ensuring that $\calD_{s}(\lambda)$ is bounded independently of $s$.

\begin{theorem}\label{thm:weight_select}
Suppose that Assumptions \eqref{asm:a_in_Linf}--\eqref{asm:f} hold. The choice of weights
	\begin{align}\label{eq:op_weights}
	\gamma_\setu \coloneqq 
	\Bigg(
	\frac{\sum_{\bsm_\setu\in\{1:\alpha\}^{|\setu|}} (|\bsm_\setu|+5)!\,\overline{\bsb}^{\bsm_\setu}\prod_{i\in\setu} S(\alpha,m_i) }
	{\sqrt{|\setu|\,[2\zeta(2\alpha\lambda)]^{|\setu|}}}
	\Bigg)^\frac{2}{1+\lambda}
	\end{align}
for $\emptyset\neq \setu\subset\N$, $|\setu|<\infty$, and $\gamma_{\emptyset}\coloneqq 1$, minimise $\calD_{s}(\lambda)$ given in \eqref{eq:implied_cons}. Here $\overline{\bsb}=(\overline{b}_j)_{j\geq 1}$ is defined in \eqref{eq:bb_bar}. In addition, if we take $\alpha \coloneqq  \lfloor 1/q+1/2 \rfloor $ and $\lambda \coloneqq \frac{q}{2-q}$, then $\calD_{s}(\lambda)$ can be bounded independently of dimension $s$. 
Thus, the multilevel kernel approximation of the dimension-truncated FE solution of \eqref{eq:variational_pde} satisfies
	\begin{align}\label{eq:MLKI_final_error}
	&\sum_{\ell=0}^L\|(I-I_\ell)(u_\ell-u_{\ell-1})\|_{L^2(\Omega \times D)} \lesssim
	\|f\|_{L^2(D)}\,\sum_{\ell=0}^L [\varphi(N_\ell)]^{-(\frac{1}{2q}-\frac{1}{4})}h_{\ell-1}^2 ,
	\end{align}
where we define $h_{-1}\coloneqq 1$, and the implied constant is independent of $s$.
\end{theorem}

\begin{proof}
The proof of this theorem follows very closely to that of \cite[Theorem~4.5]{KKKNS22}.
We seek to choose weights to bound the constant $\calD_s(\lambda)$ independently of dimension. Applying \cite[Lemma~6.2]{KSS12} to \eqref{eq:implied_cons}, we find that the weights given by \eqref{eq:op_weights} minimise $[\calD_s(\lambda)]^{2\lambda}$.

Substituting \eqref{eq:op_weights} into \eqref{eq:implied_cons} and simplifying gives
\begin{align}\label{eq:min_C}
& [\calD_s(\lambda)]^\frac{2\lambda}{1+\lambda} \leq \sum_{\bsm\in\{0:\alpha\}^{s}}\bigg((|\bsm|+5)!\prod_{i=1}^s \beta_i^{m_i}\bigg)^\frac{2\lambda}{1+\lambda},
\end{align}
where we use the  bound $\max(|\setu|,1)\leq [e^{1/e}]^{|\setu|}$ and define 
$\beta_i \coloneqq S_{\max}(\alpha)\,[2e^{1/e}\zeta(2\alpha\lambda)]^\frac{1}{2\lambda}\,\overline{b}_i$,
with $S_{\max}(\alpha) \coloneqq \max_{1\leq m \leq \alpha} S(\alpha,m)$.

Defining a new sequence $d_j\coloneqq \beta_{\lceil j/\alpha\rceil}$ for $j\geq 1$, we can write, for $\bsm\in\{0:\alpha\}^s$,
	\begin{align*}
	\prod_{i=1}^s \beta_i^{m_i} = \prod_{i\in\setv_\bsm} d_i,
	\end{align*}
where  $\setv_\bsm \coloneqq \{1,2,\ldots,m_1,\alpha+1,\alpha+2,\ldots,\alpha+m_2\ldots,(s-1)\alpha+1,\ldots,(s-1)\alpha+m_s\}$. The cardinality of $\setv_\bsm$ is $|\setv_\bsm|=\sum_{i=1}^s m_i = |\bsm|$.
Then, following \cite{KKKNS22}, the upper bound  \eqref{eq:min_C} can be further bounded above by 
	\begin{align}
\sum_{\satop{\setv\subset\N}{|\setv|<\infty}}\bigg((|\setv|+5)!\prod_{i\in\setv} d_i\bigg)^\frac{2\lambda}{1+\lambda} \leq 
	\sum_{\ell\geq 0} [(\ell+5)!]^\frac{2\lambda}{1+\lambda}
	\frac{1}{\ell!}\bigg(\sum_{i\geq1} d_i^\frac{2\lambda}{1+\lambda}\bigg)^\ell \label{eq:final_bnd}.
	\end{align}

Now following from Assumption \eqref{asm:bbar_sum}, we know $\sum_{j\geq 1} \overline b_j^q < \infty$, so we choose $\lambda$ such that $\frac{2\lambda}{1+\lambda} = q$ which gives $\lambda = \frac{q}{2-q}$. Then
	\begin{align*}
	\sum_{i\geq1} d_i^\frac{2\lambda}{1+\lambda} 
	= \sum_{i\geq1} d_i^{\,q} 
	=\alpha\sum_{i\geq1} \beta_i^q
	=\alpha\max(1,S_{\max}(\alpha)\,[2e^{1/e}\zeta(2\alpha\lambda)]^\frac{1}{2\lambda})^q\sum_{i\geq1}\overline b_i^q < \infty,
	\end{align*}
provided that $2\alpha\lambda >1$, which can be ensured by choosing $\alpha > 1/q-1/2$. This condition can be satisfied by taking 
$\alpha 
= \lfloor (\frac{1}{q} -\frac{1}{2})+1\rfloor
= \lfloor \frac{1}{q} +\frac{1}{2}\rfloor$.

To show the full series \eqref{eq:final_bnd} converges we perform the ratio test with the terms of the series given by $T_\ell \coloneqq [(\ell+5)!]^q\,\frac{1}{\ell!}\big(\sum_{i\geq1} d_i^q\big)^\ell$. This shows that $\calD_s(\lambda)$ can be bounded from above independently of dimension since
	\begin{align*}
	\lim_{\ell\to \infty} \bigg|\frac{T_{\ell+1}}{T_{\ell}}\bigg|  
	= \bigg(\sum_{i\geq1} d_i^q\bigg) \lim_{\ell\to \infty} \frac{(\ell+1+5)^q}{\ell+1}
	&\leq \bigg(\sum_{i\geq1} d_i^q\bigg) \lim_{\ell\to \infty} \frac{2^q(\ell+1)^q}{\ell+1}
	=0,
	\end{align*}
where the inequality follows from observing $2(\ell+1) \geq \ell+6$ for $\ell\geq4$.

We can conclude that for the choice of weights \eqref{eq:op_weights}, $\calD_s(\lambda)$ can be bounded independently of~$s$. Equation \eqref{eq:MLKI_final_error} follows by applying the bound on $\calD_s(\lambda)$ to Theorem~\ref{thm:ML_full_error}.
\end{proof}

\begin{remark}\label{rem:double}
The recent paper \cite{KS24} demonstrates that it is possible to achieve double the convergence rate of the $L^2$ kernel approximation error for functions that are sufficiently smooth, by leveraging the orthogonal projection property of the kernel interpolant, using a method reminiscent of the Aubin-Nitsche trick.
The theory from \cite{KS24} can be directly applied to our theoretical results, i.e., one should expect a convergence rate of $\frac{1}{2\lambda}$ rather than the $\frac{1}{4\lambda}$ as implied by \eqref{eq:lattice_KI_error}, since the analytic nature of the solution $u(\bsx,\bsy)$ of \eqref{eq:variational_pde} with respect to $\bsy$ implies that $u(\bsx,\cdot) \in \calH_{\alpha,\bsgamma}(\Omega_s)$ for all $\alpha\geq1$.
However, if we impose restrictions to ensure dimension independence of the constant $\calD_s(\lambda)$, then we cannot simply double the convergence rate in \eqref{eq:MLKI_final_error}.
Instead, the convergence rate for the QMC error determined in Theorem~\ref{thm:weight_select} will remain unmodified, while the choices of $\alpha$ and 
$\lambda$ will change. We now choose 
$\alpha = \lfloor \frac{1}{2q}+\frac{3}{4}\rfloor$ and $\lambda = \frac{2q}{2-q}$, which
 means that $\alpha$ could now be selected to be smaller than before whilst maintaining the 
 same rate of convergence and dimension independence of the constant.
\end{remark}

The main result for total error of approximation using the multilevel kernel approximation method is presented below. 
\begin{theorem}\label{thm:full_result}
Under Assumptions \eqref{asm:a_in_Linf}--\eqref{asm:f} and 
with $\alpha$, $\lambda$ and weights $\gamma_\setu$ chosen as in Theorem~\ref{thm:weight_select}, the total error of the multilevel kernel approximation of $u$ satisfies   
\begin{align*}
\|u-I^{\rm{ML}}_L u\|_{L^2(\Omega \times D)}
&\lesssim 
 \|f\|_{L^2(D)}\bigg(s^{-(\frac{1}{p}-\frac{1}{2})} + h_L^2+
	\sum_{\ell=0}^L [\varphi(N_\ell)]^{-(\frac{1}{2q} - \frac{1}{4})} h_{\ell-1}^2  \bigg),
\end{align*}
where $h_{-1}\coloneqq 1$, and the implied constant is independent of the dimension $s$.
\end{theorem} 
\begin{proof}
The result is simply obtained by combining the dimension truncation and FE bounds given in  \eqref{eq:DT_KKKNS} and \eqref{eq:FE_KKKNS}, respectively, with the multilevel kernel approximation error \eqref{eq:MLKI_final_error} to bound the total approximation error decomposition given in \eqref{eq:total_err}. 
\end{proof}

Theorem~\ref{thm:full_result} shows that Assumptions~\eqref{asm:trun_error} and~\eqref{asm:ml_error} of Theorem~\ref{thm:ml-complexity} hold with $\kappa = \frac{1}{p}-\frac{1}{2}$, $\mu = \frac{1}{2q}-\frac{1}{4}$ and $\beta=2$.

\begin{remark} \label{rem:sdpt_wgt}
CBC construction with weights of the form \eqref{eq:op_weights} is too costly to implement. Following~\cite{KKKNS22}, it can be shown that Theorem~\ref{thm:full_result} continues to hold for the  SPOD weights
\begin{align*}
	\gamma_\setu \coloneqq \sum_{\bsm_\setu \in \{1:\alpha\}^{|\setu|}}
	[(|\bsm_\setu|+5)!]^{\frac{2}{1+\lambda}}
	\prod_{i\in\setu} \bigg(\frac{\overline{b}_i^{m_i}\,S(\alpha,m_i)}{\sqrt{2e^{1/e}\zeta(2\alpha\lambda)}}\bigg)^{\frac{2}{1+\lambda}}.
\end{align*}
However, these SPOD weights still come with a quadratic cost in dimension when evaluating the kernel basis functions to construct the approximation compared to the linear cost of product weights. Instead, the paper \cite{KKS22} proposes \emph{serendipitous product weights}, whereby the order-dependent part of the weights is simply omitted, giving in our case
\begin{align}\label{eq:sdpts_wgt}
	\gamma_\setu \coloneqq \prod_{i\in\setu}\Bigg(\sum_{m=1}^\alpha
	\frac{\overline b_i^m S(\alpha,m)}{\sqrt{2e^{1/e}\zeta(2\alpha\lambda)}}\Bigg)^{\frac{2}{1+\lambda}}.
\end{align}
For these weights, the upper bound on the error of the kernel approximation will continue to have the same theoretical rates of convergence seen in Theorem~\ref{thm:full_result}, although the implied constant will no longer be independent of $s$. 

The paper \cite{KKS22} demonstrates that serendipitous product weights offer comparable performance to SPOD weights for easier problems and superior performance in problems where SPOD weights may fail completely. A possible explanation is that the SPOD weights derived from theory may be poorer due to overestimates in the bounds. A second, more practical explanation is that the magnitude of SPOD weights substantially increases as dimension grows which leads to larger, more 
peaked kernel basis functions. The resulting approximations using these basis functions 
tend to also be very peaked, resulting in a less ``smooth" approximation.
It is worth noting that integration problems are more robust to overestimates that may 
affect the quality of weights, as the weights do not directly appear in the computation of 
the integral. This is not true for the kernel approximation problem where the weights appear 
explicitly in the kernel basis functions.
\end{remark}

\section{Implementing the ML kernel approximation}
\label{sec:implementation}
In this section, we outline how to compute efficiently the multilevel approximation for the PDE problem using a matrix-vector representation of \eqref{eq:ml-ker}. Throughout we consider a fixed truncation dimension and
so omit the superscript $s$.

First, we outline how to construct the single-level kernel interpolant for the PDE problem.
As in~\cite{KKKNS22}, the single-level kernel interpolant applied to the FE approximation of the PDE solution is given by \eqref{eq:interp-g} with $g = u_h(\bsx, \cdot)$, for $\bsx \in D$.
Each kernel interpolant coefficient $a_{N,h, k} \in V_h$ is now a FE function, which we 
can expand using the FE basis functions $\{\phi_{h, i}\}_{i = 1}^M$, with $M = \dim(V_h)$, to write
\begin{align}
\label{eq:interp-fe}
I_N u_h(\bsx, \bsy)
\,=\,  \sum_{k = 0}^{N - 1} a_{N, h, k}(\bsx)\, \calK_{\alpha,\bsgamma}(\bst_k, \bsy)
\,=\,  \sum_{i = 1}^M
\sum_{k = 0}^{N - 1} a_{N, h, k, i} \, \calK_{\alpha,\bsgamma}(\bst_k, \bsy)\,\phi_{h, i}(\bsx).
\end{align}
To enforce interpolation for the FE solution, we equate the coefficients of the FE functions $\phi_{h, i}$
in \eqref{eq:fe-expansion} and \eqref{eq:interp-fe} at each lattice point $\bsy = \bst_{k'}$,
leading to the requirement
\begin{align*}
	u_{h, i}(\bst_{k'})=   \sum_{k = 0}^{N - 1} a_{N, h, k, i} \, \calK_{\alpha,\bsgamma}(\bst_k, \bst_{k'})
        \quad \text{for all } k' = 0, 1, \ldots, N - 1,\; i = 1, 2, \ldots, M.
\end{align*}
Thus, the coefficients
$\bsA_{N, h} \coloneqq [a_{N, h, k, i}]_{0\le k\le N-1,\, 1\le i\le M} \in \R^{N \times M}$
are the solution to the matrix equation
\begin{equation*}
K_{N, \alpha,\bsgamma} \,\bsA_{N, h} \,=\, \bsU_{\!N, h},
\end{equation*}
where $K_{N, \alpha,\bsgamma} = [\calK_{\alpha,\bsgamma}(\bst_k, \bst_{k'})]_{k, k' = 0}^{N - 1}$, and
$\bsU_{\!N, h} = [u_{h_, i}(\bst_k)]_{0\le k\le N-1,\, 1\le i\le M}$ is an $N \times M$ matrix of the FE nodal values at the lattice points.

Next we present a similar matrix-vector representation for the multilevel kernel approximation \eqref{eq:ml-ker}.
Since we are using an embedded lattice rule, $N_{\ell}$ is a factor of $N_{\ell-1}$
and the point sets across the levels are embedded, with an ordering such that
\begin{align}\label{eq:nest_pts}
\bst_{\ell,k}=\bst_{\ell-1,\,k(N_{\ell-1}/N_{\ell})}
\quad \text{for } \ell =  1, 2, \ldots, L.
\end{align}
It follows that $\bst_{\ell,k}=\bst_{0,\,k(N_{0}/N_{\ell})}$ and the kernel matrices 
for each level $\ell$
are nested with the same structure.

Since the FE spaces are also nested, we can expand the difference on each level $\ell \geq 1$ in the FE basis for $V_\ell= V_{h_\ell}$ as
\begin{equation}
\label{eq:diff-fe-expansion}
u_\ell(\cdot, \bsy) -  u_{\ell - 1}(\cdot, \bsy)
= \sum_{i = 1}^{M_\ell} \big[u_{\ell, i}(\bsy) - \overline{u}_{\ell - 1, i}(\bsy)\big] \phi_{h_\ell, i}
\in V_\ell,
\end{equation}
where $M_\ell \coloneqq \dim(V_\ell)$ and $[\overline{u}_{\ell - 1, i}(\bsy)]_{i = 1}^{M_\ell}$
are the FE coefficients for $u_{\ell - 1}(\cdot, \bsy)$ after (exact) interpolation
onto $V_\ell$.

Similar to \eqref{eq:interp-fe}, the kernel approximation on level $\ell$ is
\begin{equation}
\label{eq:I_ell-diff}
I_\ell (u_\ell - u_{\ell - 1})(\bsx, \bsy) = \sum_{i = 1}^{M_\ell} \sum_{k = 0}^{N_\ell - 1} 
a_{\ell, k, i} \, \calK_{\alpha,\bsgamma}(\bst_{\ell, k}, \bsy)%
\,\phi_{h_\ell, i}(\bsx),
\end{equation}
and we now choose the coefficients $a_{\ell, k, i} = a_{N_\ell,h_\ell,k,i}$ such that
\eqref{eq:I_ell-diff} interpolates the difference  \eqref{eq:diff-fe-expansion} at each lattice point $\bsy = \bst_{\ell, k}$.
This leads to $L + 1$ matrix equations
\begin{align}
\label{eq:K*A=U-ell}
K_{N_\ell, \alpha, \bsgamma}\,\bsA_{\ell} \,=\, 
\begin{cases}
{\bsU}_{\!0} &\mbox{for }\ell=0, \\
\bsU_{\!\ell} - \overline\bsU_{\!\ell-1}  &\mbox{for }\ell=1,\ldots,L,
\end{cases}
\end{align}
where now each coefficient matrix is $\bsA_{\ell} \coloneqq \bsA_{N_\ell,h_\ell} = [a_{\ell, k, i}] \in \R^{N_\ell \times M_\ell}$,
$\bsU_{\!\ell} \coloneqq \bsU_{\!N_\ell,h_\ell}  = [u_{\ell, i}(\bst_{\ell, k})] \in \R^{N_\ell \times M_\ell}$ is the matrix of FE coefficients at the lattice points for level $\ell$,
and ${\overline\bsU}_{\!\ell-1} = [\overline{u}_{\ell - 1, i}(\bst_{\ell, k})] \in \R^{N_\ell \times M_\ell}$ is the matrix of FE
coefficients for $u_{\ell - 1}(\cdot, \bst_{\ell, k}) \in V_{\ell - 1}$ interpolated onto $V_\ell$ then evaluated at the lattice points for level $\ell$. 

To obtain each coefficient matrix $\bsA_\ell$,
we exploit the circulant structure of kernel matrices to solve the linear systems 
\eqref{eq:K*A=U-ell}
using FFT. As a result, we only require the first column of each $K_{N_\ell, \alpha, \bsgamma}$
and furthermore, since the lattice points are nested as in \eqref{eq:nest_pts}, we obtain this column by taking every $(N_0/N_\ell)$th  entry (starting from the first) of the first column of $K_{N_0, \alpha, \bsgamma}$.
Similarly, since the FE spaces are nested we obtain $\overline{\bsU}_{\ell - 1}$ for $\ell \geq 1$ 
using the FE matrix $\bsU_{\ell - 1}$ from the previous level. Specifically, we take every $(N_{\ell - 1}/N_\ell)$th row of $\bsU_{\ell - 1}$
then interpolate it onto $V_\ell$ to give $\overline{\bsU}_{\ell - 1}$.
This procedure is given in Algorithm~\ref{alg:ML-coeff}.

\begin{algorithm}[!h]
\caption{Constructing the ML approximation}
\label{alg:ML-coeff}
\begin{algorithmic}[1]
\State Compute 1st column of $K_{N_0,\alpha,\bsgamma}$ 
\Comment kernel matrix for entire lattice point set
\State Compute $\bsU_{\!0}$
\Comment FE solutions in $V_0$  at all lattice points
\State Solve $K_{N_0,\alpha,\bsgamma}\,\bsA_0 = \bsU_{\!0}$ using FFT
\Comment solve for coefficients
\For{$\ell=1,\ldots,L$}
\State Compute 1st column of $K_{N_\ell,\alpha,\bsgamma}$
\Comment $(N_0/N_\ell)$th entries of column 1 of $K_{N_0,\alpha,\bsgamma}$ 
\State Compute $\bsU_\ell$
\Comment FE solutions in $V_\ell$ at $N_\ell$ lattice points
\State Interpolate $(N_{\ell - 1}/N_\ell)$th rows of $\bsU_{\ell - 1}$ onto $V_\ell$ to obtain $\overline{U}_{\!\ell - 1}$
\State Solve $K_{N_\ell,\alpha,\bsgamma}\,\bsA_\ell = {\bsU}_{\!\ell} - {\overline\bsU}_{\!\ell-1}$ using FFT
\Comment solve for coefficients
\EndFor
\end{algorithmic}
\end{algorithm}
	
On completion of Algorithm~\ref{alg:ML-coeff}, we have a collection of 
coefficient matrices 
$\{\bsA_\ell\}_{\ell=0}^L$ that we can use to compute 
$I_L^\mathrm{ML} u(\bsx^*, \bsy^*)$, using \eqref{eq:ml-ker} and \eqref{eq:I_ell-diff},
to approximate the PDE solution at any
$(\bsx^*,\bsy^*) \in D\times\Omega_s$. 
By using the embedded property of the lattice points and the local support
of the FE basis, we can evaluate $I_L^\mathrm{ML} u(\bsx^*, \bsy^*)$
efficiently with cost $\calO(N_0)$ as we explain below.

We can write \eqref{eq:I_ell-diff} as
\begin{align*}
I_\ell(u_\ell - u_{\ell - 1}) (\bsx^*, \bsy^*)
\,&=\,
\sum_{k = 0}^{N_\ell - 1} \Bigg(\underbrace{\sum_{\substack{i = 1\\\bsx^* \in \supp(\phi_{h_\ell, i})}}^{M_\ell}
a_{\ell, k, i}\, \phi_{h_\ell, i}(\bsx^*)}_{[\bsa_\ell(\bsx^*)]_{k}}\Bigg) \calK_{\alpha,\bsgamma}(\bst_{\ell, k}, \bsy^*)
\\
\,&=\, \sum_{k' = 0}^{N_0 - 1} 
\bbI(k' N_\ell \equiv 0 \,({\rm mod}\,N_0)) \, [\bsa_\ell(\bsx^*)]_{\lfloor k'N_\ell / N_0 \rfloor}\, \calK_{\alpha,\bsgamma}(\bst_{0, k'}, \bsy^*).
\end{align*}
Effectively, for each $\ell$, we compute the sum over $i$ in \eqref{eq:I_ell-diff} by ``interpolating in space''
along the columns of $\bsA_\ell$ to get a kernel coefficient vector at $\bsx^*$, denoted
$\bsa_{\ell}(\bsx^*) \in \R^{N_\ell}$.
Since the FE basis functions are locally supported,
there will be a small number of indices $i$ in \eqref{eq:I_ell-diff}
such that $\phi_{h_\ell, i}(\bsx^*)$ is nonzero. Multiplying the corresponding columns $\bsA_\ell$
by $\phi_{h_\ell, i}(\bsx^*)$ and summing them up then gives
$\bsa_\ell(\bsx^*)$ at a cost of $\calO(N_\ell)$.
In the second equality above we used the embedded property of the lattice points \eqref{eq:nest_pts} to write this as a sum over the entire lattice point set. The indicator function $\bbI$ accounts for the extra terms that were added.

Summing over $\ell = 0 , 1, \ldots, L$ gives a single kernel coefficient vector
$\bsa(\bsx^*) \in \R^{N_0}$,
\[
[\bsa(\bsx^*)]_{k'}
\,=\, [\bsa_{0}(\bsx^*)]_{k'} + \sum_{\ell=1}^L 
\bbI(k' N_\ell \equiv 0 \,({\rm mod}\,N_0)) \, [\bsa_\ell(\bsx^*)]_{\lfloor k'N_\ell/N_0 \rfloor},
\] 
at a cost $\calO(N_0)$. Since the points are embedded $\sum_{\ell = 0}^L N_L = \calO(N_0)$
and hence, the total cost to construct the kernel coefficient vector $\bsa(\bsx^*)$ is $\calO(N_0)$.

Similarly, we store all of the kernel evaluations at $\bsy$ anchored at each of the lattice points
in a single vector $\bskappa(\bsy^*) \coloneqq [\calK_{\alpha,\bsgamma}(\bst_k,\bsy^*)]_{k=0}^{N_0-1}$,
which costs $\calO(s^\rho \alpha^\varsigma N_0)$ where $\rho,\varsigma >0$ depend on the structure of the kernel and are specified in Section~\ref{sec:SLKI}.

Finally, the ML kernel approximation of $u(\bsx^*, \bsy^*)$ is computed by the product
\[
u(\bsx^*,\bsy^*) \,\approx\,
I_{L}^\mathrm{ML} u(\bsx^*, \bsy^*) 
\,=\,
\bsa(\bsx^*)^\top \bskappa(\bsy^*).
\]
Combining the costs above leads to a total cost of evaluating $I_{L}^\mathrm{ML} u(\bsx^*, \bsy^*)$ of $\calO(s^\rho \alpha^\varsigma N_0)$.

In the case of non-nested FE spaces, equation \eqref{eq:K*A=U-ell} for level $\ell>0$ becomes
\[
K_{N_\ell,\alpha,\gamma} \bsA_\ell = \overline{\bsU}_\ell - \overline{\bsU}_{\ell-1}, 
\]
where $\overline{\bsU}_\ell - \overline{\bsU}_{\ell-1}$ is the vector of differences $u_\ell - u_{\ell-1}$ evaluated at the $S_\ell$ nodes of the supermesh associated with the FE meshes for levels $\ell$ and $\ell-1$. In \eqref{eq:I_ell-diff}, the mesh sum over index $i$ now goes up to $S_\ell$, 
with the basis functions $\phi_{h_\ell,i}$ being replaced by the basis functions for the supermesh.

\graphicspath{{./plots/}} 

\section{Numerical experiments}\label{sec:numerical}

In this section, we present the results of numerical experiments conducted using the high performance 
computational cluster Katana \cite{Katana}. 

\subsection{Problem specification}

\ourparagraph{The parametric PDE.}
We consider the spatial domain $D=(0,1)^2$ with source term $f(\bsx) = x_2$.
For the parametric coefficient defined in \eqref{eq:coeff}, we choose $\psi_0\equiv 1$ and 
	\begin{align*}
	\psi_j(\bsx) = \frac{C}{\sqrt6}j^{-\vartheta}\sin(j\pi x_1)\sin(j\pi x_2) 
	\qquad \mbox{for}\quad j\ge 1,
	\end{align*}
where $C>0$ is a scaling parameter to vary the magnitude of the random coefficient, 
and $\vartheta>1$ is a decay parameter determining how quickly the importance of each 
random parameter decreases. The approximation problem becomes more challenging as $C$ increases or 
$\vartheta$ decreases.
The factor $\frac{1}{\sqrt 6}$ is included for easier comparison to the experiments in \cite{KKKNS22} and \cite{KKS22}. For this choice, the bounds on the coefficient assumed in~\eqref{asm:a_bounds} are given by $a_{\min} \coloneqq\! 1 - \frac{C}{\sqrt6}\zeta(\vartheta)$ 
and $a_{\max} \coloneqq\! 1 + \frac{C}{\sqrt6}\zeta(\vartheta)$. Thus, from~\eqref{eq:bb_bar} we have	
\begin{align*}
		b_j = \frac{Cj^{-\vartheta}}{a_{\min}\sqrt6} \qquad \mbox{and}
		\qquad \overline{b}_j = \frac{C\pi j^{1-\vartheta}}{a_{\min}\sqrt6} 
		\qquad \mbox{for}\quad j\geq 1.
\end{align*}
Assumptions~\eqref{asm:p_summ} and~\eqref{asm:bbar_sum} hold provided that $1>p>\frac{1}{\vartheta}$ and $1>q>\frac{1}{\vartheta-1}$. Having chosen $p>\frac{1}{\vartheta}$, we can set $q = \frac{p}{1-p}$.

We will present numerical results for two choices of parameters: $C=1.5,\vartheta = 3.6$ and $C=0.2,\vartheta = 1.2$, with the first being easier than the second, and consider the truncated parametric domain $[0,1]^{64}$, i.e., $s=64$. These parameters are chosen to match the single-level experiments presented in \cite{KKKNS22} and \cite{KKS22}. We also experimented with many other pairings of $C$ and $\vartheta$, and have observed consistent results as below.

\ourparagraph{Smoothness, weights, and CBC construction.} 
The smoothness parameter of the kernel is fixed as $\alpha = 1$. We refrain from using higher values of $\alpha$ to avoid stability issues in inaccurate eigenvalue computation via FFT in double precision. Issues in double precision computations arise even for small values of $\alpha$ and relatively small $N$, necessitating the use of arbitrary-precision computing (see e.g., \cite{KMNSS22}). 

Following Remark~\ref{rem:sdpt_wgt}, we use serendipitous weights 
\eqref{eq:sdpts_wgt} for our experiments. We take $\alpha=1$ and $\lambda = \frac{1}{2\alpha}+0.1 = 0.6$ in \eqref{eq:sdpts_wgt}, and use the CBC algorithm described in \cite{KMN23} to construct one embedded lattice rule designed for each parameter set to be used in the range $2^{7} \le N\le 2^{17}$ up to dimension $s=64$.

\ourparagraph{Convergence and cost.} 
According to \eqref{eq:KI_KKKNS} and Theorem~\ref{thm:ML_full_error}, and ignoring the dependence of the constants on $s$ (which is fixed here), we have theoretically $\calO(N_\ell^{-\mu})$ convergence in Assumption~\eqref{asm:ml_error} of Theorem~\ref{thm:ml-complexity} with
$\mu = 1/(4\lambda) \approx 0.417$. As explained in Remark~\ref{rem:double}, since the PDE solution has a higher smoothness order than $\alpha=1$, the theoretical convergence rate doubles to $\mu = 1/(2\lambda) \approx 0.833$.

Recall that the computational cost to achieve an error $\varepsilon>0$ for the single-level approximation
is \eqref{eq:SLcostFinal}, and for the multilevel approximation it is \eqref{eq:complexity}.
Our spatial domain is the unit square, so $d=2$.
The FE convergence rate in Assumption~\eqref{asm:ml_error} of Theorem~\ref{thm:ml-complexity} is determined by Theorem~\ref{thm:ML_full_error} to be $\beta=2$, which also applies for the single-level FE error as indicated by \eqref{eq:FE_KKKNS}. 
Since the serendipitous weights \eqref{eq:sdpts_wgt} are product weights, we have $\rho= 1$. 
With a fixed dimension $s=64$, we may informally set $\kappa = \infty$ and $\tau\approx d$ in Theorem~\ref{thm:ml-complexity}. The third case in \eqref{eq:complexity} is not relevant when using product weights. Hence, we obtain the theoretical cost bounds
\begin{align}\label{eq:num_cost}
\begin{cases}
 \mbox{cost}(I_Nu_h^s)  
 \lesssim \varepsilon^{-\frac{1}{\mu}-1} 
 = \varepsilon^{-2\lambda-1} = \varepsilon^{-2.2}, 
 \\
 \mbox{cost}(I^{\text{ML}}_Lu) 
 \lesssim \varepsilon^{-\max(\frac{1}{\mu},1)}
 = \varepsilon^{-\max(2\lambda,1)}
 = \varepsilon^{-1.2}.
\end{cases}
\end{align}

\subsection{Diagnostic plots}

\ourparagraph{FE.} 
Figure~\ref{fig:FE_plots} plots the computational cost, measured by CPU time in seconds, of assembling and solving the FE linear system required to obtain the PDE solution, and the FE error 
$\| u^{s}_{h^*} - u^{s}_{h}\|_{L^2(D\times\Omega)} \approx 
 \sqrt{\frac{1}{N} \sum_{k=0}^{N-1} \| u^{s}_{h^*}(\cdot,\bst_k) - u^{s}_{h}(\cdot,\bst_k)\|_{L^2(D)}^2}$
for mesh widths $h\in \{2^{-3},\ldots, 2^{-8}\}$ with respect to the reference solution with mesh width $h^* = 2^{-9}$ and $s=64$. We use $N=2^{16}$ lattice points for the integral over $L^2(\Omega)$, and the $L^2(D)$ norm is computed exactly from the FE coefficients.
We observe an error convergence rate of $\calO(h^2)$ and a cost of $\calO(h^{-2})$, demonstrating that indeed $\beta = 2$ and $\tau \approx d = 2$ in Assumptions~\eqref{asm:trun_error} and~\eqref{asm:cost} of Theorem~\ref{thm:ml-complexity}.

\begin{figure}[t]
\begin{center}
 \includegraphics[width=0.48\textwidth]{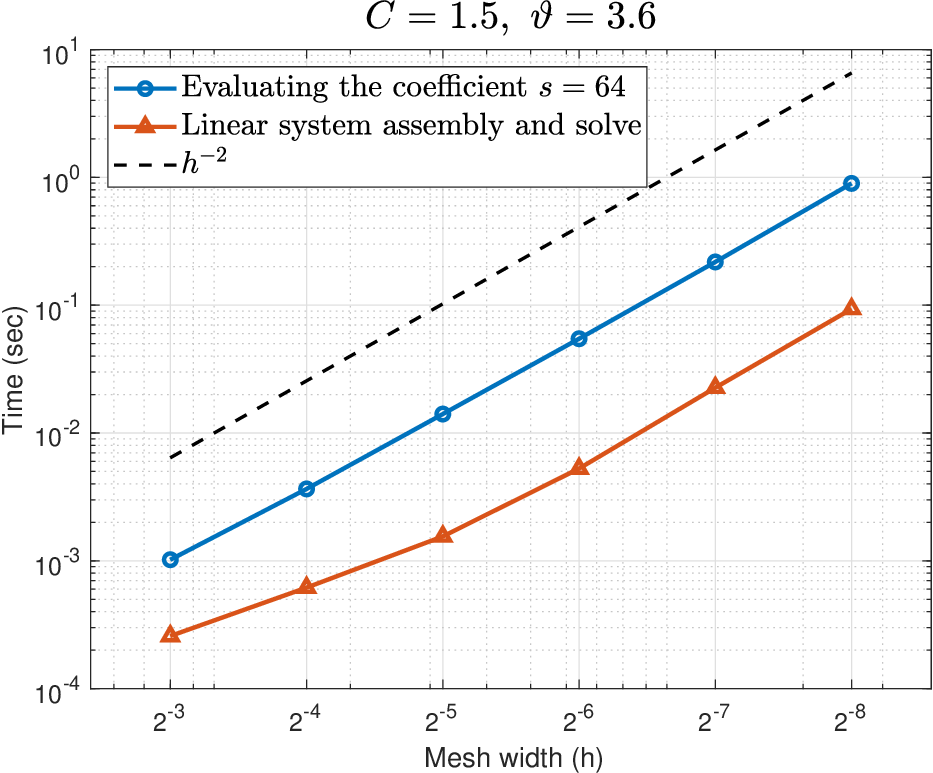} \quad
 \includegraphics[width=0.48\textwidth]{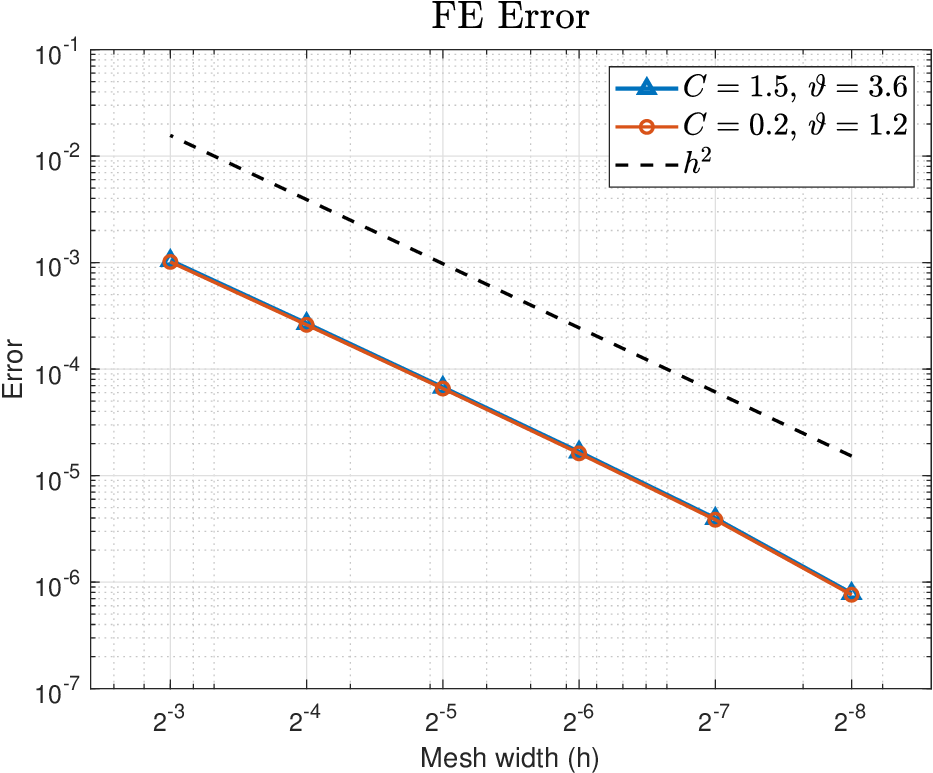} \ 
 \caption{Cost for the FE solve on the left. The FE error $\|u^s_{h^*}-u^s_h\|_{L^2(D\times\Omega)}$ on the right. They demonstrate $\beta = 2$ and $\tau\approx d = 2$ in 
 Assumptions~\eqref{asm:trun_error} and~\eqref{asm:cost} of Theorem~\ref{thm:ml-complexity}.}\label{fig:FE_plots}
\end{center}
\end{figure}

\begin{figure}[t]
\begin{center}
 \includegraphics[width=0.48\textwidth]{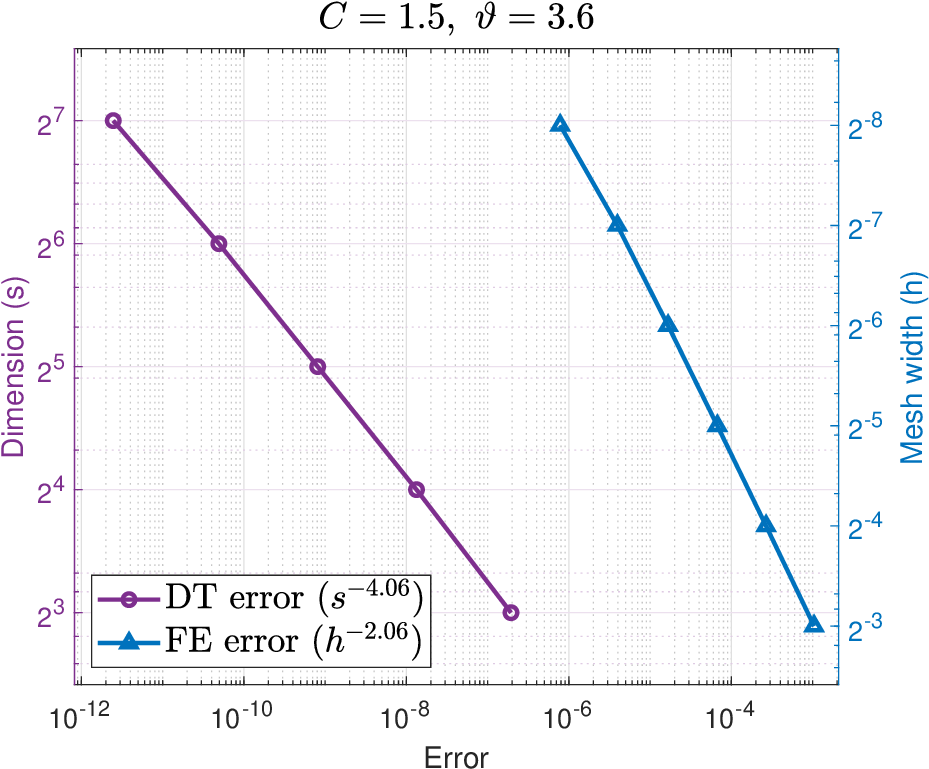}\quad
  \includegraphics[width=0.48\textwidth]{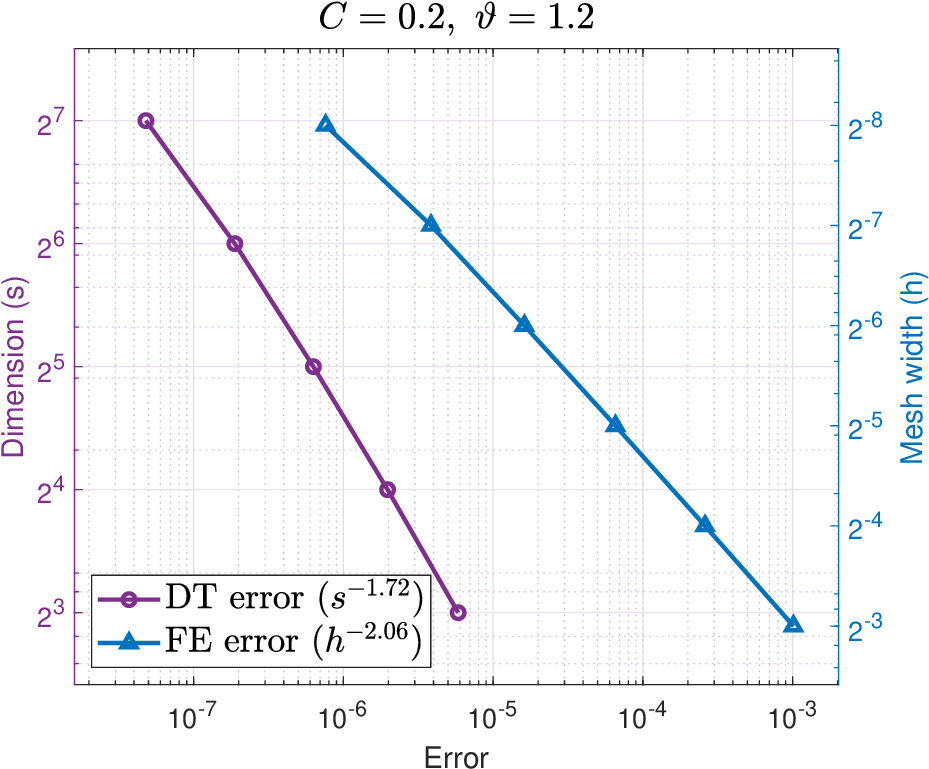}
\caption{Comparing dimension truncation error  with FE error.}\label{fig:dim}
\end{center}
\end{figure}

\ourparagraph{Dimension truncation.}
For $s \in \{2^3,\ldots,2^7\}$ and reference values $s^* = 2^8$ and $h =2^9$, we compute the error 
$\| u^{s^*}_{h} - u^s_{h}\|_{L^2(D\times\Omega)} \approx 
\sqrt{\frac{1}{N} \sum_{k=0}^{N-1} \| u^{s^*}_{h}(\cdot,\bst_k) - u^s_{h}(\cdot,\bst_k)\|_{L^2(D)}^2}$
using $N = 2^{16}$ lattice points. In Figure~\ref{fig:dim}, on the horizontal axis, we overlay the dimension truncation errors with the FE errors for $h\in \{2^{-3},\ldots, 2^{-8}\}$, with the vertical axis on the left in terms of $s$ (purple circles), and on the right in terms of $h$ (blue triangles). (The relative slope between the overlaid plots is arbitrary.)
We observe $\calO(s^{-\kappa})$ convergence (c.f.~Assumption~\eqref{asm:trun_error} of Theorem~\ref{thm:ml-complexity}) with $\kappa = 4.06$ for the easier problem $C=1.5,\vartheta = 3.6$, and $\kappa=1.72$ for the harder problem $C=0.2,\vartheta = 1.2$. 
For both problems with $s = 2^6 = 64$ we obtain dimension truncation errors much less than $10^{-6}$, i.e., smaller than all the FE errors with $h\in \{2^{-3},\ldots, 2^{-8}\}$ in Figure~\ref{fig:FE_plots}.

\ourparagraph{Single-level kernel interpolation.} Following \cite{KKKNS22}, we use an efficient formula to estimate the single-level kernel interpolation error 
\begin{align*} 
	\| (I-I_N) u^s_{h^*} \|_{L^2(D\times\Omega)} 
	& \approx
	 \sqrt{
	\frac{1}{RN}\sum_{r=1}^R\sum_{k=0}^{N-1}
	\|u^s_{h^*}(\cdot,\bsy_r+\bst_k)- I_N u^s_{h^*}(\cdot,\bsy_r+\bst_k)\|^2_{L^2(D)}
	 }, 
	\end{align*}
where $\{\bsy_r\}_{r=1}^R$ is a sequence of Sobol$'$ points, $\{\bst_k\}_{k=0}^{N-1}$ is a sequence of lattice points, with $h^* = 2^{-9}$, $s=64$, $R=10$, and $N\in \{2^4,\ldots,2^{16}\}$. 
In Figure~\ref{fig:match}, on the horizontal axis, we overlay the interpolation errors with the FE errors for $h\in \{2^{-3},\ldots, 2^{-8}\}$, with the vertical axis on the left in terms of $h$ (blue triangles), and on the right in terms of $N$ (red circles). We observe $\calO(N^{-1.5})$ convergence for the easier problem $C=1.5,\vartheta = 3.6$,
and $\calO(N^{-0.99})$ for the harder problem $C=0.2,\vartheta = 1.2$. 

\begin{figure}[t!]
\begin{center}
 \includegraphics[width = 0.48\textwidth]{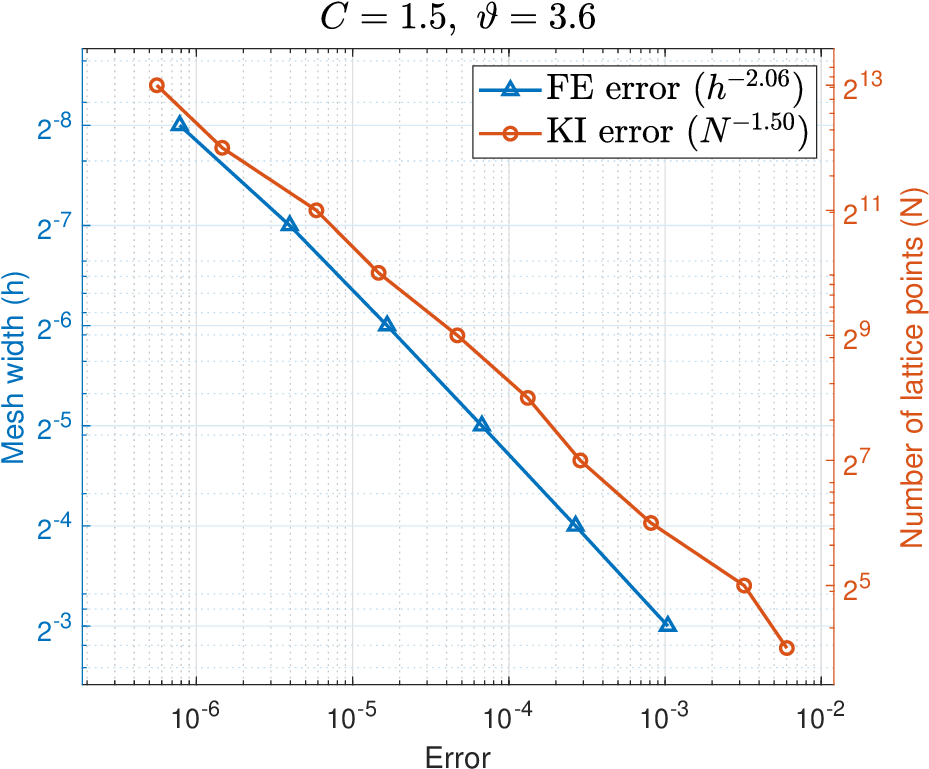} \quad
 \includegraphics[width = 0.48\textwidth]{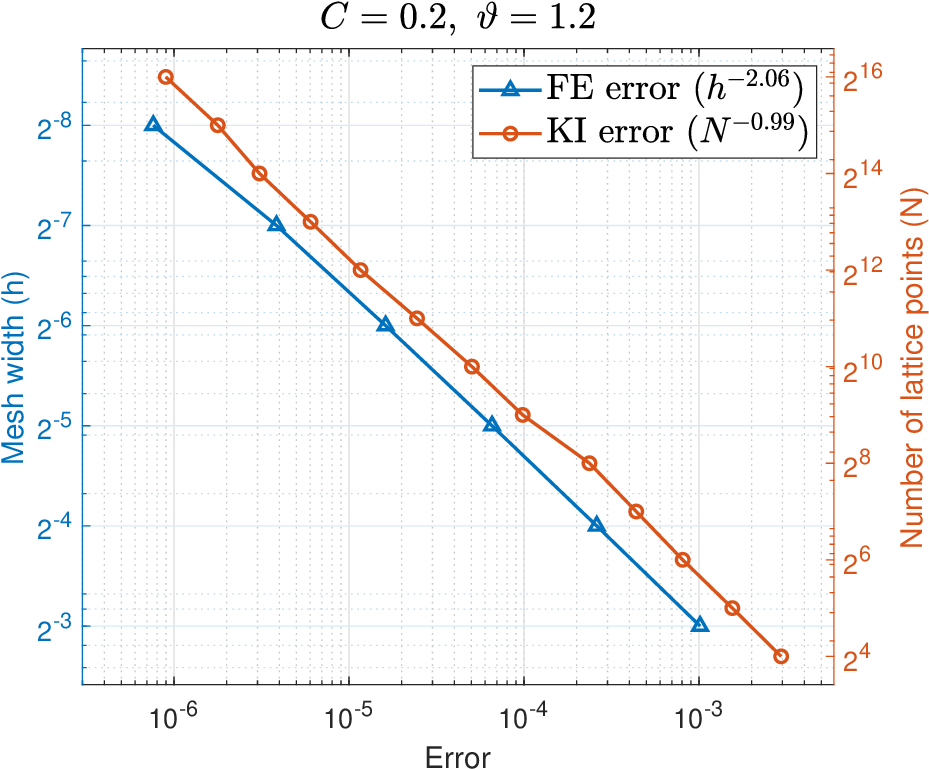}
 \caption{Matching single-level kernel interpolation error and FE error.}\label{fig:match}
\end{center}
\end{figure} 

\renewcommand{\arraystretch}{1.15} 
 \begin{table}[t]
\centering
 \begin{tabular}{|c|cc|cccccc|}
\hline
 &  SL & ${\rm ML}_\ell$      &    0     & 1        &     2    &     3    & 4         &   5      \\
 \hline
 &  $h$ & $h_\ell$ & $2^{-3}$  &  $2^{-4}$ &  $2^{-5}$ &  $2^{-6}$ &  $2^{-7}$  &  $2^{-8}$   \\
 \hline
 $C=1.5,\vartheta=3.6$ & $N$ & $N_{L-\ell}$ &  $2^{6}$ &  $2^{7}$  &  $2^{9}$ &   $2^{10}$  &    $2^{11}$ &$2^{13}$        
 \\
 $C=0.2,\vartheta=1.2$ & $N$ & $N_{L-\ell}$ &   $2^{6}$ &  $2^{8}$  &  $2^{10}$ &   $2^{12}$  & $2^{14}$  & $2^{16}$ 
 \\
 \hline
 \end{tabular}
 \caption{Combinations of FE mesh widths and number of lattice points used for the single-level and multilevel approximations. }
\label{tab:points}
\end{table}

We use Figure 3 to decide how to match the FE mesh width to the number of lattice points so as to give comparable errors. For each $h$ from the FE error data point in Figure~\ref{fig:match}, we trace upward to maintain the same error and look for the nearest interpolation error data point in order to find a matching $N$ with a similar error. (Note that the relative slope between the overlaid plots is arbitrary and does not affect the outcome.) This yields the pairings $(h,N)$ for the single-level kernel interpolant for the two problems in Table~\ref{tab:points}. We will also use Table~\ref{tab:points} to form the pairings
$(h_\ell,N_\ell)$ for our multilevel kernel approximation below. Notice the reverse labeling $N_{L-\ell}$ in Table~\ref{tab:points}, reflecting the fact that the values of $N_\ell$ decrease with increasing $\ell$. For example, when $L=2$, we have $h_0 = 2^{-3}$, $h_1 = 2^{-4}$, $h_2 = 2^{-5}$, and according to Table~\ref{tab:points} for the case $C=1.5, \vartheta=3.6$, we take $N_0 = 2^9$, $N_1 = 2^7$, $N_2 = 2^6$.

\begin{figure}[t!]
\begin{center}
 \includegraphics[width = 0.48\textwidth]{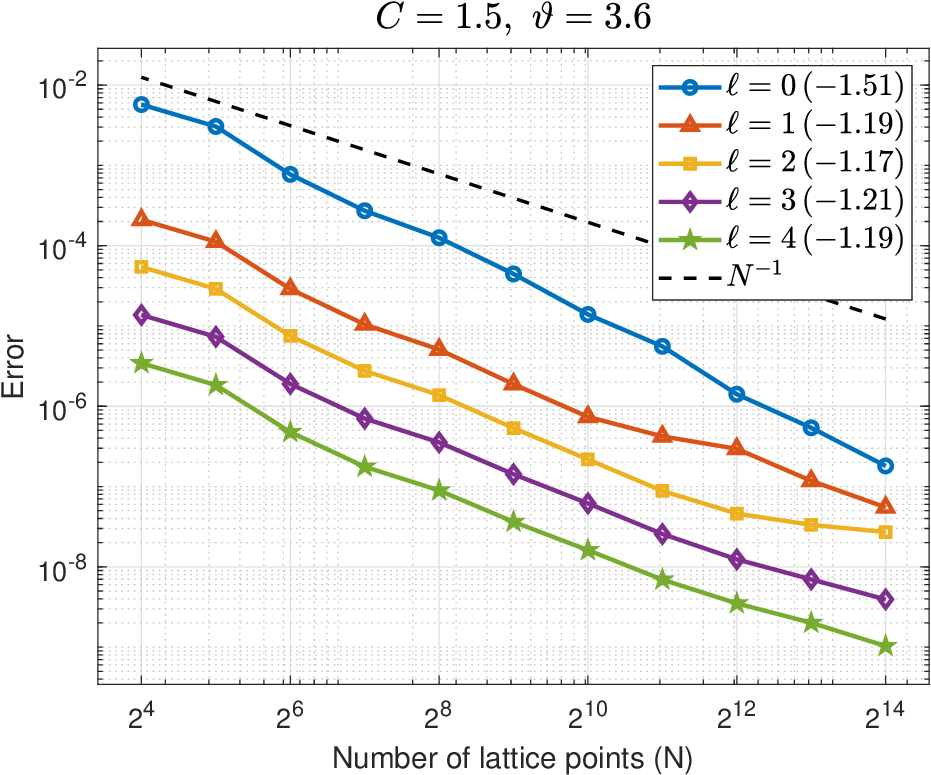} \quad
 \includegraphics[width = 0.48\textwidth]{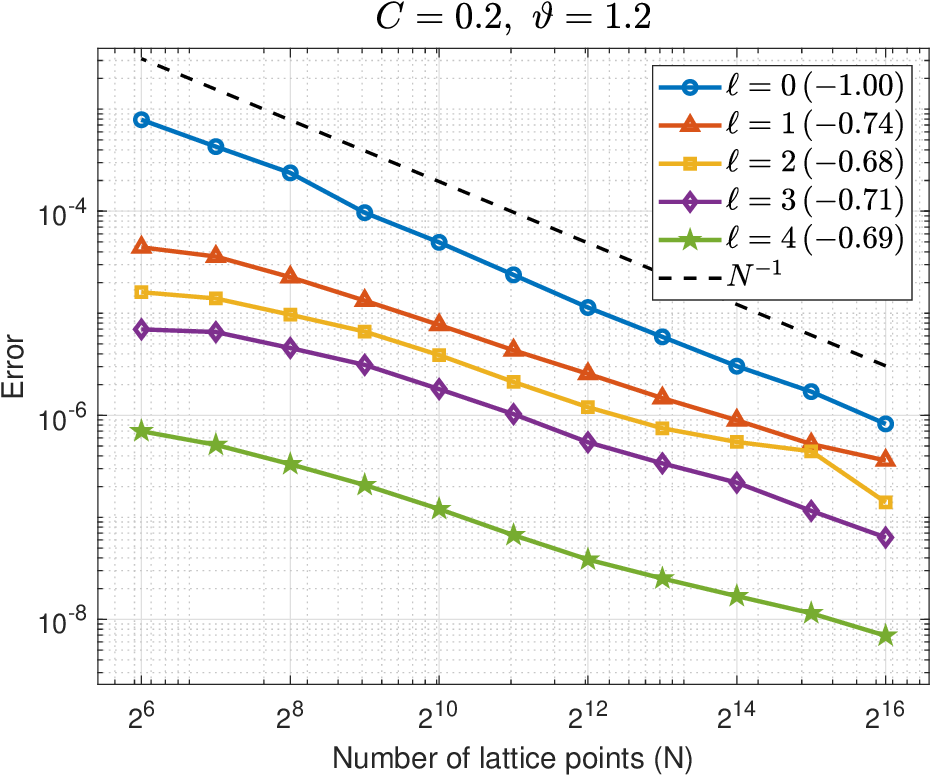} \\
 \vspace{0.5cm}
 \includegraphics[width = 0.48\textwidth]{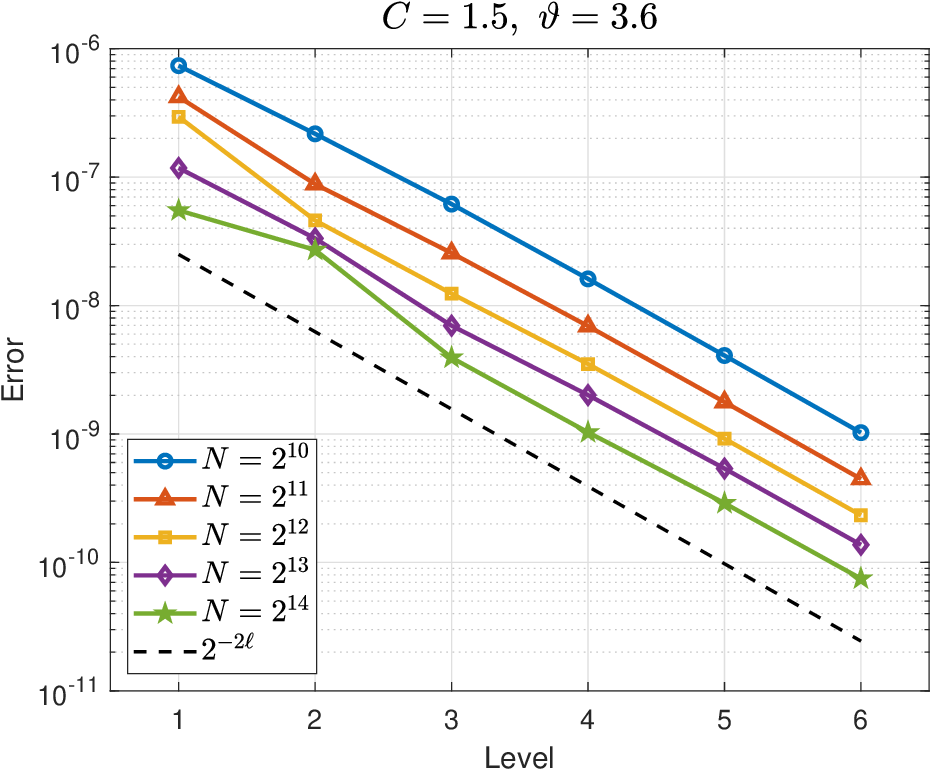} \quad
 \includegraphics[width = 0.48\textwidth]{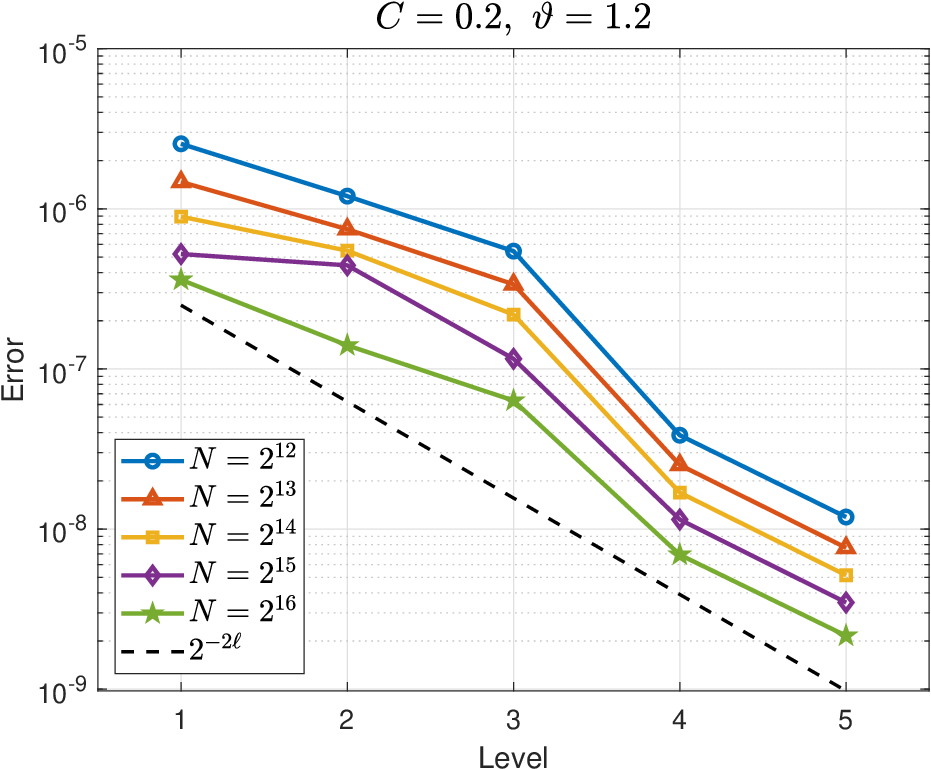}
\caption{Estimates of $\|(I-I_N)(u_\ell - u_{\ell-1})\|_{L^2(D\times\Omega)}$ with 
varying $N$ and $\ell$.}
 \label{fig:level_plot4}
\end{center}
\end{figure}

\ourparagraph{Interpolation error of the FE difference.}
Figure~\ref{fig:level_plot4} plots the estimates of 
\begin{align*} 
 &\|(I-I_N)(u_\ell - u_{\ell-1})\|_{L^2(D\times\Omega)} \nonumber\\
 &\approx
	 \sqrt{
	\frac{1}{RN}\sum_{r=1}^R\sum_{k=0}^{N-1}
	\| (u^s_{h_\ell}-u^s_{h_{\ell-1}})(\cdot,\bsy_r+\bst_k)
	  - I_N (u^s_{h_\ell}-u^s_{h_{\ell-1}})(\cdot,\bsy_r+\bst_k)\|^2_{L^2(D)}
	 }, 
\end{align*}
for both sets of parameters, with $s = 64$, $R=10$, $h_\ell = 2^{-\ell-3}$, and varying $N$ and~$\ell$.

The top two plots in Figure~\ref{fig:level_plot4} show that $\|(I-I_N)(u_\ell - u_{\ell-1})\|_{L^2(D\times\Omega)}$ decreases with increasing $N$ for each $\ell\in \{0,\ldots,4\}$. The case $\ell=0$ is exactly the single-level interpolation error $\|(I-I_N)u^s_{h_0}\|_{L^2(D\times\Omega)}$, and we observe a faster rate of convergence compared to the interpolation error of the FE differences for $\ell\in\{1,\ldots,4\}$. For the easier problem $C = 1.5, \vartheta = 3.6$ on the left, we observe $\calO(N^{-1.51})$ for $\ell = 0$ and $\calO(N^{-1.19})$ on average for the other $\ell$. For the harder problem $C = 0.2, \vartheta = 1.2$ on the right, we observe $\calO(N^{-1.00})$ for $\ell=0$ and $\calO(N^{-0.71})$ on average for the other $\ell$. (The rates for $\ell=0$ are different to those observed in Figure~\ref{fig:match} because we have $h_0 = 2^{-3}$ here while $h^* = 2^{-9}$ in Figure~\ref{fig:match}.) These observed rates correspond to the value of $\mu$ in Assumption~\eqref{asm:ml_error} of Theorem~\ref{thm:ML_full_error}.

The bottom plots in Figure~\ref{fig:level_plot4} show that $\|(I-I_N)(u_\ell - u_{\ell-1})\|_{L^2(D\times\Omega)}$ decreases with increasing level $\ell$ for a range of values of $N$. (Larger values of $N$ are used for the harder problem $C = 0.2, \vartheta = 1.2$ on the right.)
This illustrates how the FE error changes on each level. We observe an error reduction of roughly $\calO(h_\ell^2)$, thus verifying that $\beta = 2$ in Assumption~\eqref{asm:ml_error} of Theorem~\ref{thm:ML_full_error}.

\subsection{Multilevel results}
 
Figure~\ref{fig:MLSL4} plots the computational cost, measured by CPU time in seconds, of the single-level and multilevel approximations against their respective errors for the two parameter sets. The errors are estimated by
\begin{align*} 
	\| (I-A) u^s_{h^*} \|_{L^2(D\times\Omega)} 
	& \approx
	 \sqrt{
	\frac{1}{RN^*}\sum_{r=1}^R\sum_{k=0}^{N^*-1}
	\|u^s_{h^*}(\cdot,\bsy_r+\bst_k)- A u^s_{h^*}(\cdot,\bsy_r+\bst_k)\|^2_{L^2(D)}
	 }, 
\end{align*}
where $Au^s_{h^*}$ denotes either a single-level or multilevel approximation of $u^s_{h^*}$,
with $s = 64$, $R=10$, $h^* = 2^{-9}$, and $N^*$ is the maximum number of lattice points used to construct the approximation (i.e., $N^* = N_0$ for a multilevel approximation).

\begin{figure}[t!]
\begin{center}
 \includegraphics[width = 0.48\textwidth]{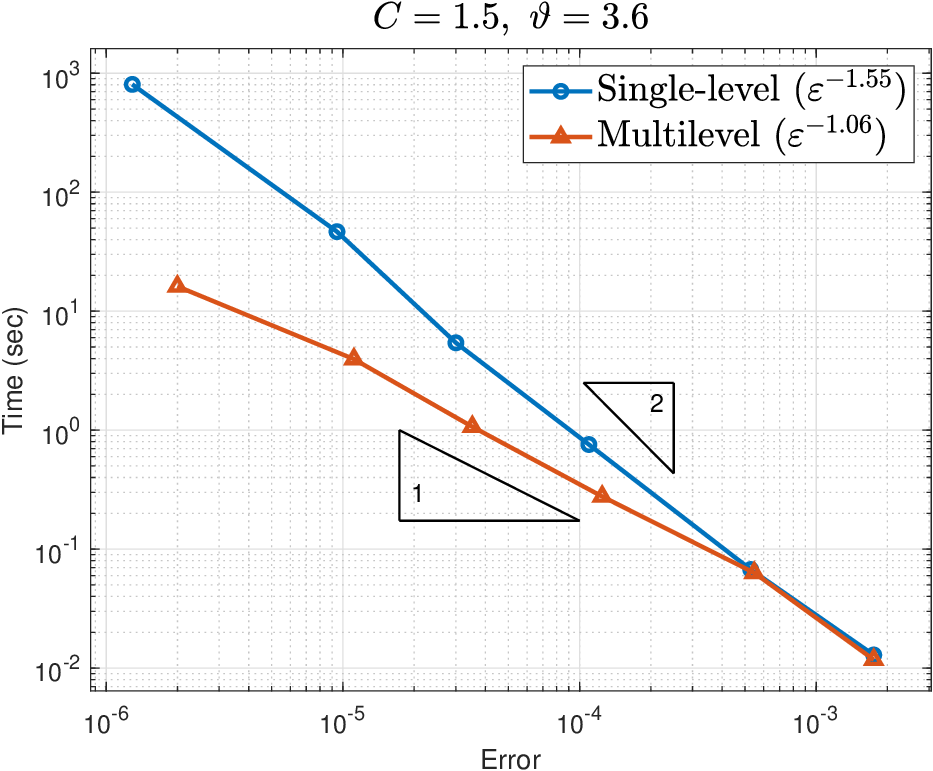} \quad
 \includegraphics[width = 0.48\textwidth]{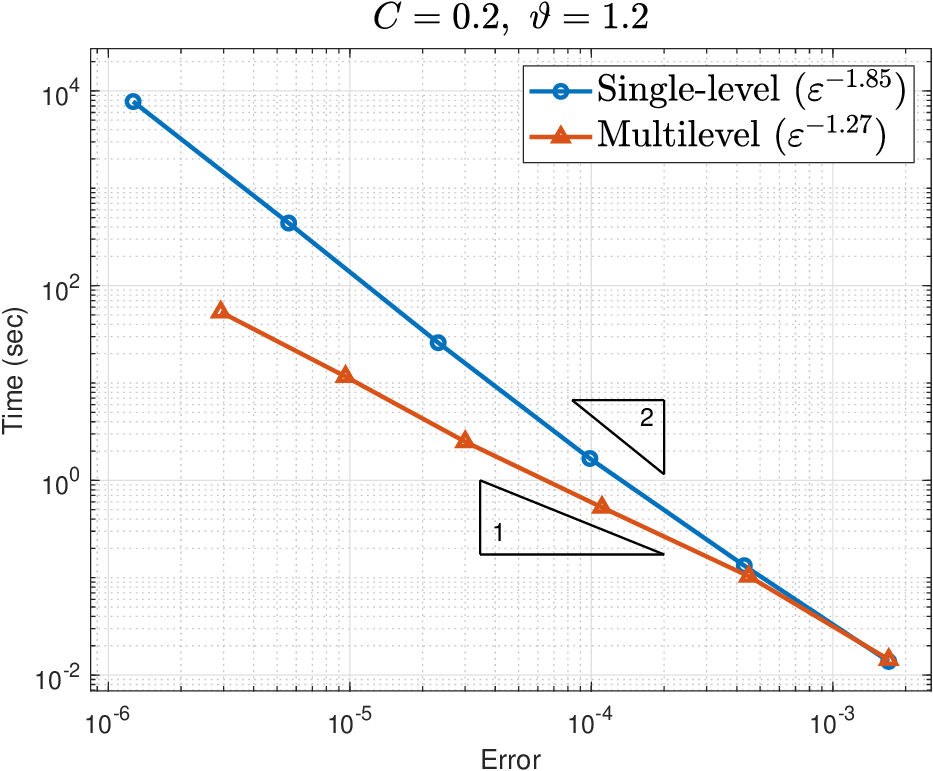}
 \caption{Computation time for the single-level and multilevel approximations against their errors.}\label{fig:MLSL4}
\end{center}
\end{figure} 

Each data point of the single-level result in Figure~\ref{fig:MLSL4}, starting from the least accurate to the most (i.e., right to left on the horizontal axis), corresponds to a decreasing FE mesh width $h\in \{2^{-3},\ldots,2^{-8}\}$ with corresponding increasing number of lattice points $N$ from Table~\ref{tab:points}, while each data point of the multilevel result corresponds to an increasing total number of levels $L\in\{0,\ldots,5\}$ with the corresponding choices for $\{(h_\ell,N_\ell) : \ell=0,\ldots,L\}$ taken from Table~\ref{tab:points}.
The $\ell=0$ data point for the multilevel error coincides with the single-level error as expected.

Recall from \eqref{eq:num_cost} that the theoretical costs for the single-level and multilevel approximations to achieve an error $\varepsilon>0$ are $\calO(\varepsilon^{-2.2})$ and $\calO(\varepsilon^{-1.2})$, respectively. The triangles in the plots show gradients 1 and 2 as denoted.
The observed costs shown in the legends of Figure~\ref{fig:MLSL4} are mostly lower (better) than the theoretical costs. It may be that the theoretical rate of convergence $\mu$ used in estimating the cost in \eqref{eq:num_cost} is overly pessimistic. We can instead estimate the cost using the observed rates for $\mu$ from Figure~\ref{fig:level_plot4}. 

For the easier problem $C=1.5, \vartheta = 3.6$ on the left of Figure~\ref{fig:level_plot4}, we observe $\mu = 1.51$ and $\mu=1.19$ (average) for single-level and multilevel errors, respectively, giving the expected costs $\calO(\varepsilon^{-1.66})$ and $\calO(\varepsilon^{-1})$, respectively. These are closer to the corresponding observed costs $\calO(\varepsilon^{-1.55})$ and $\calO(\varepsilon^{-1.06})$ on the left of Figure~\ref{fig:MLSL4}.

For the harder problem $C=0.2, \vartheta = 1.2$ on the right of Figure~\ref{fig:level_plot4}, we observe $\mu = 1.00$ and $\mu=0.71$ (average) for single-level and multilevel errors, respectively, giving the expected costs $\calO(\varepsilon^{-2})$ and $\calO(\varepsilon^{-1.41})$, respectively. These again are closer to the corresponding observed costs $\calO(\varepsilon^{-1.85})$ and $\calO(\varepsilon^{-1.27})$ on the right of Figure~\ref{fig:MLSL4}.

\section{Conclusion}

This paper introduces a multilevel kernel method for approximating solutions to PDEs with periodic 
coefficients over the parametric domain. A theoretical framework is developed with full detail,
including $L^2$ error estimates for the multilevel method and a comprehensive cost comparison between the
single-level and multilevel approaches. The construction of the multilevel approximation is also outlined and is supported by numerical experiments demonstrating the advantages of the multilevel approximation for different choices of parameters with varying levels of difficulty.

When applying multilevel methods to integration, randomisation can be used to compute an unbiased mean-square error estimate, and the number of levels and/or samples per level can be adaptively adjusted on the run. In contrast, a practical challenge in applying the multilevel kernel approximation is that there is no efficient adaptive multilevel approximation strategy. This is because the error cannot be effectively evaluated during the implementation, since any estimation must be done with respect to some reference solution. This requires either solving the PDE again at several sample points or constructing an expensive single-level proxy of the true solution, defeating the purpose of applying multilevel approximation. We leave investigations into adaptive multilevel kernel approximation for future research.

\bibliographystyle{plain}

{
	
}	
	
\begin{appendices}

\renewcommand{\thesection}{\Roman{section}}
\renewcommand{\theequation}{\Roman{section}.\arabic{equation}}

\section{Combinatorial identities and proofs}
This appendix provides some important combinatorial results that are required in the proofs of several regularity theorems. 
First, we provide an identity to simplify sums involving Stirling numbers of the second kind \eqref{eq:stirling_def} from \cite{QG15},
	\begin{align}\label{eq:stirling_id}
	\sum_{k=1}^{m-\ell}{m\choose k}S(m-k,\ell) = (\ell+1)S(m,\ell +1)\quad \text{for } \ell<m.
	\end{align}

\begin{lemma}\label{lem:d_recursive_id}
Let $c>0$ and let $\bsb = (b_j)_{j\geq 1}$, $(\bbA_{\bsnu})_{\bsnu\in\calF}$ and $(\bbB_{\bsnu})_{\bsnu\in\calF}$  be sequences of non-negative real numbers that satisfy the recurrence
	\begin{align}\label{eq:recurs}
	\bbA_{\bsnu} = \sum_{j\geq1}\sum_{k=1}^{\nu_j} {\nu_j\choose k}c^k \,b_j \, \bbA_{\bsnu-k\bse_j} + \bbB_\bsnu
	\quad \text{for all } \bsnu \in \calF \mbox{ including }\bsnu = \bszero,
  	\end{align}
  where $\bse_j$ is the multi-index whose the $j$th component is 1 and all other components are 0.
Then
	\begin{align}\label{eq:closed-form}
	\bbA_{\bsnu} = \sum_{\bsm \leq \bsnu}c^{|\bsm|}{\bsnu \choose \bsm} \bbB_{\bsnu-\bsm}
	 \sum_{\bsl \leq \bsm} |\bsl|!\, \bsb^{\bsl}  \prod_{i\geq1} S(m_i,\ell_i).
	\end{align}
	where $S(m_i,\ell_i)$ for $i\geq 1$ are Stirling numbers of the second kind given by \eqref{eq:stirling_def}.
	The result also holds with both equalities replaced by inequalities $\leq$.
\end{lemma}

\begin{proof}
The statement is proved by using  induction   on $|\bsnu|$. The statement holds trivially for $\bsnu=\bszero$ since $S(0,0)=1$. Assume the statement \eqref{eq:closed-form} holds for all multi-indices of order less than $|\bsnu|$, i.e., if $1\leq k \leq \nu_j$ for some $j\geq 1$, then 
	\begin{align}\label{eq:induct_hyp}
	\bbA_{\bsnu-k\bse_j} = \sum_{\bsm \leq \bsnu-k\bse_j}c^{|\bsm|}{\bsnu-k\bse_j\choose \bsm} 
	\bbB_{\bsnu-k\bse_j-\bsm}
	 \sum_{\bsl \leq \bsm} |\bsl|! \,\bsb^{\bsl}  \prod_{i\geq1} S(m_i,\ell_i).
	\end{align}

We now prove the statement for indices of order $|\bsnu|$. Substituting the induction hypothesis  \eqref{eq:induct_hyp} into \eqref{eq:recurs}, we have
	\begin{align}\label{eq:rtp_induction}
	\bbA_{\bsnu} = \bbB_\bsnu +\sum_{j\geq1}\Phi_j,
	\end{align}
	where
	\begin{align*}
	\Phi_j \coloneqq \sum_{k=1}^{\nu_j} {\nu_j\choose k}c^k \,b_j 
 \sum_{\bsm \leq \bsnu-k\bse_j}c^{|\bsm|}{\bsnu-k\bse_j\choose \bsm} \bbB_{\bsnu-k\bse_j-\bsm}
	 \sum_{\bsl \leq \bsm} |\bsl|! \,\bsb^{\bsl}  \prod_{i\geq1} S(m_i,\ell_i).
	\end{align*}

We use a dash to denote multi-indices with the $j$th component removed, for example, $\bsnu'~=~(\nu_1,\ldots,\nu_{j-1},\nu_{j+1},\ldots)$ and
adopt the notation $\bbB_{\bsnu} = \bbB_{\bsnu',\nu_j} $. Then
	\begin{align}\label{eq:full_sum}
	\Phi_j=&\sum_{k=1}^{\nu_j} {\nu_j\choose k}c^k \,b_j \,
	 \Bigg[\sum_{\bsm' \leq \bsnu'}\sum_{m_j=0}^{\nu_j-k}c^{|\bsm'|+m_j}{\bsnu'\choose \bsm'} {\nu_j-k\choose m_j} \bbB_{\bsnu'-\bsm',\nu_j-k-m_j} \notag\\
	 &\hspace{4cm} \times\sum_{\bsl' \leq \bsm'}\sum_{\ell_j=0}^{m_j} (|\bsl'|+\ell_j)!\, \bsb'^{\bsl'}\, b_j^{\ell_j}  S(m_j,\ell_j) \prod_{\satop{i\geq1}{i\neq j}} S(m_i,\ell_i)\Bigg]\notag\\
	 =&\sum_{\bsm' \leq \bsnu'}c^{|\bsm'|}{\bsnu'\choose \bsm'} \sum_{\bsl' \leq \bsm'}\bsb'^{\bsl'}\, \bigg(\prod_{\satop{i\geq1}{i\neq j}} S(m_i,\ell_i)\bigg)\,\bbB_{\bsnu'-\bsm',\nu_j-k-m_j}\notag\\
	 &\hspace{2cm}	\times \underbrace{\sum_{k=1}^{\nu_j} 
	 \sum_{m_j=0}^{\nu_j-k}c^{m_j+k}{\nu_j\choose k} {\nu_j-k\choose m_j} \,b_j 
	 \sum_{\ell_j=0}^{m_j} (|\bsl'|+\ell_j)!\,  b_j^{\ell_j}  S(m_j,\ell_j)}_{\eqcolon\, \Theta_1},
	 \end{align}
where we have rearranged the terms and swapped the order of the sums.

Recognising that  $${\nu_j\choose k} {\nu_j-k\choose m_j} ={\nu_j\choose m_j+k} {m_j+k\choose k},$$ 
we have
	 \begin{align}
	 \Theta_1 &=\sum_{k=1}^{\nu_j} 
	 \sum_{m_j=0}^{\nu_j-k}\!\!c^{m_j+k}{\nu_j\choose m_j+k}\! {m_j+k\choose k} \bbB_{\bsnu'-\bsm',\nu_j-(k+m_j)}
	 \!\!\sum_{\ell_j=0}^{m_j} (|\bsl'|+\ell_j)!\,  b_j^{\ell_j+1}  S(m_j,\ell_j)\notag\\
	 &= \sum_{k=1}^{\nu_j} 
	 \sum_{m_j=k}^{\nu_j}c^{m_j}{\nu_j\choose m_j} {m_j\choose k} \bbB_{\bsnu'-\bsm',\nu_j-m_j}
	 \sum_{\ell_j=0}^{m_j-k} (|\bsl'|+\ell_j)! \, b_j^{\ell_j+1}  S(m_j-k,\ell_j)\notag\\
	 &= \sum_{m_j=1}^{\nu_j} c^{m_j}{\nu_j\choose m_j} \bbB_{\bsnu'-\bsm',\nu_j-m_j}
	 \underbrace{\sum_{k=1}^{m_j} {m_j\choose k}  \sum_{\ell_j=0}^{m_j-k} (|\bsl'|+\ell_j)!\,  b_j^{\ell_j+1}  S(m_j-k,\ell_j)}_{\eqcolon \, \Theta_2}
	 \label{eq:mkl_sum} ,	 
	 \end{align}
where we get to the second line by relabelling the summation indices of the second sum from $m_j=0,\ldots,\nu_j-k$ to $m_j=k,\ldots,\nu_j$ and to the third line by swapping the order of sums indexed by $k$ and $m_j$. Swapping the order of the summations indexed by $k$ and $\ell_j$ in $\Theta_2$ gives
	\begin{align}
\Theta_2 &=\sum_{\ell_j=0}^{m_j-1}(|\bsl'|+\ell_j)!\,b_j^{\ell_j+1}\sum_{k=1}^{m_j-\ell_j}{m_j\choose k}S(m_j-k,\ell_j) \notag\\
	&=\sum_{\ell_j=0}^{m_j-1}(|\bsl'|+\ell_j)!\,b_j^{\ell_j+1}(\ell_j+1)\,S(m_j,\ell_j+1)\notag\\
	&=\sum_{\ell_j=1}^{m_j}(|\bsl'|+\ell_j -1)!\,\ell_j\,b_j^{\ell_j}S(m_j,\ell_j)
	=\sum_{\satop{\ell_j=0}{|\bsl'|+\ell_j\neq 0}}^{m_j}(|\bsl'|+\ell_j -1)!\,\ell_j\,b_j^{\ell_j}S(m_j,\ell_j), \label{eq:new_S2}
	\end{align}
where the second equality is obtained using \eqref{eq:stirling_id} and the third equality is from a simple re-indexing of the sum. To reach the final line, we can add the terms corresponding to $\ell_j=0$ into the sum so that the sum begins from $\ell_j=0$ because these terms are all equal to 0 due to the presence of an $\ell_j$ factor in the terms of the series, provided that we also introduce the condition $|\bsl'|+\ell_j\neq 0$ to ensure the factorial term is defined. 
Substituting \eqref{eq:new_S2} into \eqref{eq:mkl_sum} gives
	\begin{align*}
	\Theta_1 
	 &= \sum_{m_j=0}^{\nu_j} c^{m_j}{\nu_j\choose m_j}
	 \bbB_{\bsnu'-\bsm',\nu_j-m_j}
	 \sum_{\satop{\ell_j=0}{|\bsl'|+\ell_j\neq 0}}^{m_j}(|\bsl'|+\ell_j -1)!\,\ell_j\,b_j^{\ell_j}S(m_j,\ell_j),
	 \end{align*}
where we have added the terms corresponding to $m_j=0$ into the summation over $m_j$ since these terms are all also equal to 0 as $\ell_j=0$ when $m_j = 0$. Substituting this formula for $\Theta_1$ back into \eqref{eq:full_sum}  then rearranging gives	 
	 \begin{align*}
	 \Phi_j &= \sum_{\bsm' \leq \bsnu'}c^{|\bsm'|}{\bsnu'\choose \bsm'} \sum_{\bsl' \leq \bsm'}\bsb'^{\bsl'}\, \bigg(\prod_{\satop{i\geq1}{i\neq j}} S(m_i,\ell_i)\bigg)\notag\\
	 &\qquad\quad \times\sum_{m_j=0}^{\nu_j} c^{m_j}{\nu_j\choose m_j}\bbB_{\bsnu'-\bsm',\nu_j-m_j}
	 \sum_{\satop{\ell_j=0}{|\bsl'|+\ell_j\neq 0}}^{m_j}(|\bsl'|+\ell_j -1)!\,\ell_j\,b_j^{\ell_j}S(m_j,\ell_j)\\
	  &=\sum_{\bsm \leq \bsnu} c^{|\bsm|}{\bsnu\choose \bsm} \bbB_{\bsnu-\bsm}\sum_{\bszero\neq\bsl \leq \bsm}   \bsb^{\bsl}(|\bsl| -1)!\,\ell_j\,\prod_{i\geq1} S(m_i,\ell_i),
	\end{align*}
	where we obtain the last line by combining the sums over $\bsm'$ and $m_j$ into a single sum over the original index $\bsm$ and combine the sums over $\bsl'$ and $\ell_j$ into a single sum over the index $\bsl$.
	
Now substituting this back into \eqref{eq:rtp_induction}, we have
	\begin{align*}
	\bbA_{\bsnu} &=  \bbB_\bsnu +\sum_{j\geq1}\sum_{\bsm \leq \bsnu} c^{|\bsm|}{\bsnu\choose \bsm} \bbB_{\bsnu-\bsm}\sum_{\bszero\neq\bsl \leq \bsm}   \bsb^{\bsl}\,(|\bsl| -1)!\,\ell_j\,\prod_{i\geq1} S(m_i,\ell_i)\\
	&=\bbB_\bsnu +\sum_{\bsm \leq \bsnu}c^{|\bsm|}{\bsnu\choose \bsm} \bbB_{\bsnu-\bsm} \sum_{\bszero\neq\bsl \leq \bsm}    \bsb^{\bsl}\,|\bsl|!\,\prod_{i\geq1} S(m_i,\ell_i)\\
	&=\bbB_\bsnu +\sum_{ \bsm \leq \bsnu} c^{|\bsm|}{\bsnu\choose \bsm} \bbB_{\bsnu-\bsm}
	 \bigg(-\prod_{i\geq1}S(m_i,0)+\sum_{\bsl \leq \bsm}   \bsb^{\bsl}\,|\bsl|!\,\prod_{i\geq1} S(m_i,\ell_i)\bigg)\\
	&=\sum_{\bsm \leq \bsnu} c^{|\bsm|}{\bsnu\choose \bsm}  \sum_{\bsl \leq \bsm}   \bsb^{\bsl}\,|\bsl|!\,\bbB_{\bsnu-\bsm}\prod_{i\geq1} S(m_i,\ell_i),
	\end{align*}
	which is our desired result. We move to the second line interchanging the order of the summations and then to the third line by including the case when $\bsl= \bszero$ and subtracting the corresponding term $\prod_{i\geq 1}S(m_i,0)$ to maintain equality. Since $S(m_i,0)=0$ when $m_i\neq 0$, the term $\prod_{i\geq 1}S(m_i,0)$ only contributes when $\bsm=\bszero$, and the resulting sum cancels out with the term $\bbB_\bsnu$ giving the required result. 
	
	Since this induction proof holds for equality, in the case when \eqref{eq:recurs} and \eqref{eq:closed-form} have their equalities replaced by inequalities $\leq$, the statement will continue to hold.
\end{proof}

From \cite[Lemma A.3]{HHKKS23}, we also have that for some $\bsnu\in \calF$,
	\begin{align}\label{eq:stirling_identity}
	&\sum_{\bsm\leq\bsnu}\sum_{\bsl\leq\bsm}\sum_{\bsk\leq\bsnu-\bsm}{\bsnu\choose\bsm} 
	\bbA_{\bsl}\,\bbB_{\bsk} \prod_{i\geq1} \bigg(S(m_i,\ell_i)\,S(\nu_i - m_i,k_i)\bigg)\notag\\
	=&\sum_{\bsm\leq\bsnu}	 \bigg(\sum_{\bsl\leq\bsm} {\bsm\choose\bsl} \bbA_{\bsl}\,\bbB_{\bsm-\bsl}  \bigg) 
	  \prod_{i\geq1} S(\nu_i,m_i),
	\end{align}
	where  $(\bbA_\bsnu)_{\bsnu\in \calF}$ and $(\bbB_\bsnu)_{\bsnu\in \calF}$ are arbitrary sequences of real numbers.

\section{Parametric regularity proofs}\label{app:stoch_reg_pf}
\begin{proofof}{Lemma \ref{lem:laplace_d_u}}
We follow a similar strategy to the proof of \cite[Lemma 6.2]{KN16}. 
For $\bsnu=\bszero$, we rewrite the strong formulation \eqref{eq:pde} using the product rule to obtain
	\begin{align*}
	-\Psi(\bsx,\bsy)\,\Delta u(\bsx,\bsy) &= \nabla \Psi(\bsx,\bsy) \cdot \nabla u(\bsx,\bsy)+f(\bsx) ,
	\end{align*}
	from which it follows that
	\begin{align*}
	\Psi_{\min} \|\Delta u(\cdot,\bsy)\|_{L^2(D)} 
	\leq \|\nabla \Psi(\cdot,\bsy)\|_{L^\infty(D)} \| u(\cdot,\bsy)\|_{V}  + \|f\|_{L^2(D)},
	\end{align*} 
	where we have used Assumption \eqref{asm:a_bounds} and the definition $\| u(\cdot,\bsy)\|_{V} = \|\nabla u(\cdot,\bsy)\|_{L^2(D)}$.  
Combining this with \eqref{eq:lax-milgram}, we have
	\begin{align*}
	\|\Delta u(\cdot,\bsy)\|_{L^2(D)} 
	&\leq \frac{1}{\Psi_{\min}}\Bigg(\sup_{\bsy\in \Omega}\|\nabla \Psi(\cdot,\bsy)\|_{L^\infty(D)} 
	\frac{\|f\|_{V'}}{\Psi_{\min}} + \|f\|_{L^2(D)} \Bigg) \\
	&\leq \frac{1}{\Psi_{\min}}\Bigg(\sup_{\bsy\in \Omega}\|\nabla\Psi(\cdot,\bsy)\|_{L^\infty(D)} 
	\frac{C_{\rm{Poi}}\|f\|_{L^2(D)}}{\Psi_{\min}} + \|f\|_{L^2(D)} \Bigg) \\
	&\leq C_1\|f\|_{L^2(D)},
	\end{align*}			
where 
\begin{align}\label{eq:A_0bnd}
	C_1 \coloneqq 
	\frac{\max\{1,C_{\rm{Poi}}\}}{\Psi_{\min}}\bigg(\frac{ \|\nabla \Psi\|_{L^\infty(D\times\Omega)}}{\Psi_{\min}}+1\bigg),
	\end{align}
	and $C_{\rm{Poi}}$ is the Poincar\'e constant of the embedding $V\hookrightarrow L^2(D)$.
It follows from Assumption \eqref{asm:a_W1inf} that  $\|\nabla \Psi\|_{L^\infty(D\times\Omega)}\coloneqq \sup_{\bsy\in\Omega}\|\nabla \Psi(\cdot,\bsy)\|_{L^\infty(D)}$ is finite. Hence, $u(\cdot,\bsy)\in H^2(D)$ and \eqref{eq:laplace_derivative} holds for $\bsnu = \bszero$.

Recall that the gradient $\nabla$ and Laplacian $\Delta$ are taken with respect to $\bsx$ and that $\partial^{\bsnu}$ is taken with respect to $\bsy$. For $\bsnu \neq \bszero$, we take the $\bsnu$th derivative of \eqref{eq:pde} using the Leibniz product rule. 
Since $f$ is independent of $\bsy$, the right-hand side is 0 and we can rearrange the resulting expression to obtain, omitting the dependence of $\bsx$~and~$\bsy$,
	\begin{align}\label{eq:diff_pde_y}
\nabla \cdot(\Psi\nabla \partial^{\bsnu}u) 
&= -\nabla \cdot \bigg(\ \sum_{\bszero \neq \bsm\leq \bsnu} {\bsnu\choose \bsm} \big(\partial^{\bsm}\Psi\big)\big(\nabla\partial^{\bsnu-\bsm}u\big) \bigg) \notag \\ 
&=  - \nabla \cdot \bigg(\sum_{j\geq 1} \sum_{k=1}^{\nu_j} {\nu_j \choose k}
	(2\pi)^k \sin\bigg(2\pi y_j + \frac{k\pi}{2}\bigg)\psi_j	\big(\nabla\partial^{\bsnu-k\bse_j}u\big) \bigg), 
	\end{align}
where we use the fact that the mixed partial derivatives of $\Psi$ with respect to $\bsy$ are
	\begin{align} \label{eq:d_a(x,y)}
	\partial^\bsm \Psi(\bsx,\bsy) = 
	\begin{cases}
	\Psi(\bsx,\bsy)  &\text{if } \bsm=\bszero,\\
	(2\pi)^k\sin\bigg(2\pi y_j +\dfrac{k\pi}{2}\bigg)\psi_j(\bsx) &\text{if } \bsm=k\bse_j,\, k\geq 1,\\
	0   &\text{otherwise.}
	\end{cases}
	\end{align}	
	
Expanding \eqref{eq:diff_pde_y} using the product rule for $\nabla$ and rearranging gives 
	\begin{align*}
	&\Psi \Delta \partial^\bsnu u = -\nabla \Psi \cdot \nabla \partial^\bsnu u \\ 
	&\qquad\quad -\sum_{j\geq 1} \sum_{k=1}^{\nu_j} {\nu_j \choose k}
	(2\pi)^k\sin\bigg(2\pi y_j + \frac{k\pi}{2}\bigg)
	\Big(\big(\nabla \psi_j\big)\cdot\big(\nabla\partial^{\bsnu-k\bse_j}u\big) + \psi_j\big(\Delta \partial^{\bsnu-k\bse_j}u\big)\Big).
	\end{align*}
Now taking the $L^2(D)$ norm and applying the triangle inequality gives
	\begin{align*}
&\Psi_{\min} \|\Delta\partial^\bsnu u\|_{L^2(D)} \leq \|\nabla \Psi\|_{L^\infty(D)}\|\nabla\partial^\bsnu u\|_{L^2(D)} \\
	& + \sum_{j\geq 1} \sum_{k=1}^{\nu_j} {\nu_j \choose k}
	(2\pi)^k\big( \|\psi_j\|_{L^\infty(D)} \|\Delta \partial^{\bsnu-k\bse_j}u\|_{L^2(D)}+\|\nabla \psi_j\|_{L^\infty(D)} \|\nabla\partial^{\bsnu-k\bse_j}u\|_{L^2(D)} \big),
	\end{align*}		
where we used $|\sin(x)| \leq 1$ for all real $x$. Formulating the above into a recursion gives
	\begin{align}\label{eq:delta_u_recurs}
	& \|\Delta\partial^\bsnu u\|_{L^2(D)} \leq \sum_{j\geq 1} \sum_{k=1}^{\nu_j} {\nu_j \choose k}
	(2\pi)^k\, b_j \, \|\Delta \partial^{\bsnu-k\bse_j}u\|_{L^2(D)}  + B_\bsnu,
	\end{align}
where $b_j$ and $\overline{b}_j$ are defined in \eqref{eq:bb_bar} and we define 
 	\begin{align*}
B_\bsnu \coloneqq	\frac{\|\nabla \Psi\|_{L^\infty(D)}}{\Psi_{\min}}\|\partial^\bsnu u\|_{V}+
 	 \sum_{j\geq 1} \sum_{k=1}^{\nu_j} {\nu_j \choose k}
(2\pi)^k\,\overline{b}_j\,\|\partial^{\bsnu-k\bse_j}u\|_{V}.
 	\end{align*}
 	Substituting \eqref{eq:mixed_stoch_der} into $B_\bsnu$, we can bound $B_\bsnu$ by
 	\begin{align}\label{eq:Bnu_bnd}
 		B_\bsnu  &\leq  \frac{\|\nabla \Psi\|_{L^\infty(D)}}{\Psi_{\min}}
\frac{\|f\|_{V'}}{\Psi_{\min}}(2\pi)^{|\bsnu|} \sum_{\bsm\leq\bsnu} |\bsm|!\,\bsb^\bsm \prod_{i\geq1}S(\nu_i,m_i)\notag\\
&\qquad+\sum_{j\geq 1} \sum_{k=1}^{\nu_j} {\nu_j \choose k}
(2\pi)^k\,\overline{b}_j\frac{\|f\|_{V'}}{\Psi_{\min}}(2\pi)^{|\bsnu|-k} 
\!\!\!\!\!
\sum_{\bsm\leq\bsnu-k\bse_j} \!\!\!\!\! |\bsm|!\,\bsb^\bsm S(\nu_j-k,m_j)\prod_{\satop{i\geq1}{i\neq j}}S(\nu_i,m_i)\notag\\
	&=\frac{\|f\|_{V'}}{\Psi_{\min}}(2\pi)^{|\bsnu|}
	\bigg[\frac{\|\nabla a\|_{L^\infty(D)}}{\Psi_{\min}}  \sum_{\bsm\leq\bsnu} |\bsm|!\,\bsb^\bsm \prod_{i\geq1}S(\nu_i,m_i) \notag  \\
	&\qquad+  \underbrace{\sum_{j\geq 1} \sum_{k=1}^{\nu_j} {\nu_j \choose k}
\,\overline{b}_j \sum_{\bsm\leq\bsnu-k\bse_j} |\bsm|!\,\bsb^\bsm S(\nu_j-k,m_j)\prod_{\satop{i\geq1}{i\neq j}}S(\nu_i,m_i)}_{\eqcolon\,\Theta} 
	\bigg].
 	\end{align} 

We simplify $\Theta$ using the same strategy in the proof of Lemma \ref{lem:d_recursive_id} by separating out the $j$th component of the sum over $\bsm$ and interchanging the order of the sums to give
	\begin{align*}
	\Theta
	&=\sum_{j\geq 1} \sum_{k=1}^{\nu_j} {\nu_j \choose k} \overline{b}_j \, \sum_{\bsm'\leq\bsnu'}
	\sum_{m_j=0}^{\nu_j-k} (|\bsm'|+m_j)! \, \bsb'^{\bsm'}b_j^{m_j} S(\nu_j-k,m_j) \prod_{\satop{i\geq1}{i\neq j}}S(\nu_i,m_i)\\
	&=\sum_{j\geq 1} 
	 \sum_{\bsm'\leq\bsnu'}\bsb'^{\bsm'} \bigg(\prod_{\satop{i\geq1}{i\neq j}}S(\nu_i,m_i)\bigg)
	\sum_{m_j=0}^{\nu_j-k} 
\sum_{k=1}^{\nu_j} {\nu_j \choose k} \,\overline{b}_j 	\,
	(|\bsm'|+m_j)! \, b_j^{m_j} S(\nu_j-k,m_j) \\
	&\leq\sum_{j\geq 1} 
	 \sum_{\bsm'\leq\bsnu'} \bsb'^{\bsm'}\bigg(\prod_{\satop{i\geq1}{i\neq j}}S(\nu_i,m_i)\bigg)
	\sum_{\satop{m_j=0}{|\bsm'|+m_j\neq 0}}^{\nu_j}
	(|\bsm'|+m_j-1)!\,m_j\,\overline{b}_j^{m_j}\,S(\nu_j,m_j)\\
	&\leq  \sum_{\bsm\leq\bsnu} |\bsm|!\, \overline{\bsb}^{\bsm} \prod_{i\geq 1} S(\nu_i,m_i),
	\end{align*}		 	
 	where we drop the condition $|\bsm'|+m_j\neq 0$ since if $\bsm=\bszero$ and $\bsnu\neq \bszero $, there exists some index $i$ such that $\nu_i\neq 0$ and  $S(\nu_i,m_i) =0$. We have also replaced $\bsb$ with $\overline{\bsb}$ since $b_i\leq \overline{b}_i$ for all $i\geq1$.
 	
Substituting the bound on $\Theta$  into \eqref{eq:Bnu_bnd} and using $\bsb\leq \overline \bsb$, we can bound $B_\bsnu$ by 
	\begin{align*}
	B_\bsnu &\leq \|f\|_{L^2(D)} \frac{C_{\rm{Poi}}}{\Psi_{\min}}\bigg(\frac{\|\nabla \Psi\|_{L^\infty(D\times \Omega)}}{\Psi_{\min}} + 1\bigg) 
	(2\pi)^{|\bsnu|}\sum_{\bsm\leq\bsnu} |\bsm|!\, \overline{\bsb}^{\bsm} \prod_{i\geq 1} S(\nu_i,m_i)\\
	&\leq  C_1 \|f\|_{L^2(D)} 
	(2\pi)^{|\bsnu|}\sum_{\bsm\leq\bsnu} |\bsm|!\, \overline{\bsb}^{\bsm} \prod_{i\geq 1} S(\nu_i,m_i)\eqcolon \bbB_{\bsnu},
	\end{align*}	 	
 where $C_1$ is defined in \eqref{eq:A_0bnd}.
 	Thus, defining $\bbA_\bsnu \coloneqq \|\Delta \partial^\bsnu u\|_{L^2(D)}$, we can write \eqref{eq:delta_u_recurs} as 
 	\begin{align*}
 	\bbA_\bsnu \leq \sum_{j\geq 1} \sum_{k=1}^{\nu_j} {\nu_j \choose k}
	(2\pi)^k\, b_j\,\bbA_{\bsnu-k\bse_j} + \bbB_\bsnu.
 	\end{align*}
 	Noting that $\bbA_0\leq \bbB_0$, we can then apply Lemma \ref{lem:d_recursive_id} (we cannot apply Lemma~\ref{lem:d_recursive_id} to \eqref{eq:delta_u_recurs} with $B_{\bsnu}$ since it is not true that $\bbA_\bszero \leq B_\bszero$) to give
 	\begin{align*}
 	\|\Delta \partial^\bsnu u\|_{L^2(D)} 
 	&\leq 
 	\sum_{\bsm\leq \bsnu} (2\pi)^{|\bsm|} {\bsnu \choose \bsm} \bigg(\sum_{\bsl\leq\bsm} |\bsl|!\,\bsb^\bsl \prod_{i\geq 1} S(m_i,\ell_i)\bigg)\\
	&\times
	 { C_1 \|f\|_{L^2(D)} }(2\pi)^{|\bsnu-\bsm|}\sum_{\bsk\leq\bsnu-\bsm} |\bsk|!\, \overline{\bsb}^{\bsk} \prod_{i\geq 1} S(\nu_i-m_i,k_i)\\
	&\leq  { C_1 \|f\|_{L^2(D)} }(2\pi)^{|\bsnu|} \\
 	&
 	\times \!\! \sum_{\bsm\leq \bsnu}\!\! {\bsnu \choose \bsm} \bigg(\sum_{\bsl\leq\bsm}\! |\bsl|!\,\overline{\bsb}^\bsl \prod_{i\geq 1} S(m_i,\ell_i)\!\bigg)
	\bigg(\sum_{\bsk\leq\bsnu-\bsm}\!\!\! |\bsk|!\, \overline{\bsb}^{\bsk} \prod_{i\geq 1}\! S(\nu_i-m_i,k_i)\!\bigg),
 	\end{align*}
Then, using \eqref{eq:stirling_identity} and the identity
$ 	\sum_{\bsl\leq \bsm} {\bsm\choose \bsl}|\bsl|!\,|\bsm-\bsl|! = (|\bsm|+1)!$
	from  \cite{KN16}, we obtain,
\begin{align*}
\|\Delta(\partial^\bsnu u)\|_{L^2(D)} 
&\leq { C_1 \|f\|_{L^2(D)}} (2\pi)^{|\bsnu|} \sum_{\bsm\leq\bsnu} \overline\bsb^\bsm \bigg[\sum_{\bsl\leq \bsm} {\bsm\choose \bsl}|\bsl|!\,|\bsm-\bsl|!\bigg] \prod_{i\geq 1} S(\nu_i,m_i)\\
	&= C_1 \|f\|_{L^2(D)}(2\pi)^{|\bsnu|}
	\sum_{\bsm\leq\bsnu}(|\bsm|+1)!\,\overline{\bsb}^\bsm \prod_{i\geq 1} S(\nu_i,m_i),
\end{align*}
as required.
\end{proofof}

\begin{proofof}{Lemma \ref{lem:d_u-uh}}
We follow the proof strategy presented in \cite[Lemma 6.3]{KN16}.
Let $f\in L^2(D)$, $\bsy\in \Omega$ and $\bsnu\in \calF$. 
Since $u_h$ is analytic in $V$, we have that $\partial^{\bsnu}u_h \in V_h$ for every $\bsnu\in\calF$ and hence
	\begin{align*}
	(\sfI - \sfP^h_\bsy)(\partial^\bsnu u_h(\cdot,\bsy)) \equiv 0,
	\end{align*}
where $\sfP^h_\bsy$ is the orthogonal projection defined by \eqref{eq:orth_proj}.
It then follows that
	\begin{align}\label{eq:norm_d_err}
	\|\partial^{\bsnu}(u-u_h)\|_{V} 
	&= \|\sfP_\bsy^h\partial^{\bsnu}(u-u_h) + (\sfI-\sfP_\bsy^h)\partial^{\bsnu}(u-u_h)\|_{V} \notag\\
	&\leq \|\sfP_\bsy^h\partial^{\bsnu}(u-u_h)\|_{V} 
	+\|(\sfI-\sfP_\bsy^h)\partial^{\bsnu}u \|_{V},
	\end{align}
where we have omitted the dependence of $u$ on $\bsx$ and $\bsy$ for brevity.
	
Starting with the Galerkin orthogonality property given in \eqref{eq:gal_orth}, we take the $\partial^{\bsnu}$ derivative with respect to $\bsy$ using the Leibniz product rule to give
	\begin{align}\label{eq:d_gal_orth}
	\int_D \sum_{\bsm\le\bsnu} {\bsnu\choose\bsm}(\partial^\bsm \Psi)
	\nabla\partial^{\bsnu-\bsm}(u-u_h)\cdot \nabla v_h\,\rd\bsx = 0 \quad \text{for all } v_h\in V_h.
	\end{align}		
Next, separating out the term when $\bsm=\bszero$ and then substituting the mixed derivatives of the coefficient \eqref{eq:d_a(x,y)} into \eqref{eq:d_gal_orth}, we obtain
	\begin{align*}
	&\int_D \Psi\,\nabla\partial^\bsnu(u-u_h)\cdot\nabla v_h \,\rd\bsx \notag\\
	&\qquad  = - \sum_{j\geq 1}\sum_{k=1}^{\nu_j} {\nu_j \choose k} 
	(2\pi)^k\sin\bigg(2\pi y_j+\frac{k\pi}{2}\bigg)\int_D \psi_j \nabla\partial^{\bsnu-k\bse_j}(u-u_h) \cdot \nabla v_h\, \rd\bsx
	\end{align*}	
	for all $v_h\in V_h$.
Now, letting $v_h = \sfP_\bsy^h\partial^\bsnu(u-u_h)$, we have
	\begin{align}\label{eq:sub}
	&\int_D \Psi\,\nabla\partial^\bsnu(u-u_h)\cdot\nabla \sfP_\bsy^h\partial^\bsnu(u-u_h) \,\rd\bsx \notag\\
	&= - \sum_{j\geq 1}\sum_{k=1}^{\nu_j} {\nu_j \choose k} 
	(2\pi)^k\sin\!\bigg(2\pi y_j+\frac{k\pi}{2}\bigg)\!\int_D\! \psi_j \nabla\partial^{\bsnu-k\bse_j}(u-u_h) \cdot \nabla \sfP^h_\bsy\partial^\bsnu(u-u_h) \, \rd\bsx.
	\end{align}
Since $\partial^{\bsnu}(u-u_h)\in V$, using \eqref{eq:orth_proj} with $w =\partial^{\bsnu}(u-u_h)$ and $v_h = \sfP_\bsy^h\partial^\bsnu(u-u_h)$ yields
	\begin{align*}
	\int_D \Psi\nabla((\sfI-\sfP^h_\bsy)\partial^\bsnu(u-u_h))\cdot\nabla \sfP_\bsy^h\partial^\bsnu(u-u_h)\,\rd\bsx=0,
	\end{align*}
	which can then be rearranged to give
	\begin{align}
	\int_D \Psi\nabla\partial^\bsnu(u-u_h)\cdot \nabla \sfP^h_\bsy\partial^\bsnu(u-u_h)\,\rd\bsx
	&=\int_D \Psi\,|\nabla \sfP_\bsy^h\partial^\bsnu(u-u_h)|^2\,\rd\bsx \notag\\
	 &\geq \Psi_{\min} \|\sfP^h_\bsy \partial^\bsnu(u-u_h)\|^2_{V}, \label{eq:LHS_identity}
\end{align}	 	
where we have used Assumption \eqref{asm:a_bounds}. Substituting in the lower bound \eqref{eq:LHS_identity} for the left-hand side of \eqref{eq:sub}, applying the Cauchy-Schwarz inequality to the right hand side and then dividing through by $ \Psi_{\min}\|\sfP^h_\bsy \partial^\bsnu(u-u_h)\|_{V}$ yields
	\begin{align*}
	\|\sfP^h_\bsy \partial^\bsnu(u-u_h)\|_{V}
	\leq \frac{1}{\Psi_{\min}}\sum_{j\geq 1}  \sum_{k=1}^{\nu_j} {\nu_j \choose k}(2\pi)^k\,\|\psi_j\|_{L^\infty(D)}\,
	\| \partial^{\bsnu-k\bse_j}(u-u_h)\|_V.
	\end{align*}
Then substituting this into \eqref{eq:norm_d_err} gives
	\begin{align*}
	\|\partial^{\bsnu}(u-u_h)\|_{V} 
	&\leq \sum_{j\geq 1}  \sum_{k=1}^{\nu_j} {\nu_j \choose k}(2\pi)^k b_j\,\,
	\| \partial^{\bsnu-k\bse_j}(u-u_h)\|_V
	+\|(\sfI-\sfP^h_\bsy)\partial^{\bsnu}u \|_{V},
	\end{align*}	 
	where $b_j$ is defined in \eqref{eq:bb_bar}.

	Applying Lemma \ref{lem:d_recursive_id} with $\bbA_\bsnu = \| \partial^{\bsnu}(u-u_h)\|_V$, $\bbB_\bsnu= \|(\sfI-\sfP^h_\bsy)\partial^{\bsnu}u \|_{V}$ and $c=2\pi$ to the above inequality gives
	\begin{align*}
	\|\partial^{\bsnu}(u-u_h)\|_{V}  \leq 
	\sum_{\bsm \leq \bsnu}(2\pi)^{|\bsm|}{\bsnu \choose \bsm} 
	 \bigg(\sum_{\bsl \leq \bsm}  |\bsl|! \,\bsb^{\bsl}  \prod_{i\geq1} S(m_i,\ell_i)\bigg) \|(\sfI-\sfP^h_\bsy)\partial^{\bsnu-\bsm}u \|_{V}.
	\end{align*}
	Since $\partial^{\bsnu-\bsm}u\in H^2(D)$,  we have $\|(\sfI-\sfP^h_\bsy)\partial^{\bsnu-\bsm}u \|_{V}\leq C\,h\,\|\Delta\partial^{\bsnu-\bsm}u\|_{L^2(D)}$ from \eqref{eq:d_fe_error} (with $C$ independent of $\bsy$), which in turn gives 
	\begin{align*}
	\|\partial^{\bsnu}(u-u_h)\|_{V}  &\leq C\,h\sum_{\bsm \leq \bsnu}(2\pi)^{|\bsm|}{\bsnu \choose \bsm} 
	 \bigg(\sum_{\bsl \leq \bsm} |\bsl|! \,\bsb^{\bsl}  \prod_{i\geq1} S(m_i,\ell_i)\bigg)\|\Delta \partial^{\bsnu-\bsm}u\|_{L^2(D)}.
\end{align*}		
	
Then using \eqref{eq:laplace_derivative} from Lemma \ref{lem:laplace_d_u} with constant $C_1$, $\bsb\leq \overline{\bsb}$ and defining $C_2~\coloneqq~{C\,C_1}/{2}$, we obtain
	\begin{align*}
	&\|\partial^{\bsnu}(u-u_h)\|_{V} \\
	 &\leq 2C_2\,h\,(2\pi)^{|\bsnu|}\,\|f\|_{L^2(D)}\\
	 &\qquad\times\sum_{\bsm \leq \bsnu}{\bsnu \choose \bsm} 
	 \bigg(\sum_{\bsl \leq \bsm} |\bsl|!\, \overline{\bsb}^{\bsl}  \prod_{j\geq1} S(m_j,\ell_j)\bigg)
	 \bigg(\sum_{\bsk\leq \bsnu-\bsm} (|\bsk|+1)!\,\overline{\bsb}^{\bsk} \prod_{i\geq 1} S(\nu_i-m_i,k_i)\bigg)\\
	 &=  2C_2\,h\,(2\pi)^{|\bsnu|}\,\|f\|_{L^2(D)}
	 \sum_{\bsm \leq \bsnu} \overline{\bsb}^\bsm
	  \bigg(\sum_{\bsl \leq \bsm} {\bsm \choose \bsl} |\bsl|!\,(|\bsm-\bsl|+1)!\bigg)
	   \prod_{i\geq 1}S(\nu_i,m_i) \\
	 &=  C_2\,h\, (2\pi)^{|\bsnu|}\,\|f\|_{L^2(D)}
	 \sum_{\bsm \leq \bsnu}(|\bsm|+2)!\,\overline{\bsb}^\bsm\prod_{i\geq 1}S(\nu_i,m_i).
	\end{align*}
	We arrive at the first equality using \eqref{eq:stirling_identity} with $\bbA_{\bsl} = |\bsl|!\,\overline{\bsb}^{\bsl}$ and $\bbB_{\bsk} = (|\bsk|+1)!\,\overline{\bsb}^\bsk$ and then move to the last line using the identity from \cite{KN16}
	\begin{align*}
	\sum_{\bsl\leq \bsm} {\bsm \choose \bsl} |\bsl|! (|\bsm-\bsl|+1)! = \frac{(|\bsm|+2)!}{2},
	\end{align*}
	which gives the required result. The constant $C_2$ is independent of $h$ and $\bsy$.
\end{proofof}

\begin{proofof}{Lemma \ref{lem:L2_FE_Err_der}}
We use an Aubin-Nitsche duality argument. For some linear functional $\calG~\in~L^2(D)$,  define $v^g\in V$ to be the solution to the dual problem,
	\begin{align*}
	\calA(\bsy;w,v^g(\cdot,\bsy)) = \calG(w)  \quad \text{for all } w\in V,
	\end{align*}
which, since $\calA$ is symmetric, is equivalent to the parametric variational problem \eqref{eq:variational_pde} with $f$ replaced by $g$, the representer of $\calG$. Thus, $v^g(\cdot,\bsy)\in V$ inherits the regularity of the solution to \eqref{eq:variational_pde} and the FE approximation $v_g^h(\cdot,\bsy)\in V_h$ also satisfies \eqref{eq:FE_error_d_bnd}.

Letting $w=u-u_h$ (and suppressing the dependence on $\bsy$), it follows from Galerkin orthogonality\eqref{eq:gal_orth} that $\calA(\bsy;u-u_h,v^g_h)=0$, which leads to
	\begin{align*}
	\calG(u-u_h) = \calA(\bsy;u-u_h,v_g) = \calA(\bsy;u-u_h,v^g-v^g_h).
	\end{align*}
Differentiating this with respect to $\bsy$ gives
	\begin{align*}
	\calG\big(\partial^{\bsnu}(u-u_h)\big) &= \int_D \partial^{\bsnu}\big(\Psi\,\nabla(u-u_h)\cdot \nabla(v^g-v^g_h)\big)\,\rd\bsx.
	\end{align*}
Applying the Leibniz product rule, the integrand on the right becomes
	\begin{align*}
	&\sum_{\bsm\leq\bsnu} {\bsnu\choose \bsm}\big(\partial^{\bsm}\Psi\big) \,
	\partial^{\bsnu-\bsm}\big(\nabla(u-u_h)\cdot \nabla(v^g-v^g_h)\big)\notag\\
	&=\Psi\,\partial^{\bsnu}\big(\nabla(u-u_h)\cdot \nabla(v^g-v^g_h)\big) \notag \\
	&\hspace{1cm} +
	\sum_{\bszero \neq\bsm\leq\bsnu} {\bsnu\choose \bsm}\big(\partial^{\bsm}\Psi\big) \,
	\partial^{\bsnu-\bsm}\big(\nabla(u-u_h)\cdot \nabla(v^g-v^g_h)\big)\notag\\
	&= \Psi\,\partial^{\bsnu}\big(\nabla(u-u_h)\cdot \nabla(v^g-v^g_h)\big)\notag\\
	&\hspace{1cm} +\sum_{j\geq1}\sum_{k=1}^{\nu_j}{\nu_j \choose k}(2\pi)^k\sin\bigg(2\pi y_j +\frac{k\pi}{2}\bigg)\,\psi_j \,\partial^{\bsnu-k\bse_j}\big(\nabla(u-u_h)\cdot \nabla(v^g-v^g_h)\big)\notag\\
	&= \Psi\!\!\sum_{\bsm\leq\bsnu} \!\!{\bsnu\choose \bsm}\nabla\partial^{\bsm}(u-u_h)\cdot \nabla\partial^{\bsnu-\bsm}(v^g-v^g_h)+\sum_{j\geq1}\sum_{k=1}^{\nu_j}\!{\nu_j \choose k}(2\pi)^k\sin\!\bigg(\!2\pi y_j +\frac{k\pi}{2}\!\bigg)\psi_j \notag\\
	&\hspace{2cm}\times\bigg(\sum_{\bsl\leq\bsnu-k\bse_j} {\bsnu-k\bse_j\choose \bsl} \nabla \partial^{\bsl}(u-u_h)
	\cdot \nabla\partial^{\bsnu-k\bse_j-\bsl}(v^g-v^g_h)\bigg),
	\end{align*}
where we have substituted in the bound \eqref{eq:d_a(x,y)} and applied the Leibniz product rule  to $\partial^{\bsnu-k\bse_j}\big(\nabla(u-u_h)\cdot \nabla(v^g-v^g_h)\big)$.

Now, taking the absolute value and using the Cauchy-Schwarz inequality gives
	\begin{align}
	\label{eq:FE_d_err_sums}
	&\big|\calG\big(\partial^\bsnu( u-u_h)\big)\big| \leq \Psi_{\max} \,\sum_{\bsm\leq\bsnu} {\bsnu\choose \bsm} \| \partial^{\bsm}(u-u_h)\|_{V}\, \|\partial^{\bsnu-\bsm}(v^g-v^g_h)\|_{V}\\
	&+ \Psi_{\min}\sum_{j\geq1}\sum_{k=1}^{\nu_j} {\nu_j \choose k} (2\pi)^k b_j
	\!\!\!\sum_{\bsl\leq\bsnu-k\bse_j}\!\!\! {\bsnu-k\bse_j\choose \bsl}\| \partial^{\bsl}(u-u_h)\|_{V} \|\partial^{\bsnu-k\bse_j-\bsl}(v^g-v^g_h)\|_{V}, 
	\notag
	\end{align}
where we have also used the definition of $b_j$ in \eqref{eq:bb_bar}.  
	
The terms in the first sum in \eqref{eq:FE_d_err_sums} can be bounded by \eqref{eq:FE_error_d_bnd} from  Lemma \ref{lem:d_u-uh} to give
	\begin{align}
	\sum_{\bsm\leq\bsnu} & {\bsnu\choose \bsm} \| \partial^{\bsm}(u-u_h)\|_{V}\, \|\partial^{\bsnu-\bsm}(v^g-v^g_h)\|_{V}\notag\\
	\leq \,& C_2\,h^2\,\|f\|_{L^2(D)}\|g\|_{L^2(D)}(2\pi)^{|\bsnu|}\,\sum_{\bsm\leq\bsnu} {\bsnu\choose \bsm}\notag\\
	&\times \bigg(\sum_{\bsl \leq \bsm}(|\bsl|+2)!\,\overline{\bsb}^\bsl\prod_{i\geq 1}S(m_i,\ell_i)\bigg)\,  
	\bigg(\sum_{\bsk \leq \bsnu-\bsm}(|\bsk|+2)!\,\overline{\bsb}^\bsk\prod_{i\geq 1}S(\nu_i-m_i,k_i)\bigg)\notag\\
	=\, &\frac{C_2}{30}\,h^2\,\|f\|_{L^2(D)}\|g\|_{L^2(D)}(2\pi)^{|\bsnu|}\,\sum_{\bsm\leq\bsnu} \overline{\bsb}^\bsm (|\bsm|+5)!\prod_{i\geq 1}S(\nu_i,m_i),\label{eq:term1}
	\end{align}	
where $C_2$ denotes the constant factor from \eqref{eq:FE_error_d_bnd} and we obtain the last equality using \eqref{eq:stirling_identity} with $\bbA_{\bsl} = (|\bsl|+2)!\,\overline{\bsb}^\bsl$ and $\bbB_{\bsk}=(|\bsk|+2)!\,\overline{\bsb}^\bsk$,  along with the identity from \cite{KN16}
	\begin{align}\label{eq:combinatorial_id}
	\sum_{\bsl\leq \bsm} {\bsm \choose \bsl} (|\bsl|+2)!\, (|\bsm-\bsl|+2)!=\frac{(|\bsm|+5)!}{30}.
	\end{align}

	Similarly, for the summation over the index $\bsl$ in \eqref{eq:FE_d_err_sums}, we again use \eqref{eq:FE_error_d_bnd} from Lemma~\ref{lem:d_u-uh} along with \eqref{eq:stirling_identity} and  \eqref{eq:combinatorial_id} to obtain
	\begin{align*}
	&\sum_{\bsl\leq\bsnu-k\bse_j} {\bsnu-k\bse_j\choose \bsl}\| \partial^{\bsl}(u-u_h)\|_{V}\, \|\partial^{\bsnu-k\bse_j-\bsl}(v^g-v^g_h)\|_{V}\\
&	 \leq C_2\,h^2\,\|f\|_{L^2(D)}\|g\|_{L^2(D)}(2\pi)^{|\bsnu|-k}\sum_{\bsl\leq\bsnu-k\bse_j}\overline{\bsb}^\bsl\,(|\bsl|+5)!\,S(\nu_j-k,\ell_j)\prod_{\satop{i\geq 1}{i\neq j}}S(\nu_i,\ell_i).
	\end{align*}

Substituting this back into the sum indexed by $j$ in \eqref{eq:FE_d_err_sums}, we have 
	\begin{align}
	&\sum_{j\geq1}\sum_{k=1}^{\nu_j} {\nu_j \choose k} (2\pi)^k b_j
	 \sum_{\bsl\leq\bsnu-k\bse_j} {\bsnu-k\bse_j\choose \bsl}\| \partial^{\bsl}(u-u_h)\|_{V}\, \|\partial^{\bsnu-k\bse_j-\bsl}(v^g-v^g_h)\|_{V}\notag\\
&	 \leq C_2\,h^2\,\|f\|_{L^2(D)}\|g\|_{L^2(D)}(2\pi)^{|\bsnu|}\notag\\
&\qquad\qquad\qquad
	 \times\sum_{j\geq1}\underbrace{\sum_{k=1}^{\nu_j} {\nu_j \choose k} b_j\!\!\!\!
	 \sum_{\bsl\leq\bsnu-k\bse_j}\!\!\!\!\overline{\bsb}^\bsl\,(|\bsl|+5)!\,S(\nu_j-k,\ell_j)\prod_{\satop{i\geq 1}{i\neq j}}S(\nu_i,\ell_i)}_{\eqcolon\, \Theta_j}.\label{eq:term2}
	\end{align}

To bound $\Theta_j$, we use the same technique as in Lemma \ref{lem:d_recursive_id}.
We separate out component $\ell_j$ from the innermost sum over $\bsl$, bound $b_j$
by $\overline{b}_j$ then swap the order of the sums over $k$ and $\ell_j$
 so that \eqref{eq:stirling_id} can be used to evaluate the sum over $k$. 
This gives
	\begin{align*}
	\Theta_j &= \sum_{k=1}^{\nu_j} {\nu_j \choose k} b_j
	 \sum_{\bsl'\leq\bsnu'}\sum_{\ell_j=0}^{\nu_j-k}\overline{\bsb'}^{\bsl'}\,\bigg(\prod_{\satop{i\geq 1}{i\neq j}}S(\nu_i,\ell_i)\bigg)\,\overline b_j^{\ell_j}\,(|\bsl'|+\ell_j+5)!\,S(\nu_j-k,\ell_j)\\ 	
	 &\leq \sum_{\bsl'\leq\bsnu'}\overline{\bsb'}^{\bsl'}\,\bigg(\prod_{\satop{i\geq 1}{i\neq j}}S(\nu_i,\ell_i)\bigg)
	 \sum_{\ell_j=0}^{\nu_j-1} \sum_{k=1}^{\nu_j-\ell_j}{\nu_j \choose k}\,\overline{b}_j^{\ell_j+1}\,(|\bsl'|+\ell_j+5)!\,S(\nu_j-k,\ell_j) \\
	 &= \sum_{\bsl'\leq\bsnu'}\overline{\bsb'}^{\bsl'}\,\bigg(\prod_{\satop{i\geq 1}{i\neq j}}S(\nu_i,\ell_i)\bigg)
	 \sum_{\ell_j=0}^{\nu_j-1}\overline{b}_j^{\ell_j+1}\,(|\bsl'|+\ell_j+5)!\,(\ell_j+1)\,S(\nu_j,\ell_j+1) \\
	 &= \sum_{\bsl'\leq\bsnu'}\overline{\bsb'}^{\bsl'}\,\bigg(\prod_{\satop{i\geq 1}{i\neq j}}S(\nu_i,\ell_i)\bigg)
	 \sum_{\ell_j=1}^{\nu_j}\overline{b}_j^{\ell_j}\,(|\bsl'|+\ell_j+4)!\,\ell_j\,S(\nu_j,\ell_j),
	\end{align*}	 
We can add the terms $\ell_j=0$ to the sum due to the presence of the factor $\ell_j$ and thus
	\begin{align*}
	\Theta_j \leq 
	\sum_{\bsl\leq\bsnu}\overline{\bsb}^{\bsl}(|\bsl|+4)!\,\ell_j\,\prod_{i\geq 1}S(\nu_i,\ell_i),
	\end{align*}
and using $\sum_{j\geq 1}\ell_j = |\bsl| \leq |\bsl|+5$ we have,
	\begin{align}
	&\sum_{j\geq1}\Theta_j 
	\leq 
	\sum_{\bsl\leq\bsnu}\overline{\bsb}^{\bsl}(|\bsl|+5)!\,\prod_{i\geq 1}S(\nu_i,\ell_i).\label{eq:term2_simplified}
	\end{align}

Combining  \eqref{eq:term2_simplified}, \eqref{eq:term2}, \eqref{eq:term1} and \eqref{eq:FE_d_err_sums}, we have
	\begin{align*}
	&|\calG(\partial^{\bsnu}(u-u_h))| \notag \\
	&\leq \Psi_{\max} \frac{C_2}{30}\,h^2\,\|f\|_{L^2(D)}\|g\|_{L^2(D)}(2\pi)^{|\bsnu|}\,\sum_{\bsm\leq\bsnu} \overline{\bsb}^\bsm (|\bsm|+5)!\prod_{i\geq 1}S(\nu_i,m_i) \notag\\
	&\hspace{1 cm}+\Psi_{\min}\,C_2\,h^2\,\|f\|_{L^2(D)}\|g\|_{L^2(D)}(2\pi)^{|\bsnu|}
	 \sum_{\bsl\leq\bsnu}\overline{\bsb}^{\bsl}(|\bsl|+5)!\,\prod_{i\geq 1}S(\nu_i,\ell_i) \notag\\
	&=C_2\Big(\frac{\Psi_{\max}}{30} + \Psi_{\min}\Big)\,h^2\,\|f\|_{L^2(D)}\|g\|_{L^2(D)}(2\pi)^{|\bsnu|} 
	\sum_{\bsm\leq\bsnu} \overline{\bsb}^\bsm (|\bsm|+5)!\prod_{i\geq 1}S(\nu_i,m_i).
	\end{align*}

Finally,  we let $\calG$ be the functional with representer $g =\partial^{\bsnu}(u-u_h)(\cdot,\bsy)$, i.e., $\calG(v)= |\langle \partial^{\bsnu}(u-u_h),v\rangle|$ for  $v\in V$, which gives
	\begin{align*}
	 &\|\partial^{\bsnu}(u-u_h)\|^2_{L^2(D)}\\
	 &\qquad
	  \lesssim\,h^2\,\|f\|_{L^2(D)}\|\partial^{\bsnu}(u-u_h)\|_{L^2(D)} (2\pi)^{|\bsnu|} \sum_{\bsm\leq\bsnu} \overline{\bsb}^\bsm (|\bsm|+5)!\prod_{i\geq 1}S(\nu_i,m_i),
	 \end{align*}
where the implied constant is independent of $h$ and $\bsy$.
Dividing through by $\|\partial^{\bsnu}(u-u_h)\|_{L^2(D)}$
yields the required result \eqref{eq:FE_error_d_bnd}.
\end{proofof}

\end{appendices}

\end{document}